\documentclass[reqno]{amsart}
\usepackage{amssymb}
\usepackage{graphicx}

\usepackage[usenames, dvipsnames]{color}
\usepackage{verbatim}
\usepackage{mathrsfs}
\usepackage{bm}
\usepackage{cite}
\usepackage{esint}

\numberwithin{equation}{section}

\newtheorem{theorem}{Theorem}[section]
\newtheorem{corollary}[theorem]{Corollary}
\newtheorem{lemma}[theorem]{Lemma}

\theoremstyle{definition}
\newtheorem{remark}[theorem]{Remark}

\theoremstyle{definition}

\theoremstyle{definition}

\makeatletter
\def\dashint{\operatorname%
{\,\,\text{\bf-}\kern-.98em\DOTSI\intop\ilimits@\!\!}}
\makeatother

\def\\det{\text{det}}

\def\Xint#1{\mathchoice
 {\XXint\displaystyle\textstyle{#1}}%
 {\XXint\textstyle\scriptstyle{#1}}%
 {\XXint\scriptstyle\scriptscriptstyle{#1}}%
 {\XXint\scriptscriptstyle\scriptscriptstyle{#1}}%
 \!\int}
\def\XXint#1#2#3{{\setbox0=\hbox{$#1{#2#3}{\int}$}
  \vcenter{\hbox{$#2#3$}}\kern-.5\wd0}}

\def\dashint{\Xint-}

\def\.5{\frac{1}{2}}

\newcommand{\RN}[1]{%
  \textup{\uppercase\expandafter{\romannumeral#1}}%
}

\newcommand{\pll}{\kern 0.56em/\kern -0.8em /\kern 0.56em}
\renewcommand{\epsilon}{\varepsilon}

\newcounter{marnote}

\begin{document}
\title[Stress concentration for nonlinear insulated conductivity problem]{Stress concentration for nonlinear insulated conductivity problem with adjacent inclusions}

\author[Q.L. Chen]{Qionglei Chen}
\address[Q.L. Chen] {Institute of Applied Physics and Computational Mathematics, Beijing, 100088, China.}
\email{chen\_qionglei@iapcm.ac.cn}

\author[Z.W. Zhao]{Zhiwen Zhao}
\address[Z.W. Zhao]{Beijing Computational Science Research Center, Beijing 100193, China.}
\email{zwzhao365@163.com}


\date{\today} 



\begin{abstract}
A high-contrast two-phase nonlinear composite material with adjacent inclusions of $m$-convex shapes is considered for $m>2$. The mathematical formulation consists of the insulated conductivity problem with $p$-Laplace operator in $\mathbb{R}^{d}$ for $p>1$ and $d\geq2$. The stress, which is the gradient of the solution, always blows up with respect to the distance $\varepsilon$ between two inclusions as $\varepsilon$ goes to zero. We first establish the pointwise upper bound on the gradient possessing the singularity of order $\varepsilon^{-\beta}$ with $\beta=(1-\alpha)/m$ for some $\alpha\geq0$, where $\alpha=0$ if $d=2$ and $\alpha>0$ if $d\geq3$. In particular, we give a quantitative description for the range of horizontal length of the narrow channel in the process of establishing the gradient estimates, which provides a clear understanding for the applied techniques and methods. For $d\geq2$, we further construct a supersolution to sharpen the upper bound with any $\beta>(d+m-2)/(m(p-1))$ when $p>d+m-1$. Finally, a subsolution is also constructed to show the almost optimality of the blow-up rate $\varepsilon^{-1/\max\{p-1,m\}}$ in the presence of curvilinear squares. This fact reveals a novel dichotomy phenomena that the singularity of the gradient is uniquely determined by one of the convexity parameter $m$ and the nonlinear exponent $p$ except for the critical case of $p=m+1$ in two dimensions.
\end{abstract}

\maketitle



\section{Introduction}
\subsection{Background and motivation}
This work is devoted to studying the stress concentration phenomena appearing in high-contrast composite materials whose constituents comprise of densely packed fibers and the matrix with significantly different properties. Our initial interest was motivated by the problem of material failure initiation, which has been paid great attention due to its extensive engineering applications. High-contrast composites with close-to-touching fibers can be described by elliptic equations and systems with discontinuous coefficients. The problem of quantitatively characterizing the behavior of the stress, which is the gradient of the solution, was first proposed in the famous work \cite{BASL1999} in relation to numerical analysis of damage initiation and growth in fiber-reinforced composites. From the view of mathematics, one may represent the background matrix by a bounded domain $D\subseteq\mathbb{R}^{d}\,(d\geq2)$ with $C^{1.1}$ boundary, whose interior contains two $C^{1,1}$ inclusions $D_{1}$ and $D_{2}$. These two inclusions are $\varepsilon$ apart and far from the external boundary $\partial D$, where $\varepsilon$ is a sufficiently small positive constant. Denote $\Omega:=D\setminus\overline{D_{1}\cup D_{2}}$. A scalar mathematical model of high-contrast two-phase composite materials is formulated as follows:
\begin{align}\label{pro006}
\begin{cases}
\mathrm{div}(a_{k}(x)|\nabla u_{k}|^{p-2}\nabla u_{k})=0,&\mathrm{in}\;D,\\
u_{k}=\varphi,&\mathrm{on}\;\partial D,
\end{cases}
\end{align}
where $u_{k}\in W^{1,p}(D)$, $p>1$, $\varphi\in C^{1,1}(\partial D)$, and
\begin{align*}
a_{k}(x)=
\begin{cases}
k\in(0,\infty),&\mathrm{in}\;D_{1}\cup D_{2},\\
1,&\mathrm{in}\;\Omega.
\end{cases}
\end{align*}
This model is also called the conductivity model and $k$ is called the conductivity. When $p=2$, it describes a linear background medium and in this case the gradient of the solution has been proved to be bounded independently of the distance $\varepsilon$, see \cite{BV2000,LV2000}. Li and Nirenberg \cite{LN2003} further extended the results to general divergence form elliptic systems with piecewise H\"{o}lder continuous coefficients, which completely demonstrates the numerical observation in \cite{BASL1999}. For more studies on the finite coefficients, see \cite{DZ2016,DL2019,CEG2014,KL2019} and the references therein. However, when the background matrix is nonlinear, to the best of our knowledge, there has not been any corresponding result for problem \eqref{pro006} with $p\neq2$.

When the conductivity $k$ degenerates to be zero or infinity, the situation becomes very different. The key feature of the conductivity problem with degenerate conductivities is that the stress always exhibits singular behavior with respect to the distance between two inclusions. Observe that applying the proofs in the appendixes of \cite{BLY2009,BLY2010} with minor modification, we obtain that by sending $k\rightarrow0$ and $k\rightarrow\infty$, the solution sequence $\{u_{k}\}$ of the original problem \eqref{pro006} converges weakly in $W^{1,p}$ to the solution $u$ of the following insulated and perfect conductivity problems
\begin{align}\label{con002}
\begin{cases}
-\mathrm{div}(|\nabla u|^{p-2}\nabla u)=0,&\hbox{in}\;\Omega,\\
\frac{\partial u}{\partial\nu}=0,&\mathrm{on}\;\partial D_{i},\,i=1,2,\\
u=\varphi, &\mathrm{on}\;\partial D,
\end{cases}
\end{align}
and
\begin{equation}\label{YHU0.002}
\begin{cases}
-\mathrm{div}(|\nabla u|^{p-2}\nabla u)=0,& \mathrm{in}\;\Omega,\\
u=C_{i},&\mathrm{on}\;\partial D_{i},\;i=1,2,\\
\int_{\partial D_{i}}\frac{\partial u}{\partial\nu}=0,&\mathrm{for}\;i=1,2,\\
u=\varphi,&\mathrm{on}\;\partial D,
\end{cases}
\end{equation}
respectively, where $C_{i}$, $i=1,2$ are the free constants determined by the third line of \eqref{YHU0.002} and $\nu$ is the unit outer normal to the domain. From the degenerate elliptic regularity theory in \cite{L1988}, we know that the regularity of weak solution $u$ to problem \eqref{con002} or \eqref{YHU0.002} can be improved to be of $C^{1,\alpha}(\overline{\Omega})$ for some $0<\alpha<1$. With regard to the linear perfect conductivity problem \eqref{YHU0.002} with $p=2$, there has been a long series of literature involving different techniques and methods to study the singular behavior of the gradient. The principal result is that for two strictly convex inclusions, the gradient blow-up rate has been revealed to be
\begin{align*}
\rho_{d}(\varepsilon)=&
\begin{cases}
\varepsilon^{-1/2},&\mathrm{if}\;d=2,\\
|\varepsilon\ln\varepsilon|^{-1},&\mathrm{if}\;d=3,\\
\varepsilon^{-1},&\mathrm{if}\;d\geq4,
\end{cases}
\end{align*}
see \cite{AKLLL2007,BC1984,BLY2009,AKL2005,Y2007,Y2009,K1993} for $d=2$, \cite{BLY2009,LY2009,BLY2010,L2012} for $d=3$ , and \cite{BLY2009} for $d\geq4$, respectively. It is worth emphasizing that the gradient blow-up rate for the perfect conductivity problem only depends on the dimension, but not on the curvatures of interfacial boundaries of inclusions. This is greatly different from the insulated conductivity problem. Specifically, Dong, Li and Yang \cite{DLY2022} found that for $p=2$ and $d\geq3$, the optimal blow-up rate is in close touch with the first non-zero eigenvalue of a divergence form elliptic operator on $\mathbb{S}^{d-2}$, which is determined by the relative principal curvatures of the insulators. When the relative principal curvatures of these two inclusions keep the same (see \cite{DLY2021}), the blow-up rate is more explicit and of order $\varepsilon^{-1/2+\beta}$ with $\beta=[-(d-1)+\sqrt{(d-1)^{2}+4(d-2)}]/4$. For a special case of two spherical insulators in three dimensions, Yun revealed the blow-up rate in the earlier work \cite{Y2016} by establishing the optimal gradient estimates in the shortest segment between two inclusions. In fact, with regard to the study on the stress blow-up rate in dimensions greater than two, it has experienced a long time suspicion about whether the blow-up rate $\varepsilon^{-1/2}$ captured in the previous work \cite{BLY2010} is optimal for $d\geq3$ until Li and Yang \cite{LY202102} achieved a qualitative leap from $\varepsilon^{-1/2}$ to $\varepsilon^{-(1-\gamma)/2}$ for some $\gamma>0$. This improvement leads to the subsequent work to pursue a sharp blow-up rate. For two unit balls in dimensions greater than three, Weinkove \cite{W2021} developed a clear and concise procedure to obtain an explicit blow-up exponent $\gamma$ by constructing an appropriate auxiliary function and using a direct maximum principle argument, which further sharpens the Li-Yang estimate for $d\geq4$. Regarding the case when one insulator is in close proximity to the matrix boundary, see \cite{M2023}. As for the case of $d=2$, the singularity of the gradient in the linear case has been made clear about two decades ago in \cite{AKLLL2007,AKL2005}.

However, the techniques and methods developed in these above-mentioned work are mainly used to deal with the linear problems. There makes no progress on the study of nonlinear case until Gorb and Novikov \cite{GN2012} constructed the explicit upper and lower barriers of the solution to the nonlinear perfect conductivity problem \eqref{YHU0.002}. These barrier functions precisely capture the singularities of the gradient. The subsequent work \cite{CS2019,CS201902} completed by Ciraolo and Sciammetta further extended the results to the anisotropic perfect conductivity problem modeled by the Finsler $p$-Laplacian. Their results showed that for $d\geq2$,
\begin{align*}
|\nabla u|\sim
\begin{cases}
\varepsilon^{-\frac{d-1}{2(p-1)}},&\text{if }p>\frac{d+1}{2},\\
\varepsilon^{-1}|\ln\varepsilon|^{-\frac{1}{p-1}},&\text{if }p=\frac{d+1}{2},\\
\varepsilon^{-1},&\text{if }p<\frac{d+1}{2}.
\end{cases}
\end{align*}

Recently, Dong, Yang and Zhu \cite{DYZ2023} made a breakthrough in the investigation on stress concentration for the nonlinear insulated conductivity problem \eqref{con002}. In the presence of two adjacent strictly convex insulators, they first used mean oscillation estimates to establish the pointwise gradient estimate with the upper bound depending on the height of the narrow region and the local oscillation of the solution. Subsequently, the upper bound was further improved at the sacrifice of the oscillation in a little larger domain by establishing the Krylov-Safonov Harnack inequality in the narrow region. However, the blow-up exponent in the improved gradient estimate is still not explicit. So they continued to construct a supersolution to capture an explicit blow-up rate of the gradient in all dimensions and meanwhile verified the almost optimality of the blow-up rate in two dimensions by providing an explicit subsolution. For the purpose of achieving these above results, a lot of techniques and tools, such as the flattening map, the rescale argument, the extension theory, the comparison estimates, the iterate scheme and the Harnack inequality, are involved. We here would like to point out that the key to understanding their proofs lies in making clear the required range of the value for horizontal length of the thin gap in the process of applying these techniques and tools. In their proofs, the horizontal length of the narrow region is only assumed to be small. So one of our major objectives in this paper is to give a quantitative description for the range of the length of the narrow channel, which greatly contributes to the readers' understanding on the employed idea and techniques. Moreover, we extend their results to the general $m$-convex inclusions (see the definition in Subsection \ref{SUBSEC02}), especially including curvilinear squares with rounded-off angles which are frequently used to design the optimal shape of insulated composites. In fact, we see from \cite{LY2021,Z2023,Z2022} that the blow-up rate exponent will decrease to zero as the convexity exponent $m$ increases to infinity. From the view of geometry, the increase of the convexity exponent $m$ implies that the interfacial boundaries of two adjacent insulators in the narrow region become more and more flat. Besides the local $m$-convex structures, the whole concentric core-shell structure has been revealed to be another type of optimal insulation shape of composite materials in \cite{Z202301}. In addition, the two-dimensional results obtained in this paper reveal a novel blow-up phenomena that contrary to the case of $p<m+1$, the singular behavior of the gradient in the case of $p>m+1$ is completely determined by the nonlinear degree exponent $p$, but independent of the convexity exponent $m$ which characterizes the geometrical feature of the thin gap. Then the case of $p=m+1$ can be called the critical case according to the dichotomy blow-up phenomena.

\subsection{Assumptions and main results}\label{SUBSEC02}
In order to properly state the major results, we first introduce some notations and parameterize the domain. After appropriate shifting and rotation of the coordinate, we suppose that $D_{1}$ and $D_{2}$ are translations of the following two touching inclusions
\begin{align*}
D_{1}:=D_{1}^{\ast}+(0',\varepsilon/2),\quad\mathrm{and}\; D_{2}:=D_{2}^{\ast}+(0',-\varepsilon/2),
\end{align*}
where $D_{i}^{\ast}$, $i=1,2$ satisfy
\begin{align*}
D_{i}^{\ast}\subset\{(x',x_{d})\in\mathbb{R}^{d}\,|\,(-1)^{i+1}x_{d}>0\},\quad\mathrm{and}\;\partial D_{1}^{\ast}\cap\partial D_{2}^{\ast}=\{0\}\subset\mathbb{R}^{d},\quad i=1,2.
\end{align*}
Here and below, we denote $(d-1)$-dimensional variables and domains by adding superscript prime (for instance, $x'$ and $B'$).

For some constant $R_{0}>0$ independent of $\varepsilon$, let the portions of $\partial D_{1}$ and $\partial D_{2}$ near the origin be, respectively, the graphs of two $C^{m}$ functions $\varepsilon/2+h_{1}(x')$ and $-\varepsilon/2+h_{2}(x')$, and $h_{i}$, $i=1,2$ satisfy that for $m\geq2$,
\begin{enumerate}
{\it\item[(\bf{H1})]
$\kappa_{1}|x'|^{m}\leq h_{1}(x')-h_{2}(x')\leq\kappa_{2}|x'|^{m},\;\mathrm{if}\;x'\in B_{2R_{0}}',$
\item[(\bf{H2})]
$|\nabla_{x'}h_{i}(x')|\leq \kappa_{3}|x'|^{m-1},\;\mathrm{if}\;x'\in B_{2R_{0}}',$
\item[(\bf{H3})]
$\|h_{1}\|_{C^{m}(B'_{2R_{0}})}+\|h_{2}\|_{C^{m}(B'_{2R_{0}})}\leq \kappa_{4},$}
\end{enumerate}
where $\kappa_{i}$, $i=1,2,3,4$ are four positive $\varepsilon$-independent constants, $C^{m}:=C^{[m],m-[m]}$ with $[m]$ representing the integer part of $m$. Remark that the parameter $m$ is introduced to describe the convexity feature of interfacial boundaries of the insulators. Especially when $m>2$, the curvatures of the interfacial boundaries degenerate to be zero at the points $(0',\pm\varepsilon/2)$ of the shortest distance between two inclusions which are also called the planar points.

We denote by $B_{r}(x)$ the ball of radius $r$ centered at $x$ and set
\begin{align*}
B_{r}=B_{r}(0),\quad\Omega_{r}(x)=\Omega\cap B_{r}(x).
\end{align*}
For $z'\in B'_{R_{0}},\,0<t\leq2R_{0}$, define a narrow channel by
\begin{align*}
\Omega_{z',t}:=&\{x\in \mathbb{R}^{d}\,|~|x'-z'|<t,\,-\varepsilon/2+h_{2}(x')<x_{d}<\varepsilon/2+h_{1}(x')\}.
\end{align*}
For simplicity, write $\Omega_{t}:=\Omega_{0',t}$ if $z'=0'$. Its upper and lower boundaries are, respectively, represented by
\begin{align*}
\Gamma^{+}_{t}:=\{x\in\mathbb{R}^{d}\,|\,x_{d}=\varepsilon/2+h_{1}(x'),\;|x'|<t\},
\end{align*}
and
\begin{align*}
\Gamma^{-}_{t}:=\{x\in\mathbb{R}^{d}\,|\,x_{d}=-\varepsilon/2+h_{2}(x'),\;|x'|<t\}.
\end{align*}
For $x\in B'_{2R_{0}},$ define
\begin{align}\label{delta}
\delta:=\delta(x')=\varepsilon+h_{1}(x')-h_{2}(x').
\end{align}
From the classical $C^{1,\alpha}$ estimates and the maximum principle (see, for example, \cite{GT1998,L1988}), we know
\begin{align*}
\|u\|_{L^{\infty}(\Omega)}+\|u\|_{C^{1,\alpha}(\Omega\setminus\Omega_{R_{0}})}\leq C\|\varphi\|_{C^{1,1}(\partial D)}.
\end{align*}
This means that it only needs to estimate $|\nabla u|$ in the narrow channel $\Omega_{R_{0}}$. Then we reduce the original problem \eqref{con002} to the problem in the neck as follows:
\begin{align}\label{problem006}
\begin{cases}
-\mathrm{div}(|\nabla u|^{p-2}\nabla u)=0,&\hbox{in}\;\Omega_{2R_{0}},\\
\frac{\partial u}{\partial\nu}=0,&\mathrm{on}\;\Gamma^{\pm}_{2R_{0}}.
\end{cases}
\end{align}

Introduce the following five constants:
\begin{align}
c_{0}=&\kappa_{1}^{-1/m}\min\{1,2^{-(m+1)}\kappa_{1}\kappa_{3}^{-1}\},\label{ZMQC960}\\
\tilde{c}_{0}=&\kappa_{1}^{-1/m}\min\{1,\kappa_{1}\kappa_{3}^{-1}(51840+4^{m+1})^{-1/2}\},\label{C681}\\
R_{0,1}=&2^{-\frac{m}{m-1}}(\kappa_{2}\kappa_{3})^{-\frac{1}{2m-2}}\min\left\{1,\frac{(\kappa_{1}^{1/m}c_{0})^{\frac{1}{2m-2}}}{2^{\frac{1}{m-1}}},\frac{(\kappa_{1}\kappa_{3}^{-1})^{\frac{1}{2m-2}}}{2^{\frac{2m-1}{2m-2}}}\right\},\label{C09}\\
R_{0,2}=&\min\{2^{-\frac{2m+1}{2m-1}}\kappa_{3}^{-\frac{1}{m-1}},(\sqrt{2}\kappa_{2}\kappa_{4}(d-1)^{2})^{-\frac{1}{m}}\},\label{C12}\\
R_{0,3}=&\left(\frac{12}{13}\tilde{c}_{0}\mu_{0}^{j_{0}}\right)^{\frac{1}{m-1}}\kappa_{2}^{-1/m},\label{WAD099}
\end{align}
where $j_{0}=[\frac{6d}{\min\{p,2\}}]+1$ and $\mu_{0}$ is given by \eqref{QM86} below. The first main result is listed as follows.
\begin{theorem}\label{thm001}
Suppose that $D_{1},\,D_{2}\subset D\subseteq\mathbb{R}^{d}\,(d\geq2)$ are defined as above, conditions $\mathrm{(}${\bf{H1}}$\mathrm{)}$--$\mathrm{(}${\bf{H3}}$\mathrm{)}$ hold. Let $u\in W^{1,p}(\Omega_{2R_{0}})$ be the solution of \eqref{problem006} with $p>1$. If $0<R_{0}\leq\min\limits_{1\leq i\leq3}R_{0,i}$, then for a sufficiently small $\varepsilon>0$ and $x\in\Omega_{R_{0}}$,
\begin{align}\label{U002}
|\nabla u(x)|\leq C(\varepsilon+|x'|^{m})^{-1/m}\mathop{osc}\limits_{\Omega_{x',\varrho}}u,\quad\varrho=\frac{\tilde{c}_{0}}{3}\delta^{1/m},
\end{align}
where $C=C(d,m,p,\kappa_{1},\kappa_{3},\kappa_{4})$, $\delta$ and $\tilde{c}_{0}$ are, respectively, defined by \eqref{delta} and \eqref{C681}.
\end{theorem}
\begin{remark}
Observe from \eqref{U002} that the upper bound comprises of the blow-up term $(\varepsilon+|x'|^{m})^{-1/m}$ and the oscillation term $\mathop{osc}\limits_{\Omega_{x',\varrho}}u$. The oscillation term is actually a small term due to the boundness of the solution $u$ and the smallness of oscillation domain $\Omega_{x',\varrho}$. Then in the following we aim to extract the remaining decay rate exponent from this oscillation term.
\end{remark}
\begin{remark}
We now give some explanations for the implications of these parameters introduced in \eqref{ZMQC960}--\eqref{WAD099}. First, the constants $c_{0}$ and $\tilde{c}_{0}$ are geometrical quantities which are used to characterize the height equivalence between $\delta(x')$ and other height lines in the thin gap $\Omega_{x',\varrho}$, see \eqref{QWN001} below. Second, the values of horizontal length parameters $R_{0,i}$, $i=1,2,3$ of the thin gap are gradually determined in the process of establishing the pointwise gradient estimate, which represent the valid scope of the techniques and methods employed in this paper. Furthermore, there are some distinct differences for their values. To be specific, $R_{0,i}$, $i=1,2$ are two geometrical quantities independent of the equation, while the dependence of $R_{0,3}$ is more comprehensive and its value depends on $d,m,p,\kappa_{1},\kappa_{3},\kappa_{4}$, that is, $R_{0,3}$ is determined by the domain and $p$-Laplace equation, which can be regarded as the most important parameter to ensure the validity in terms of applications of the techniques and methods.

\end{remark}

For any fixed $0<R_{0}\leq\min\limits_{1\leq i\leq3}R_{0,i}$ and $0<\beta<1$, we define the following constants:
\begin{align}
r_{0,1}=&\min\Big\{2^{1/m},\frac{2}{3}R_{0}\Big\},\label{Z016}\\
r_{0,2}=&\left(\frac{\min\{1,4^{-(m+1)}\kappa_{1}\}}{\kappa_{2}(\kappa_{3}+\kappa_{4})+\kappa_{3}(\kappa_{3}+1)}\right)^{\frac{1}{m-1}}
\begin{cases}
\frac{6^{-\frac{m+2}{m-1}}}{(d-1)^{\frac{2}{m-1}}},&\text{if }2\leq m<3,\vspace{0.1cm}\\
4^{-\frac{m}{m-1}},&\text{if }m\geq3,
\end{cases}\label{Z018}\\
r_{0,3}=&2^{-\frac{3(m+1)}{m-1}}(\kappa_{1}\kappa_{3}^{-1})^{\frac{1}{m-1}},\quad r_{0,4}=\frac{\min\{R_{0}^{\frac{1}{1-\beta}},2^{-\frac{1}{\beta}}\}}{\tilde{c}_{0}(2\kappa_{2})^{1/m}},\label{Z006}
\end{align}
where $\tilde{c}_{0}$ is given by \eqref{C681}. For $d\geq3$, the upper bound in \eqref{U002} can be further improved as follows.
\begin{theorem}\label{thm002}
Let $D_{1},\,D_{2}\subset D\subseteq\mathbb{R}^{d}\,(d\geq3)$ be defined as above and conditions $\mathrm{(}${\bf{H1}}$\mathrm{)}$--$\mathrm{(}${\bf{H3}}$\mathrm{)}$ hold. Assume that $u\in W^{1,p}(\Omega_{2R_{0}})$ is the solution of \eqref{problem006} with $p>1$. For any fixed $0<R_{0}\leq\min\limits_{1\leq i\leq3}R_{0,i}$ and $0<\beta<1$, if $0<r_{0}\leq\min\limits_{1\leq i\leq4}r_{0,i}$, then for a arbitrarily small $\varepsilon>0$ and $x\in\Omega_{r_{0}}$, there exist two positive constants $0<\gamma<1$ and $C>0$ depending only on $d,m,p,\kappa_{i}$, $i=1,2,3,4$ such that
\begin{align}\label{DEM9068}
|\nabla u(x)|\leq C(\varepsilon+|x'|^{m})^{\frac{\beta\gamma-1}{m}}\mathop{osc}_{\Omega_{x',\tilde{\varrho}^{1-\beta}}}u,\quad \tilde{\varrho}=\tilde{c}_{0}\delta^{1/m},
\end{align}
where $\delta$ and $\tilde{c}_{0}$ are defined by \eqref{delta} and \eqref{C681}, respectively.
\end{theorem}

As shown in Theorem \ref{thm002}, although the upper bound on the gradient has been improved, its blow-up exponent is still not clear. So in the next theorem we aim to obtain the explicit gradient blow-up rate. For that purpose, we prepare to consider a class of axisymmetric insulators, whose upper and lower boundaries of the narrow region are, respectively, formulated as follows: for $i=1,2$ and $|x'|<\tilde{r}_{0}$,
\begin{align}\label{U0913}
x_{d}=(-1)^{i+1}\left(\frac{\varepsilon}{2}+h(x')\right)=(-1)^{i+1}\left(\frac{\varepsilon}{2}+\lambda|x'|^{m}+O(|x'|^{m+\sigma_{0}})\right),
\end{align}
where $\tilde{r}_{0}$, $\lambda$ and $\sigma_{0}$ are three positive constants independent of $\varepsilon.$ Here $O(1)$ represents that $|O(1)|\leq C_{0}$ for some $\varepsilon$-independent constant $C_{0}>0$ which may depend on $d,\,m,\,\sigma_{0}$ and $\tilde{r}_{0}$. We further assume that the derivative of $h(x')$ satisfies
\begin{align}\label{U091302}
\partial_{x_{i}}h(x')=m\lambda x_{i}|x'|^{m-2}+O(|x'|^{m-1+\sigma_{0}}),\quad \mathrm{for}\;i=1,...,d-1.
\end{align}
The prototypes of these two insulators under such assumptions are two axisymmetric ellipsoids as follows:
\begin{align}\label{F019}
|x'|^{m}+|x_{d}-\varepsilon/2-\tilde{r}_{0}|^{m}=\tilde{r}_{0}^{m},\quad\mathrm{and}\;|x'|^{m}+|x_{d}+\varepsilon/2+\tilde{r}_{0}|^{m}=\tilde{r}_{0}^{m}.
\end{align}
Here $\tilde{r}_{0}$ is called the radius of ellipsoids. In particular, when $m>2$ and $d=2,$ the shapes of insulators resembling curvilinear squares with rounded-off angles. Applying Taylor expansion to \eqref{F019}, we reformulate the top and bottom boundaries of the thin gap as follows: for $|x'|<\tilde{r}_{0}$,
\begin{align*}
x_{d}=(-1)^{i+1}\Big(\frac{\varepsilon}{2}+\frac{1}{m\tilde{r}_{0}^{m-1}}|x'|^{m}\Big)+O(|x'|^{2m}),\quad\text{for $i=1,2.$ }
\end{align*}
When $p>d+m-1$, we improve the upper bounds in Theorems \ref{thm001} and \ref{thm002} as follows.
\begin{theorem}\label{thm003}
For $d\geq2$, $m>2$ and $p>d+m-1$, let $u\in W^{1,p}(\Omega_{\tilde{r}_{0}})$ be the solution of \eqref{problem006} with $\Gamma_{\tilde{r}_{0}}^{\pm}$ given by \eqref{U0913}. Then for any $0<\tau<\frac{1}{2}(p+1-d-m),$ if $\varepsilon>0$ is sufficiently small,
\begin{align}\label{ARQZ001}
|\nabla u(x)|\leq C(\varepsilon+|x'|^{m})^{-\frac{d+m-2+2\tau}{m(p-1)}},\quad\mathrm{for}\;x\in\Omega_{\tilde{r}_{0}},
\end{align}
where $C=C(d,m,p,\tau,\lambda,\sigma_{0},\tilde{r}_{0})$.
\end{theorem}
\begin{remark}
First, the blow-up rate $\varepsilon^{-\frac{d+m-2+2\tau}{m(p-1)}}$ is more explicit than that in \eqref{DEM9068}. Second, since $0<\tau<\frac{1}{2}(p+1-d-m)$, then the blow-up rate $\varepsilon^{-\frac{d+m-2+2\tau}{m(p-1)}}$ is less than $\varepsilon^{-1/m}$, which sharpens the result of $p>d+m-1$ in Theorem \ref{thm001}.
\end{remark}
\begin{remark}
From \eqref{ARQZ001}, we see that for any fixed $d\geq2$ and $m>2$, the gradient blow-up rate will decrease as the nonlinear exponent $p$ increases. This indicates that nonlinear composite materials are generally superior to linear composites from the perspective of reducing stress concentration. In summary, in order to improve the properties of composite materials with closely located fibers, we should choose curvilinear squares and cubes as the shapes of fibers and embed them in the nonlinear matrix. Meanwhile, we weaken the convexity of interfacial boundaries of fibers and enhance the nonlinear degree of the matrix as much as possible.

\end{remark}
\begin{remark}
Unlike Theorems \ref{thm001} and \ref{thm002}, the results corresponding to the case of $m=2$ cannot be directly obtained by sending $m\rightarrow2$ in the proofs of Theorems \ref{thm003} and \ref{thm005}. In fact, a close inspection of the proofs shows that the adapted constructions for explicit supersolution and subsolution in the case of $m>2$ make the computations quite different from that of $m=2$, see Lemmas \ref{W365} and \ref{Lem958} for the finer details.

\end{remark}

For $d=2,$ we also demonstrate the almost optimality of the gradient blow-up rates captured in Theorem \ref{thm001} for $1<p\leq m+1$ and Theorem \ref{thm003} for $p>m+1$ by establishing the lower bounds as follows.
\begin{theorem}\label{thm005}
For $d=2$ and $m>2$, Let $D:=B_{5}$ and $D_{i}$, $i=1,2$ be given by \eqref{F019} with $\tilde{r}_{0}=1$. For $p>1$, suppose that $u\in W^{1,p}(D)$ is the solution of \eqref{con002} with $\varphi=x_{1}$ on $\partial D$. Then for any $\tau\in(0,m-2)$ and a sufficiently small $\varepsilon>0$,

$(i)$ if $1<p\leq m+1$,
\begin{align}\label{QMED9068}
\|\nabla u\|_{L^{\infty}(\Omega_{(4m\varepsilon/(m-2-\tau))^{1/m}})}\geq\frac{1}{C}\varepsilon^{\frac{\tau-1}{m}};
\end{align}

$(ii)$ if $p>m+1,$
\begin{align}\label{DEQ0096}
\|\nabla u\|_{L^{\infty}(\Omega_{(4m\varepsilon/(m-2-\tau))^{1/m}})}\geq\frac{1}{C}\varepsilon^{\frac{2\tau-m}{m(p-1)}},
\end{align}
where $C=C(m,p,\tau)$.
\end{theorem}
\begin{remark}
First, a combination of \eqref{U002} and \eqref{QMED9068} yields that the blow-up rate $\varepsilon^{-1/m}$ is almost optimal when $1<p\leq m+1$. Second, combining \eqref{ARQZ001} and \eqref{DEQ0096}, we prove the almost sharpness of the blow-up rate $\varepsilon^{-1/(p-1)}$ when $p>m+1.$ These two facts show a bifurcation phenomena that when $p\neq m+1$, the gradient blow-up rate depends only on one of the nonlinear exponent $p$ and the convexity parameter $m$. Therefore, the case of $p=m+1$ can be regarded as the critical case.
\end{remark}

The rest of this paper is organized as follows. In Section \ref{SEC02}, we give the proof of Theorem \ref{thm001} by establishing mean oscillation estimates on the gradient. Section \ref{Sec903} is devoted to proving Theorem \ref{thm002} by considering a smooth approximating equation which reduces the original problem to the establishment of the Krylov-Safonov Harnack inequality for a second-order uniformly elliptic equation of non-divergence form. In Section \ref{Sec905}, we aim to construct explicit supersolution and subsolution to complete the proofs of Theorems \ref{thm003} and \ref{thm005}. In Section \ref{SEC050}, we point out a potentially effective way to explore the stress blow-up rate in high dimensions by studying the asymptotic behavior for the solutions to a class of $(d-1)$-dimensional weighted $p$-Laplace equations near the origin.

\section{Pointwise upper bound on the gradient}\label{SEC02}

From classical elliptic regularity theory and Caccioppoli inequality, we know that for any $B_{R}(\bar{x})\subset\Omega$ and $0<r<\frac{R}{2}$, there has pointwise gradient estimate
\begin{align*}
|\nabla u(\bar{x})|\leq& C(d,p)\left(\dashint_{B_{r}(\bar{x})}|\nabla u|^{p}dx\right)^{\frac{1}{p}}\notag\\
\leq&\frac{C(d,p)}{r}\left(\dashint_{B_{2r}(\bar{x})}\Big| u(x)-\dashint_{B_{2r}(\bar{x})}u\Big|^{p}dx\right)^{\frac{1}{p}}\leq\frac{C(d,p)}{r}\mathop{osc}_{B_{2r}(\bar{x})}u,
\end{align*}
where $\dashint_{B_{r}(\bar{x})}:=\frac{1}{|B_{r}(\bar{x})|}\int_{B_{r}(\bar{x})}$. However, the considered domain in this paper is a narrow region such that we cannot directly use this way to achieve the pointwise gradient estimate with explicit dependence on the height of thin gap. To overcome this difficulty, we first locally flatten the boundary of the narrow region and then carry out the even extension and periodic extension for the transformed equation. Therefore, we have enough space to establish mean oscillation decay estimates on the gradient, which is critical to the establishment of the desired poinwise gradient estimates.

For any given point $x_{0}=(x_{0}',x_{d})\in\Omega_{R_{0}}$, write
\begin{align}\label{DELTA02}
\delta_{0}:=\delta(x_{0}')=\varepsilon+h_{1}(x_{0}')-h_{2}(x_{0}').
\end{align}
Observe that there appears at most three cases in terms of the relations between ball $B_{r}(x_{0})$ and the top and bottom boundaries $\Gamma^{\pm}_{2R_{0}}$ of the narrow region $\Omega_{2R_{0}}$, as the radius $r$ varies from small to large. To be specific,
\begin{itemize}
{\it
\item[\em Case 1.] $B_{r}(x_{0})$ has no intersection with $\Gamma^{\pm}_{2R_{0}}$ for small $r$;
\item[\em Case 2.] $B_{r}(x_{0})$ only intersects one of $\Gamma^{\pm}_{2R_{0}}$ for intermediate $r$;
\item[\em Case 3.] $B_{r}(x_{0})$ intersects both $\Gamma^{+}_{2R_{0}}$ and $\Gamma^{-}_{2R_{0}}$ for large $r$.}
\end{itemize}
It is worth pointing out that the second case will not occur if the narrow region is symmetric and $x_{0}$ is the midpoint of the height $\delta_{0}$. The first case implies that $B_{r}(x_{0})$ is contained in the thin gap $\Omega_{2R_{0}}$, which leads to the following interior mean oscillation decay estimate. For $u\in W^{1,p}(\Omega_{2R_{0}})$ and $B_{r}(x_{0})\subset\Omega_{2R_{0}}$, denote
\begin{align*}
\phi(x_{0},r):=\left(\dashint_{B_{r}(x_{0})}|\nabla u-(\nabla u)_{B_{r}(x_{0})}|^{p}\right)^{\frac{1}{p}}.
\end{align*}
Then we have
\begin{lemma}[Lemma 2.1 and Corollary 2.2 in \cite{DYZ2023}]\label{Lem001}
Assume that $u\in W^{1,p}(\Omega_{2R_{0}})$ is a solution of problem \eqref{problem006}. Then there exist two constants $C_{0}=C_{0}(d,p)>1$ and $0<\alpha=\alpha(d,p)<1$ such that for any $B_{r}(x_{0})\subset\Omega_{2R_{0}}$, $\rho\in(0,r]$,
\begin{align}\label{DE001}
\phi(x_{0},\rho)\leq C_{0}\Big(\frac{\rho}{r}\Big)^{\alpha}\phi(x_{0},r).
\end{align}
This implies that for any $B_{r}(x_{0})\subset\Omega_{2R_{0}}$, $\mu\in(0,\mu_{1})$, $\mu_{1}=\min\{(2C_{0})^{-\frac{1}{\alpha}},1\}$, and $0\leq i_{0}<j$,
\begin{align*}
\sum^{j}_{k=i_{0}}\phi(x_{0},\mu^{k}r)\leq2\phi(x_{0},\mu^{i_{0}}r).
\end{align*}

\end{lemma}
\begin{remark}
We here would like to point out that the interior mean oscillation estimates in \eqref{DE001} under different exponents can be, respectively, seen in Theorem 5 of \cite{DM1993}, Theorem 3.3 of \cite{DM2010} and Lemma 5.1 of \cite{L1991}.
\end{remark}
From Lemma \ref{Lem001}, we see that there is no restriction on the value of the length parameter $R_{0}$ of the narrow region $\Omega_{2R_{0}}$ in terms of the establishment of interior mean oscillation estimate \eqref{DE001}. However, when $B_{r}(x_{0})$ intersects at least one of $\Gamma^{\pm}_{2R_{0}}$, we will show that it needs to restrict the value of $R_{0}$ to be small enough so that the transformed equations also preserve the similar mean oscillation properties to $p$-Laplace equation.

\subsection{Mean oscillation decay estimates under intermediate $r$}
In the second case when $B_{r}(x_{0})$ only intersects one of $\Gamma^{\pm}_{2R_{0}}$, assume without loss of generality that $B_{r}(x_{0})\cap\Gamma_{2R_{0}}^{-}\neq\emptyset$ and $B_{r}(x_{0})\cap\Gamma_{2R_{0}}^{+}=\emptyset$. Then we can find a point $\hat{x}_{0}=(\hat{x}_{0}',-\varepsilon/2+h_{2}(\hat{x}_{0}'))\in\Gamma_{2R_{0}}^{-}$ satisfying that $\mathrm{dist}(x_{0},\Gamma^{-}_{2R_{0}})=|x_{0}-\hat{x}_{0}|$. Similarly, there exists some point $\tilde{x}_{0}=(\tilde{x}_{0}',\varepsilon/2+h_{1}(\tilde{x}_{0}'))\in\Gamma_{2R_{0}}^{+}$ such that $\mathrm{dist}(\hat{x}_{0},\Gamma^{+}_{2R_{0}})=|\hat{x}_{0}-\tilde{x}_{0}|.$

{\bf Claim 1.} If $0<R_{0}\leq R_{0,1}$, then we have $B_{r}(\hat{x}_{0})\cap\Gamma_{2R_{0}}^{+}=\emptyset$ for any $r\in(0,\frac{1}{4}\delta_{0}]$, where $R_{0,1}$ and $\delta_{0}$ are, respectively, defined by \eqref{C09} and \eqref{DELTA02}.
\begin{proof}[Proof of Claim 1.]
For $0<s\leq c_{0}\delta_{0}^{1/m}$ with $c_{0}$ given by \eqref{ZMQC960}, using conditions ({\bf{H1}}) and ({\bf{H2}}), we obtain that for $|x'-x_{0}'|\leq s$, there exist two points $x'_{\theta_{1}}$ and $x'_{\theta_{2}}$ between $x'$ and $x'_{0}$ such that
\begin{align*}
|\delta(x')-\delta_{0}|\leq&|h_{1}(x')-h_{1}(x_{0}')|+|h_{2}(x')-h_{2}(x_{0}')|\notag\\
\leq&(|\nabla_{x'}h_{1}(x'_{\theta_{1}})|+|\nabla_{x'}h_{2}(x'_{\theta_{2}})|)|x'-x_{0}'|\notag\\
\leq&\kappa_{3}s(|x'_{\theta_{1}}|^{m-1}+|x'_{\theta_{2}}|^{m-1})\notag\\
\leq&2^{m-1}\kappa_{3}s(s^{m-1}+|x_{0}'|^{m-1})\leq\frac{\delta_{0}}{2},
\end{align*}
which implies that
\begin{align}\label{QWN001}
\frac{1}{2}\delta_{0}\leq\delta(x')\leq\frac{3}{2}\delta_{0},\quad\mathrm{in}\;\overline{B'_{s}(x_{0}')}.
\end{align}

Observe that the straight line $l_{x_{0},\hat{x}_{0}}$ passing through points $x_{0}$ and $\hat{x}_{0}$ can be represented as follows: for any $x\in l_{x_{0},\hat{x}_{0}}$,
\begin{align*}
\frac{x_{i}-\hat{x}_{0,i}}{\partial_{x_{i}}h_{2}(\hat{x}_{0}')}=\frac{x_{d}+\varepsilon/2-h_{2}(\hat{x}_{0}')}{-1},\quad i=1,...,d-1.
\end{align*}
Then we have
\begin{align}\label{Z09}
|x_{0}-\hat{x}_{0}|=&\sqrt{|x_{0}'-\hat{x}_{0}'|^{2}+|x_{0,d}+\varepsilon/2-h_{2}(\hat{x}'_{0})|^{2}}\notag\\
=&|x_{0,d}+\varepsilon/2-h_{2}(\hat{x}'_{0})|\sqrt{1+|\nabla_{x'}h_{2}(\hat{x}'_{0})|^{2}}.
\end{align}
Denote $\nu_{\hat{x}_{0}}:=\frac{(-\nabla_{x'}h_{2}(\hat{x}_{0}'),1)}{\sqrt{1+|\nabla_{x'}h_{2}(\hat{x}_{0}')|^{2}}}$. Let $\hat{\theta}_{0}$ be the angle between vectors $(0',1)$ and $\nu_{\hat{x}_{0}}$. Then we have $\cos\hat{\theta}_{0}=\nu_{\hat{x}_{0}}\cdot(0',1)=\frac{1}{\sqrt{1+|\nabla_{x'}h_{2}(\hat{x}_{0}')|^{2}}}$. Hence we deduce
\begin{align}\label{WDZ01}
|x_{0}'-\hat{x}_{0}'|=&|x_{0}-\hat{x}_{0}|\sin\hat{\theta}_{0}=|x_{0,d}+\varepsilon/2-h_{2}(\hat{x}'_{0})||\nabla_{x'}h_{2}(\hat{x}_{0}')|\notag\\
\leq&\delta(\hat{x}_{0}')|\nabla_{x'}h_{2}(\hat{x}_{0}')|\leq\kappa_{3}(\varepsilon+\kappa_{2}|\hat{x}_{0}'|^{m})|\hat{x}'_{0}|^{m-1}\notag\\
\leq&2^{2m-1}\kappa_{2}\kappa_{3}R_{0}^{2m-2}|\hat{x}_{0}'|.
\end{align}
Here we used $|\hat{x}_{0}'|\leq2R_{0}$ rather than $|\hat{x}_{0}'|\leq R_{0}$ in the last inequality, since there may appear $|\hat{x}_{0}'|>R_{0}$ when the interfacial boundary of $D_{2}$ is concave near $x_{0}$. From \eqref{WDZ01}, we see that if $0<R_{0}\leq2^{-\frac{m}{m-1}}(\kappa_{2}\kappa_{3})^{-\frac{1}{2m-2}}$,
\begin{align}\label{WDZ02}
|\hat{x}_{0}'|\leq&\frac{1}{1-2^{2m-1}\kappa_{2}\kappa_{3}R_{0}^{2m-2}}|x_{0}'|\leq2|x_{0}'|.
\end{align}
By the same arguments as in \eqref{Z09}--\eqref{WDZ01}, we have
\begin{align}\label{ZT002}
|\hat{x}_{0}-\tilde{x}_{0}|=&\sqrt{|\hat{x}_{0}'-\tilde{x}_{0}'|^{2}+|\varepsilon+h_{1}(\tilde{x}'_{0})-h_{2}(\hat{x}'_{0})|^{2}}\notag\\
=&|\varepsilon+h_{1}(\tilde{x}'_{0})-h_{2}(\hat{x}'_{0})|\sqrt{1+|\nabla_{x'}h_{1}(\tilde{x}'_{0})|^{2}},
\end{align}
and
\begin{align}\label{WQ08}
|\tilde{x}_{0}'-\hat{x}_{0}'|=&|\varepsilon+h_{1}(\tilde{x}_{0}')-h_{2}(\hat{x}'_{0})||\nabla_{x'}h_{1}(\tilde{x}_{0}')|\notag\\
\leq&\delta(\hat{x}_{0}')|\nabla_{x'}h_{1}(\tilde{x}_{0}')|\leq2^{2m-1}\kappa_{2}\kappa_{3}R_{0}^{2m-2}|\tilde{x}_{0}'|,
\end{align}
which, in combination with \eqref{WDZ02}, reads that if $0<R_{0}\leq2^{-\frac{m}{m-1}}(\kappa_{2}\kappa_{3})^{-\frac{1}{2m-2}}$,
\begin{align}\label{WDZ03}
|\tilde{x}_{0}'|\leq\frac{1}{1-2^{2m-1}\kappa_{2}\kappa_{3}R_{0}^{2m-2}}|\hat{x}_{0}'|\leq \frac{1}{(1-2^{2m-1}\kappa_{2}\kappa_{3}R_{0}^{2m-2})^{2}}|x_{0}'|\leq4|x_{0}'|.
\end{align}
Then we deduce
\begin{align*}
|x_{0}'-\tilde{x}_{0}'|\leq&|x_{0}'-\hat{x}_{0}'|+|\hat{x}_{0}'-\tilde{x}_{0}'|\leq2^{2m-1}\kappa_{2}\kappa_{3}R_{0}^{2m-2}(|\hat{x}_{0}'|+|\tilde{x}_{0}'|)\notag\\
\leq&2^{2m+2}\kappa_{2}\kappa_{3}R_{0}^{2m-2}|x_{0}'|\leq2^{2m+2}\kappa_{1}^{-1/m}\kappa_{2}\kappa_{3}R_{0}^{2m-2}\delta_{0}^{1/m}.
\end{align*}
Therefore, if $R_{0}$ further satisfies
\begin{align}\label{WE03}
0<R_{0}\leq2^{-\frac{m}{m-1}}(\kappa_{2}\kappa_{3})^{-\frac{1}{2m-2}}\min\{1,2^{-\frac{1}{m-1}}(\kappa_{1}^{1/m}c_{0})^{\frac{1}{2m-2}}\},
\end{align}
then we have
\begin{align*}
|x_{0}'-\tilde{x}_{0}'|\leq c_{0}\delta^{1/m}_{0},
\end{align*}
where $c_{0}$ is given by \eqref{ZMQC960}. This, together with \eqref{QWN001}, gives that
\begin{align}\label{WDZ05}
|\varepsilon+h_{1}(\tilde{x}'_{0})-h_{2}(\hat{x}'_{0})|\geq&\delta(\tilde{x}'_{0})-|h_{2}(\tilde{x}_{0}')-h_{2}(\hat{x}_{0}')|\geq\frac{1}{2}\delta_{0}-|\nabla_{x'}h_{2}(\xi'_{0})||\tilde{x}'_{0}-\hat{x}'_{0}|\notag\\
\geq&\frac{1}{2}\delta_{0}-\kappa_{3}|\xi'_{0}|^{m-1}|\tilde{x}'_{0}-\hat{x}'_{0}|,
\end{align}
where $\xi'_{0}$ is some point between $\tilde{x}'_{0}$ and $\hat{x}'_{0}$. If condition \eqref{WE03} holds, we deduce from \eqref{WDZ02}--\eqref{WDZ03} that
\begin{align*}
|\xi'_{0}|\leq&|\xi_{0}'-\hat{x}_{0}'|+|\hat{x}_{0}'|\leq|\tilde{x}_{0}'-\hat{x}_{0}'|+|\hat{x}_{0}'|\notag\\
\leq&2^{2m-1}\kappa_{2}\kappa_{3}R_{0}^{2m-2}|\tilde{x}'_{0}|+\frac{1}{1-2^{2m-1}\kappa_{2}\kappa_{3}R_{0}^{2m-2}}|x_{0}'|\notag\\
\leq&\left(\frac{2^{2m-1}\kappa_{2}\kappa_{3}R_{0}^{2m-2}}{(1-2^{2m-1}\kappa_{2}\kappa_{3}R_{0}^{2m-2})^{2}}+\frac{1}{1-2^{2m-1}\kappa_{2}\kappa_{3}R_{0}^{2m-2}}\right)|x_{0}'|\leq4|x_{0}'|,
\end{align*}
and
\begin{align*}
|\tilde{x}_{0}'-\hat{x}_{0}'|\leq&\frac{2^{2m-1}\kappa_{2}\kappa_{3}R_{0}^{2m-2}}{(1-2^{2m-1}\kappa_{2}\kappa_{3}R_{0}^{2m-2})^{2}}|x_{0}'|\leq2^{2m+1}\kappa_{2}\kappa_{3}R_{0}^{2m-2}|x_{0}'|.
\end{align*}
Substituting these two relational expressions into \eqref{WDZ05}, we obtain that if $R_{0}$ further satisfies $0<R_{0}\leq(\frac{\kappa_{1}}{2^{4m-1}\kappa_{2}\kappa_{3}^{2}})^{\frac{1}{2m-2}}$ besides \eqref{WE03},
\begin{align*}
|\varepsilon+h_{1}(\tilde{x}'_{0})-h_{2}(\hat{x}'_{0})|\geq&\frac{1}{2}\delta_{0}-\kappa_{3}|\xi'_{0}|^{m-1}|\tilde{x}'_{0}-\hat{x}'_{0}|\geq\frac{1}{2}\delta_{0}-2^{4m-1}\kappa_{2}\kappa_{3}^{2}R_{0}^{2m-2}|x'_{0}|^{m}\notag\\
\geq&\frac{1}{2}\delta_{0}-2^{4m-1}\kappa_{1}^{-1}\kappa_{2}\kappa_{3}^{2}R_{0}^{2m-2}\delta_{0}\geq\frac{1}{4}\delta_{0}.
\end{align*}
Then in light of \eqref{ZT002}, we have
\begin{align*}
\mathrm{dist}(\hat{x}_{0},\Gamma_{2R_{0}}^{+})=|\tilde{x}_{0}-\hat{x}_{0}|\geq\frac{1}{4}\delta_{0}.
\end{align*}
This implies that Claim 1 holds.

\end{proof}
We now introduce a local coordinate $y=(y',y_{d})$ in $B_{\frac{1}{4}\delta_{0}}(\hat{x}_{0})$ satisfying that $y(\hat{x}_{0})=0$, the $y_{d}$-axis direction is consistent with the normal vector $\nu_{\hat{x}_{0}}=\frac{(-\nabla_{x'}h_{2}(\hat{x}_{0}'),1)}{\sqrt{1+|\nabla_{x'}h_{2}(\hat{x}_{0}')|^{2}}}$, and $\Gamma_{2R_{0}}^{-}\cap B_{\frac{1}{4}\delta_{0}}(\hat{x}_{0})$ can be formulated by a $C^{m}$ function $y_{d}=\chi(y')$. We now use $z=\Lambda(y)=(y',y_{d}-\chi(y'))$ to straighten the boundary. With regard to the properties of $\chi$ and $\Lambda$, we have

{\bf Claim 2.} If $0<R_{0}\leq\min\{R_{0,1},R_{0,2}\}$, we obtain that for any $r\in(0,\frac{1}{4}\delta_{0}]$,
\begin{align}\label{PER01}
\begin{cases}
\chi(0')=0,\quad \nabla_{y'}\chi(0')=0,\quad\|\chi\|_{C^{m}}\leq C\|h_{2}\|_{C^{m}},\\
|\nabla\Lambda(y)-I_{d}|=|\nabla_{y'}\chi(y')|\leq \sqrt{2}\kappa_{4}(d-1)^{2}|y'|\leq\frac{1}{2},
\end{cases}
\end{align}
and, for any $r\in(0,\frac{1}{8}\delta_{0}]$,
\begin{align}\label{PER02}
\Omega_{r/2}(\hat{x}_{0})\subset\Lambda^{-1}(B_{r}^{+})\subset\Omega_{2r}(\hat{x}_{0}),\quad |B_{r/2}^{+}|\leq|\Omega_{r}(\hat{x}_{0})|\leq|B_{2r}^{+}|,
\end{align}
where $R_{0,i}$, $i=1,2$ are given by \eqref{C09}--\eqref{C12}, $I_{d}$ is the $d\times d$ identity matrix, $B_{r}^{+}$ is the upper half of ball $B_{r}$ under the coordinate $z$, and the constant $C$ depends only on $d,\kappa_{i}$, $i=1,2,3,4$, but not on $\varepsilon$.
\begin{proof}[Proof of Claim 2.]
Note that the tangential vectors at $\hat{x}_{0}$ on the surface $\Gamma_{2R_{0}}^{-}$ are given by
\begin{align*}
\tau_{i}=e_{i}+\partial_{x_{i}}h_{2}(\hat{x}_{0})e_{d},\quad i=1,...,d-1,
\end{align*}
where $\{e_{i}\}_{i=1}^{d}$ is the standard basis of $\mathbb{R}^{d}$. Introduce the projection operator as follows:
\begin{align*}
\mathrm{proj}_{a}b=\frac{(a,b)}{(a,a)}a,
\end{align*}
where $(a,b)$ represents the inner product of these two vectors and $(a,a)=|a|^{2}.$ In order to let these above tangential vectors be orthogonal to each other, we carry out the following Gram-Schmidt process:
\begin{align*}
\tilde{\tau}_{1}=&\tau_{1},\quad\hat{\tau}_{1}=\frac{\tilde{\tau}_{1}}{|\tilde{\tau}_{1}|},\\
\tilde{\tau}_{2}=&\tau_{2}-\mathrm{proj}_{\hat{\tau}_{1}}\tau_{2},\quad\hat{\tau}_{2}=\frac{\tilde{\tau}_{2}}{|\tilde{\tau}_{2}|},\\
\vdots\;&\\
\tilde{\tau}_{d-1}=&\tau_{d-1}-\sum^{d-2}_{i=1}\mathrm{proj}_{\hat{\tau}_{i}}\tau_{d-1},\quad\hat{\tau}_{d-1}=\frac{\tilde{\tau}_{d-1}}{|\tilde{\tau}_{d-1}|}.
\end{align*}
Denote $\mathcal{R}:=(\mathcal{R}_{ij})_{d\times d}=(\hat{\tau}_{1}^{T},...,\hat{\tau}_{d-1}^{T},\nu_{\hat{x}_{0}}^{T})$. The transform from $x$ to $y$ consists of translation and rotation, which can be written as $y=\mathcal{R}^{T}(x-\hat{x}_{0})$ and $x=\hat{x}_{0}+\mathcal{R}y$. Then we obtain that for $x\in \Gamma_{2R_{0}}^{-}\cap B_{\frac{1}{4}\delta_{0}}(\hat{x}_{0})$,
\begin{align}\label{CI09}
\chi(y')=\nu_{\hat{x}_{0}}\cdot(x-\hat{x}_{0})=\frac{h_{2}(x')-h_{2}(\hat{x}_{0}')-(x'-\hat{x}_{0}')\cdot\nabla_{x'}h_{2}(\hat{x}_{0}')}{\sqrt{1+|\nabla_{x'}h_{2}(\hat{x}_{0}')|^{2}}},
\end{align}
and
\begin{align*}
\partial_{y_{j}}\chi(y')=\sum^{d-1}_{i=1}\partial_{x_{i}}\chi(y')\frac{\partial x_{i}}{\partial y_{j}}.
\end{align*}
From condition ({\bf{H3}}), we see
\begin{align*}
|\partial_{x_{i}}\chi(y')|=\frac{|\partial_{x_{i}}h_{2}(x')-\partial_{x_{i}}h_{2}(\hat{x}_{0}')|}{\sqrt{1+|\nabla_{x'}h_{2}(\hat{x}_{0}')|^{2}}}\leq\kappa_{4}|x'-\hat{x}_{0}'|,
\end{align*}
which, together with the fact that $|\frac{\partial x_{i}}{\partial y_{j}}|=|\mathcal{R}_{ij}|\leq1$, shows that
\begin{align}\label{ZM03}
|\nabla_{y'}\chi(y')|\leq\sum^{d-1}_{j=1}|\partial_{y_{j}}\chi(y')|\leq(d-1)\sum^{d-1}_{i=1}|\partial_{x_{i}}\chi(y')|\leq\kappa_{4}(d-1)^{2}|x'-\hat{x}_{0}'|.
\end{align}
Since $|y|=|x-\hat{x}_{0}|$ on $\Gamma_{2R_{0}}^{-}\cap B_{\frac{1}{4}\delta_{0}}(\hat{x}_{0})$, then
\begin{align*}
|y'|^{2}=&|x'-\hat{x}_{0}'|^{2}+(h_{2}(x')-h_{2}(\hat{x}_{0}'))^{2}-|\chi(y')|^{2}\notag\\
=&|x'-\hat{x}_{0}'|^{2}+\frac{(h_{2}(x')-h_{2}(\hat{x}_{0}'))^{2}|\nabla_{x'}h_{2}(\hat{x}_{0}')|^{2}}{1+|\nabla_{x'}h_{2}(\hat{x}_{0}')|^{2}}\notag\\
&+\frac{2(x'-\hat{x}_{0}')\cdot\nabla_{x'}h_{2}(\hat{x}_{0}')(h_{2}(x')-h_{2}(\hat{x}_{0}'))}{1+|\nabla_{x'}h_{2}(\hat{x}_{0}')|^{2}}-\frac{((x'-\hat{x}_{0}')\cdot\nabla_{x'}h_{2}(\hat{x}_{0}'))^{2}}{1+|\nabla_{x'}h_{2}(\hat{x}_{0}')|^{2}}.
\end{align*}
It then follows from condition ({\bf{H2}}) that if $0<R_{0}\leq\min\{R_{0,1},2^{-\frac{2m+1}{2m-1}}\kappa_{3}^{-\frac{1}{m-1}}\}$, where $R_{0,1}$ is given by \eqref{C09}, we have
\begin{align}\label{ZM006}
|y'|^{2}\geq&|x'-\hat{x}_{0}'|^{2}-\frac{2|x'-\hat{x}_{0}'||\nabla_{x'}h_{2}(\hat{x}_{0}')||h_{2}(x')-h_{2}(\hat{x}_{0}')|}{1+|\nabla_{x'}h_{2}(\hat{x}_{0}')|^{2}}\notag\\
&-\frac{((x'-\hat{x}_{0}')\cdot\nabla_{x'}h_{2}(\hat{x}_{0}'))^{2}}{1+|\nabla_{x'}h_{2}(\hat{x}_{0}')|^{2}}\notag\\
\geq&(1-2^{2m}\kappa_{3}^{2}R_{0}^{2m-2})|x'-\hat{x}_{0}|^{2}\geq\frac{1}{2}|x'-\hat{x}_{0}|^{2}.
\end{align}
If we further require that $0<R_{0}\leq\min\{R_{0,1},R_{0,2}\}$, then we obtain from \eqref{ZM03} and \eqref{ZM006} that
\begin{align*}
|\nabla_{y'}\chi(y')|\leq\sqrt{2}\kappa_{4}(d-1)^{2}|y'|\leq\frac{\sqrt{2}}{4}\kappa_{4}(d-1)^{2}\delta_{0}\leq\frac{\sqrt{2}}{2}\kappa_{2}\kappa_{4}(d-1)^{2}R_{0}^{m}<\frac{1}{2}.
\end{align*}
Combining these above results, we obtain that \eqref{PER01} holds. We now proceed to prove \eqref{PER02}.

Observe from conditions ({\bf{H2}})--({\bf{H3}}), \eqref{CI09} and \eqref{ZM006} that for $|y'|<r$, $r\in(0,\frac{1}{8}\delta_{0}]$,
\begin{align}\label{DQ90}
|\chi(y')|\leq\kappa_{4}|x'-\hat{x}'_{0}|^{2}\leq2\kappa_{4}r^{2}\leq\frac{1}{4}\kappa_{4}\delta_{0}r<\frac{1}{2}\kappa_{2}\kappa_{4}R_{0}^{m}r.
\end{align}
This shows that for any $z\in B_{r}^{+}$ with $r\in(0,\frac{1}{8}\delta_{0}]$, if $0<R_{0}\leq\min\{R_{0,1},R_{0,2}\}$,
\begin{align*}
|y|^{2}=&|y'|^{2}+|\chi(y')+z_{d}|^{2}=|z|^{2}+(\chi(y'))^{2}+2z_{d}\chi(y')\notag\\
<&(1+4^{-1}\kappa_{2}^{2}\kappa_{4}^{2}R_{0}^{2m}+\kappa_{2}\kappa_{4}R_{0}^{m})r^{2}<4r^{2},
\end{align*}
which implies that $z\in\Lambda(\Omega_{2r}(\hat{x}_{0}))$ and thus $\Lambda^{-1}(B_{r}^{+})\subset\Omega_{2r}(\hat{x}_{0}).$

On the other hand, for $y\in\Omega_{r/2}(\hat{x}_{0})$, $r\in(0,\frac{1}{8}\delta_{0}]$, it follows from \eqref{DQ90} that
\begin{align*}
|z|^{2}=&|y'|^{2}+|y_{d}-\chi(y')|^{2}=|y|^{2}+|\chi(y')|^{2}-2y_{d}\chi(y')\notag\\
\leq&\frac{1}{4}(1+\kappa_{2}^{2}\kappa_{4}^{2}R_{0}^{2m}+2\kappa_{2}\kappa_{4}R_{0}^{m})r^{2}<r^{2},
\end{align*}
which reads that $z\in B_{r}^{+}$ and then $\Omega_{r/2}(\hat{x}_{0})\subset\Lambda^{-1}(B_{r}^{+})$. Moreover, since $\det(\nabla\Lambda)=1$ and $\Omega_{r/2}(\hat{x}_{0})\subset\Lambda^{-1}(B_{r}^{+})\subset\Omega_{2r}(\hat{x}_{0})$, we have $|B_{r/2}^{+}|\leq|\Omega_{r}(\hat{x}_{0})|\leq|B_{2r}^{+}|$ for $r\in(0,\frac{1}{8}\delta_{0}]$. The proof of Claim 2 is complete.

\end{proof}
Remark that by precisely calculating the values of length parameters $R_{0,1}$ and $R_{0,2}$, we obtain the required properties after the transforms including translation, rotation and flattening. We now study the transformed equation and establish the corresponding mean oscillation decay estimates near $\hat{x}_{0}$. Note that $p$-Laplace equation is invariant with respect to translation and rotation of the coordinate. Set $u_{1}(z):=u(\Lambda^{-1}(z))$. Then $u_{1}$ is the solution of the following problem
\begin{align}\label{FE01}
\begin{cases}
-\mathrm{div}_{z}(|A^{T}\nabla_{z}u_{1}|^{p-2}AA^{T}\nabla_{z}u_{1})=0,&\mathrm{in}\;B^{+}_{\frac{1}{8}\delta_{0}},\\
(|A^{T}\nabla_{z}u_{1}|^{p-2}AA^{T}\nabla_{z}u_{1})_{d}=0,&\mathrm{on}\;B_{\frac{1}{8}\delta_{0}}\cap\partial\mathbb{R}^{d}_{+},
\end{cases}
\end{align}
where $A:=A(z):=(a_{ij}(z)):=\nabla\Lambda(\Lambda^{-1}(z))$. Using the conormal boundary condition, we implement the even extension of $u_{1}$, $a_{dd}$ and $a_{ij}$, $i,j=1,...,d-1$ and the odd extension of $a_{id}$ and $a_{di}$, $i=1,...,d-1$ with respect to $z_{d}=0$, respectively. For simplicity, the extended function and coefficient matrix are still represented by $u_{1}$ and $A$. Therefore, $u_{1}$ solves
\begin{align}\label{COM001}
-\mathrm{div}_{z}(\boldsymbol{\mathcal{A}}(z,\nabla_{z}u_{1}))=0,\quad\mathrm{in}\;B_{\frac{1}{8}\delta_{0}},
\end{align}
where $\boldsymbol{\mathcal{A}}$ is a nonlinear operator given by $\boldsymbol{\mathcal{A}}(z,\xi)=|A^{T}\xi|^{p-2}AA^{T}\xi$ with $z\in B_{\frac{1}{8}\delta_{0}}$ and $\xi\in\mathbb{R}^{d}$. Since $A|_{z=0}=I_{d}$, then $\boldsymbol{\mathcal{A}}(0,\xi)=|\xi|^{p-2}\xi.$ Similar to Lemma 2.3 in \cite{DYZ2023}, we have
\begin{lemma}\label{LEM030}
Let $p>1$. If $0<R_{0}\leq\min\{R_{0,1},R_{0,2}\}$, then for any $z\in B_{\frac{1}{8}\delta_{0}}$ and $\xi\in\mathbb{R}^{d}$,
\begin{align*}
|\boldsymbol{\mathcal{A}}(z,\xi)-\boldsymbol{\mathcal{A}}(0,\xi)|\leq\mathcal{M}|z'||\xi|^{p-1},
\end{align*}
where $\mathcal{M}:=\mathcal{M}(d,p,\kappa_{4})$ is given by
\begin{align}\label{INEC01}
\mathcal{M}(d,p,\kappa_{4})=&\sqrt{2}\kappa_{4}(d-1)^{2}\left(\Big(d+\frac{1}{4}\Big)^{\frac{p-1}{2}}+\bar{c}_{p-2}\Big(d+\frac{5}{4}\Big)^{\frac{p-2}{2}}\right)
\end{align}
with $\bar{c}_{p-2}$ defined by \eqref{UP01}.
\end{lemma}
\begin{proof}
By Lemma 2.1 in \cite{H1992} and Lemma 2.2 in \cite{GM1986}, we see that for any $\xi_{1},\xi_{2}\in\mathbb{R}^{d}$ and $\sigma>-1$,
\begin{align}\label{WME01}
\underline{c}_{\sigma}\leq\frac{||\xi_{1}|^{\sigma}\xi_{1}-|\xi_{2}|^{\sigma}\xi_{2}|}{|\xi_{1}-\xi_{2}|(|\xi_{1}|^{2}+|\xi_{2}|^{2})^{\frac{\sigma}{2}}}\leq \bar{c}_{\sigma},
\end{align}
where $\underline{c}_{\sigma}$ and $\bar{c}_{\sigma}$ are, respectively, defined by
\begin{align}\label{UP01}
\underline{c}_{\sigma}=
\begin{cases}
1+\sigma,&-1<\sigma\leq0,\\
5^{-(1+\frac{\sigma}{2})},&\sigma>0,
\end{cases}\quad
\bar{c}_{\sigma}=
\begin{cases}
\max\{2,10^{\frac{1}{2}|\sigma|}\},&-1<\sigma<0,\\
(1+\sigma)2^{\frac{\sigma}{2}},&\sigma\geq0.
\end{cases}
\end{align}
It then follows from \eqref{PER01} and \eqref{WME01} that
\begin{align*}
|\boldsymbol{\mathcal{A}}(z,\xi)-\boldsymbol{\mathcal{A}}(0,\xi)|\leq&||A^{T}\xi|^{p-2}(A-I_{d})A^{T}\xi|+||A^{T}\xi|^{p-2}A^{T}\xi-|\xi|^{p-2}\xi|\notag\\
\leq&|A-I_{d}||A^{T}\xi|^{p-1}+\bar{c}_{p}(|A^{T}\xi|^{2}+|\xi|^{2})^{\frac{p-2}{2}}|A^{T}\xi-\xi|\notag\\
\leq&\left(|A^{T}|^{p-1}+\bar{c}_{p}(|A^{T}|^{2}+1)^{\frac{p-2}{2}}\right)|A^{T}-I_{d}||\xi|^{p-1}\notag\\
\leq&\sqrt{2}\kappa_{4}(d-1)^{2}\left(\Big(d+\frac{1}{4}\Big)^{\frac{p-1}{2}}+\bar{c}_{p}\Big(d+\frac{5}{4}\Big)^{\frac{p-2}{2}}\right)|z'||\xi|^{p-1},
\end{align*}
where we also used the fact that $|A^{T}-I_{d}|=|A-I_{d}|$ and $|A|=\sqrt{d+|\nabla_{y'}\chi(y')|^{2}}$.

\end{proof}

For $r\in(0,\frac{1}{8}\delta_{0}]$, we define $v_{1}\in W^{1,p}(B_{r})$ as the unique solution of the following homogeneous Dirichlet problem
\begin{align}\label{COM002}
\begin{cases}
-\mathrm{div}_{z}(\boldsymbol{\mathcal{A}}(0,\nabla_{z}v_{1}))=0,&\mathrm{in}\;B_{r},\\
v_{1}=u_{1},&\mathrm{on}\;\partial B_{r},
\end{cases}
\end{align}
where $u_{1}$ satisfies equation \eqref{COM001}. In order to establish mean oscillation estimate for the solution $u_{1}$ to the transformed equation \eqref{COM001}, we need the following comparison estimate.
\begin{lemma}\label{Lem005}
For $p>1$, let $u_{1}$ and $v_{1}$ be the solutions of \eqref{COM001} and \eqref{COM002}, respectively. If $0<R_{0}\leq\min\{R_{0,1},R_{0,2}\}$,  then for any $r\in(0,\frac{1}{8}\delta_{0}]$,
\begin{align}\label{M09}
\dashint_{B_{r}}|\nabla_{z}u_{1}-\nabla_{z}v_{1}|^{p}\leq\mathfrak{S}_{p}r^{\min\{2,p\}}\dashint_{B_{r}}|\nabla_{z}u_{1}|^{p},
\end{align}
where $\mathfrak{S}_{p}$ is given by
\begin{align*}
\mathfrak{S}_{p}=
\begin{cases}
\frac{\bar{c}^{p}_{\frac{p-2}{2}}2^{\frac{p(2-p)}{2}}\mathcal{M}^{p}}{\underline{c}^{p}_{\frac{p-2}{2}}(p-1)^{p}}\left[\left(\frac{2^{p}(p-1)^{p-1}}{p^{p}}\right)^{\frac{2-p}{2}}+1\right],&1<p<2,\\
4\max\{1,2^{p-3}\}\mathcal{M}^{2},&p\geq2
\end{cases}
\end{align*}
with $\mathcal{M},\,\bar{c}_{\frac{p-2}{2}},\,\underline{c}_{\frac{p-2}{2}}$ defined by \eqref{INEC01}--\eqref{WME01}.

\end{lemma}
\begin{remark}
Although the proof of Lemma \ref{Lem005} has been contained in \cite{DM2010,DM2011} for the general degenerate elliptic equations, we also present its proof with an explicit constant $\mathfrak{S}_{p}$ in \eqref{M09}.
\end{remark}
Before proving Lemma \ref{Lem005}, we recall a classical inequality (see Chapter 12 in \cite{L2019}) as follows.
\begin{lemma}\label{Lem06}
Let $d\geq2$. For any $\xi_{1},\xi_{2}\in\mathbb{R}^{d}$,

$(i)$ if $1<p<2$,
\begin{align}\label{DP01}
\langle |\xi_{1}|^{p-2}\xi_{1}-|\xi_{2}|^{p-2}\xi_{2},\xi_{1}-\xi_{2}\rangle\geq\frac{p-1}{2^{\frac{2-p}{2}}}(|\xi_{1}|^{2}+|\xi_{2}|^{2})^{\frac{p-2}{2}}|\xi_{1}-\xi_{2}|^{2};
\end{align}

$(ii)$ if $p\geq2$,
\begin{align}\label{DP02}
\langle |\xi_{1}|^{p-2}\xi_{1}-|\xi_{2}|^{p-2}\xi_{2},\xi_{1}-\xi_{2}\rangle\geq\frac{1}{2}(|\xi_{1}|^{p-2}+|\xi_{2}|^{p-2})|\xi_{1}-\xi_{2}|^{2}.
\end{align}
\end{lemma}
For readers' convenience, the proof of Lemma \ref{Lem06} is left in the appendix.
\begin{proof}[Proof of Lemma \ref{Lem005}]
We divide the proof into two cases including $1<p<2$ and $p\geq2$.

{\bf Step 1.} Consider the case when $1<p<2$. Due to the fact that both $u_{1}$ and $v_{1}$ are the solutions, it follows from \eqref{WME01}, Lemmas \ref{LEM030} and \ref{Lem06}, H\"{o}lder's inequality and Young's inequality that
\begin{align*}
&\int_{B_{r}}(|\nabla_{z}u_{1}|^{2}+|\nabla_{z}v_{1}|^{2})^{\frac{p-2}{2}}|\nabla_{z}u_{1}-\nabla_{z}v_{1}|^{2}\notag\\
&\leq\frac{2^{\frac{2-p}{2}}}{p-1}\int_{B_{r}}\langle\boldsymbol{\mathcal{A}}(0,\nabla_{z}u_{1})-\boldsymbol{\mathcal{A}}(0,\nabla_{z}v_{1}),\nabla_{z}u_{1}-\nabla_{z}v_{1}\rangle\notag\\
&=\frac{2^{\frac{2-p}{2}}}{p-1}\int_{B_{r}}\langle\boldsymbol{\mathcal{A}}(0,\nabla_{z}u_{1})-\boldsymbol{\mathcal{A}}(z,\nabla_{z}u_{1}),\nabla_{z}u_{1}-\nabla_{z}v_{1}\rangle\notag\\
&\leq\frac{2^{\frac{2-p}{2}}\mathcal{M}r}{p-1}\int_{B_{r}}|\nabla_{z}u_{1}|^{p-1}|\nabla_{z}u_{1}-\nabla_{z}v_{1}|\notag\\
&\leq\frac{1}{2}\int_{B_{r}}|\nabla_{z}u_{1}|^{p-2}|\nabla_{z}u_{1}-\nabla_{z}v_{1}|^{2}+\frac{2^{2-p}\mathcal{M}^{2}r^{2}}{2(p-1)^{2}}\int_{B_{r}}|\nabla_{z}u_{1}|^{p},
\end{align*}
which reads that
\begin{align*}
\int_{B_{r}}(|\nabla_{z}u_{1}|^{2}+|\nabla_{z}v_{1}|^{2})^{\frac{p-2}{2}}|\nabla_{z}u_{1}-\nabla_{z}v_{1}|^{2}\leq\frac{2^{2-p}\mathcal{M}^{2}r^{2}}{(p-1)^{2}}\int_{B_{r}}|\nabla_{z}u_{1}|^{p}.
\end{align*}
For notational convenience, denote $V(\xi):=|\xi|^{\frac{p-2}{2}}\xi$ for any $\xi\in\mathbb{R}^{d}$. It then follows from \eqref{WME01} that
\begin{align}\label{EQ011}
\dashint_{B_{r}}|V(\nabla_{z}u_{1})-V(\nabla_{z}v_{1})|^{2}\leq \frac{\bar{c}^{2}_{\frac{p-2}{2}}2^{2-p}\mathcal{M}^{2}r^{2}}{(p-1)^{2}}\dashint_{B_{r}}|\nabla_{z}u_{1}|^{p},
\end{align}
where $\bar{c}_{\frac{p-2}{2}}$ and $\mathcal{M}$ are given by \eqref{INEC01}--\eqref{WME01}. Observe from \eqref{WME01} that
\begin{align}\label{EQ012}
&|\nabla_{z}u_{1}-\nabla_{z}v_{1}|^{p}\notag\\
&=\big[(|\nabla_{z}u_{1}|^{2}+|\nabla_{z}v_{1}|^{2})^{\frac{p-2}{2}}|\nabla_{z}u_{1}-\nabla_{z}v_{1}|^{2}\big]^{\frac{p}{2}}(|\nabla_{z}u_{1}|^{2}+|\nabla_{z}v_{1}|^{2})^{\frac{p(2-p)}{4}}\notag\\
&=\frac{1}{\overline{c}^{p}_{\frac{p-2}{2}}}|V(\nabla_{z}u_{1})-V(\nabla_{z}v_{1})|^{p}(|\nabla_{z}u_{1}|^{2}+|\nabla_{z}v_{1}|^{2})^{\frac{p(2-p)}{4}}.
\end{align}
Picking test function $v_{1}-u_{1}$ for equation \eqref{COM002}, it follows from Young's inequality that
\begin{align*}
\int_{B_{r}}|\nabla_{z}v_{1}|^{p}=&\int_{B_{r}}\boldsymbol{\mathcal{A}}(0,\nabla_{z}v_{1})\nabla_{z}v_{1}=\int_{B_{r}}\boldsymbol{\mathcal{A}}(0,\nabla_{z}v_{1})\nabla_{z}u_{1}\notag\\
\leq&\frac{1}{2}\int_{B_{r}}|\nabla_{z}v_{1}|^{p}+\frac{2^{p-1}(p-1)^{p-1}}{p^{p}}\int_{B_{r}}|\nabla_{z}u_{1}|^{p}.
\end{align*}
This, together with \eqref{EQ011}--\eqref{EQ012} and H\"{o}lder's inequality, yields that
\begin{align*}
&\dashint_{B_{r}}|\nabla_{z}u_{1}-\nabla_{z}v_{1}|^{p}\notag\\
&\leq\frac{1}{\underline{c}^{p}_{\frac{p-2}{2}}}\left(\dashint_{B_{r}}|V(\nabla_{z}u_{1})-V(\nabla_{z}v_{1})|^{p}\right)^{\frac{p}{2}}\left(\dashint_{B_{r}}(|\nabla_{z}u_{1}|^{2}+|\nabla_{z}v_{1}|^{2})^{\frac{p}{2}}\right)^{\frac{2-p}{2}}\notag\\
&\leq \frac{\bar{c}^{p}_{\frac{p-2}{2}}2^{\frac{p(2-p)}{2}}\mathcal{M}^{p}r^{p}}{\underline{c}^{p}_{\frac{p-2}{2}}(p-1)^{p}}\left(\dashint_{B_{r}}|\nabla_{z}u_{1}|^{p}\right)^{\frac{p}{2}}\left(\dashint_{B_{r}}(|\nabla_{z}u_{1}|^{2}+|\nabla_{z}v_{1}|^{2})^{\frac{p}{2}}\right)^{\frac{2-p}{2}}\notag\\
&\leq \frac{\bar{c}^{p}_{\frac{p-2}{2}}2^{\frac{p(2-p)}{2}}\mathcal{M}^{p}r^{p}}{\underline{c}^{p}_{\frac{p-2}{2}}(p-1)^{p}}\left(\dashint_{B_{r}}|\nabla_{z}u_{1}|^{p}\right)^{\frac{p}{2}}\left(\left(\dashint_{B_{r}}|\nabla_{z}u_{1}|^{p}\right)^{\frac{2-p}{2}}+\left(\dashint_{B_{r}}|\nabla_{z}v_{1}|^{p}\right)^{\frac{2-p}{2}}\right)\notag\\
&\leq \frac{\bar{c}^{p}_{\frac{p-2}{2}}2^{\frac{p(2-p)}{2}}\mathcal{M}^{p}}{\underline{c}^{p}_{\frac{p-2}{2}}(p-1)^{p}}\left[\left(\frac{2^{p}(p-1)^{p-1}}{p^{p}}\right)^{\frac{2-p}{2}}+1\right]r^{p}\dashint_{B_{r}}|\nabla_{z}u_{1}|^{p}.
\end{align*}

{\bf Step 2.} Consider the case of $p\geq2$. Analogously as above, since both $u_{1}$ and $v_{1}$ are the solutions, we have from H\"{o}lder's inequality, Young's inequality, Lemmas \ref{LEM030} and \ref{Lem06} that
\begin{align*}
&\int_{B_{r}}(|\nabla_{z}u_{1}|^{p-2}+|\nabla_{z}v_{1}|^{p-2})|\nabla_{z}u_{1}-\nabla_{z}v_{1}|^{2}\notag\\
&\leq2\int_{B_{r}}\langle \boldsymbol{\mathcal{A}}(0,\nabla_{z}u_{1})-\boldsymbol{\mathcal{A}}(0,\nabla_{z}v_{1}),\nabla_{z}u_{1}-\nabla_{z}v_{1}\rangle\notag\\
&=2\int_{B_{r}}\langle \boldsymbol{\mathcal{A}}(0,\nabla_{z}u_{1})-\boldsymbol{\mathcal{A}}(z,\nabla_{z}u_{1}),\nabla_{z}u_{1}-\nabla_{z}v_{1}\rangle\notag\\
&\leq2\mathcal{M}r\int_{B_{r}}|\nabla_{z}u_{1}|^{p-1}|\nabla_{z}u_{1}-\nabla_{z}v_{1}|\notag\\
&\leq\frac{1}{2}\int_{B_{r}}|\nabla_{z}u_{1}|^{p-2}|\nabla_{z}u_{1}-\nabla_{z}v_{1}|^{2}+2\mathcal{M}^{2}r^{2}\int_{B_{r}}|\nabla_{z}u_{1}|^{p},
\end{align*}
which gives that
\begin{align*}
\int_{B_{r}}(|\nabla_{z}u_{1}|^{p-2}+|\nabla_{z}v_{1}|^{p-2})|\nabla_{z}u_{1}-\nabla_{z}v_{1}|^{2}\leq4\mathcal{M}^{2}r^{2}\int_{B_{r}}|\nabla_{z}u_{1}|^{p}.
\end{align*}
Then we have
\begin{align*}
\dashint_{B_{r}}|\nabla_{z}u_{1}-\nabla_{z}v_{1}|^{p}\leq&\max\{1,2^{p-3}\}\dashint_{B_{r}}(|\nabla_{z}u_{1}|^{p-2}+|\nabla_{z}v_{1}|^{p-2})|\nabla_{z}u_{1}-\nabla_{z}v_{1}|^{2}\notag\\
\leq&4\max\{1,2^{p-3}\}\mathcal{M}^{2}r^{2}\int_{B_{r}}|\nabla_{z}u_{1}|^{p}.
\end{align*}
The proof is complete.

\end{proof}

For $r\in(0,\frac{1}{8}\delta_{0}]$, denote
\begin{align}\label{PSI}
\psi(\hat{x}_{0},r)=\left(\dashint_{\Omega_{r}(\hat{x}_{0})}|\nabla_{y'}u-(\nabla_{y'}u)_{\Omega_{r}(\hat{x}_{0})}|^{p}+|\partial_{y_{d}}u|^{p}\right)^{\frac{1}{p}},
\end{align}
where $u\in W^{1,p}(\Omega_{2R_{0}})$ is the solution to \eqref{problem006}. Similar to Lemma 2.5 in \cite{DYZ2023}, by Lemmas \ref{Lem001} and \ref{Lem005}, we can obtain mean oscillation estimate for the solution $u$ near $\hat{x}_{0}$ as follows.
\begin{lemma}\label{Lem096}
Assume as above. Let $u\in W^{1,p}(\Omega_{2R_{0}})$ be the solution to \eqref{problem006}. If $0<R_{0}\leq\min\{R_{0,1},R_{0,2}\}$, then for any $r\in(0,\frac{1}{8}\delta_{0}]$ and $\mu\in(0,\frac{1}{4})$,
\begin{align*}
\psi(\hat{x}_{0},\mu r)\leq C_{1}\mu^{\alpha}\psi(\hat{x}_{0},r)+C_{2}\mu^{-\frac{d}{p}}r^{\theta}\left(\dashint_{\Omega_{r}(\hat{x}_{0})}|\nabla u|^{p}\right)^{\frac{1}{p}},
\end{align*}
where $\theta:=\theta_{p}=\min\{1,2/p\}$, $C_{1}=C_{1}(d,p)$, $C_{2}=C_{2}(d,p,\kappa_{4})$, $\alpha$ is given by Lemma \ref{Lem001} and $\psi$ is defined by \eqref{PSI}. Further, choose $\mu_{2}=\min\{(2C_{1})^{-\frac{1}{\alpha}},\frac{1}{4}\}$ such that for any $r\in(0,\frac{1}{8}\delta_{0}]$, $\mu\in(0,\mu_{2})$, and $0\leq i_{0}<j$,
\begin{align*}
&\sum^{j}_{k=i_{0}}\psi(\hat{x}_{0},\mu^{k}r)\notag\\
&\leq2\psi(\hat{x}_{0},\mu^{i_{0}}r)+C_{2}\mu^{-\frac{d}{p}}\sum^{j}_{k=i_{0}+1}\sum^{k-1}_{i=i_{0}}(C_{1}\mu^{\alpha})^{k-1-i}(\mu^{i}r)^{\theta}\left(\dashint_{\Omega_{\mu^{i}r}(\hat{x}_{0})}|\nabla u|^{p}\right)^{\frac{1}{p}}.
\end{align*}
\end{lemma}
\begin{proof}
Let $u_{1}\in W^{1,p}(B_{\frac{1}{8}\delta_{0}}^{+})$ be the solution to equation \eqref{FE01}. We first establish mean oscillation estimate for the solution $u_{1}$ as follows: for any $\mu\in(0,1)$ and $r\in(0,\frac{1}{8}\delta_{0}]$,
\begin{align}\label{DJ093}
&\left(\dashint_{B_{\mu r}^{+}}|\nabla_{z'}u_{1}-(\nabla_{z'}u_{1})_{B^{+}_{\mu r}}|^{p}+|\partial_{z_{d}}u_{1}|^{p}\right)^{\frac{1}{p}}\notag\\
&\leq C_{1}\mu^{\alpha}\left(\dashint_{B_{r}^{+}}|\nabla_{z'}u_{1}-(\nabla_{z'}u_{1})_{B^{+}_{r}}|^{p}+|\partial_{z_{d}}u_{1}|^{p}\right)^{\frac{1}{p}}+C_{2}\mu^{-\frac{d}{p}}r^{\theta}\left(\int_{B_{r}^{+}}|\nabla_{z}u_{1}|^{p}\right)^{\frac{1}{p}},
\end{align}
where $C_{1}=C_{1}(d,p)$, $C_{2}=C_{2}(d,p,\kappa_{4})$, $\theta=\min\{1,2/p\}$ and $\alpha$ is given by Lemma \ref{Lem001}.

Utilizing Lemma 6.13 in \cite{L1996}, the triangle inequality, Lemmas \ref{Lem001} and \ref{Lem005}, we have
\begin{align}\label{FEA90}
&\left(\dashint_{B_{\mu r}}|\nabla_{z}u_{1}-(\nabla_{z}u_{1})_{B_{\mu r}}|^{p}\right)^{\frac{1}{p}}\leq2\left(\dashint_{B_{\mu r}}|\nabla_{z}u_{1}-(\nabla_{z}v_{1})_{B_{\mu r}}|^{p}\right)^{\frac{1}{p}}\notag\\
&\leq2\left(\dashint_{B_{\mu r}}|\nabla_{z}v_{1}-(\nabla_{z}v_{1})_{B_{\mu r}}|^{p}\right)^{\frac{1}{p}}+2\left(\dashint_{B_{\mu r}}|\nabla_{z}u_{1}-\nabla_{z}v_{1}|^{p}\right)^{\frac{1}{p}}\notag\\
&\leq C\mu^{\alpha}\left(\dashint_{B_{r}}|\nabla_{z}v_{1}-(\nabla_{z}v_{1})_{B_{r}}|^{p}\right)^{\frac{1}{p}}+2\mu^{-\frac{d}{p}}\left(\dashint_{B_{r}}|\nabla_{z}u_{1}-\nabla_{z}v_{1}|^{p}\right)^{\frac{1}{p}}\notag\\
&\leq C\mu^{\alpha}\left(\dashint_{B_{r}}|\nabla_{z}u_{1}-(\nabla_{z}u_{1})_{B_{r}}|^{p}\right)^{\frac{1}{p}}+C\mu^{-\frac{d}{p}}\left(\dashint_{B_{r}}|\nabla_{z}u_{1}-\nabla_{z}v_{1}|^{p}\right)^{\frac{1}{p}}\notag\\
&\leq C_{1}\mu^{\alpha}\left(\dashint_{B_{r}}|\nabla_{z}u_{1}-(\nabla_{z}u_{1})_{B_{r}}|^{p}\right)^{\frac{1}{p}}+C_{2}\mu^{-\frac{d}{p}}r^{\theta}\left(\dashint_{B_{r}}|\nabla_{z}u_{1}|^{p}\right)^{\frac{1}{p}},
\end{align}
where $C_{1}=C_{1}(d,p)$, $C_{2}=C_{2}(d,p,\kappa_{4})$, and we also used the fact that
\begin{align*}
&\left(\dashint_{B_{r}}|\nabla_{z}v_{1}-(\nabla_{z}v_{1})_{B_{r}}|^{p}\right)^{\frac{1}{p}}\leq2\left(\dashint_{B_{r}}|\nabla_{z}v_{1}-(\nabla_{z}u_{1})_{B_{r}}|^{p}\right)^{\frac{1}{p}}\notag\\
&\leq2\left(\dashint_{B_{r}}|\nabla_{z}u_{1}-(\nabla_{z}u_{1})_{B_{r}}|^{p}\right)^{\frac{1}{p}}+2\left(\dashint_{B_{r}}|\nabla_{z}v_{1}-\nabla_{z}u_{1}|^{p}\right)^{\frac{1}{p}}.
\end{align*}
Note that $u_{1}$ is even in $z_{d}$ and thus $\partial_{z_{d}}u_{1}$ is odd in $z_{d}$. Then \eqref{DJ093} follows from \eqref{FEA90}.

Recall an elementary inequality that for any $a,b\in\mathbb{R}^{d}$ and $p>1$,
\begin{align}\label{ZMD90}
|a+b|^{p}\geq2^{1-p}|a|^{p}-|b|^{p}.
\end{align}
For $\mu\in(0,1)$ and $r\in(0,\frac{1}{16}\delta_{0}]$, it follows from \eqref{PER01}--\eqref{PER02} and \eqref{ZMD90} that
\begin{align}\label{Q96}
&\dashint_{B_{\mu r}^{+}}|\nabla_{z'}u_{1}-(\nabla_{z'}u_{1})_{B^{+}_{\mu r}}|^{p}+|\partial_{z_{d}}u|^{p}\notag\\
&=\dashint_{\Lambda^{-1}(B^{+}_{\mu r})}|\nabla_{y'}u+\partial_{y_{d}}u\nabla_{y'}\chi-(\nabla_{y'}u+\partial_{y_{d}}u\nabla_{y'}\chi)_{\Lambda^{-1}(B^{+}_{\mu r})}|^{p}+|\partial_{y_{d}}u|^{p}\notag\\
&\geq2^{1-p}\dashint_{\Omega_{\mu r/2}(\hat{x}_{0})}(|\nabla_{y'}u-(\nabla_{y'}u)_{\Omega_{\mu r/2}(\hat{x}_{0})}|^{p}+|\partial_{y_{d}}u|^{p})-2^{p}\dashint_{\Omega_{\mu r/2}(\hat{x}_{0})}|\partial_{y_{d}}u\nabla_{y'}\chi|^{p}\notag\\
&\geq2^{1-p}(\psi(\hat{x}_{0},\mu r/2))^{p}-(\sqrt{2}\kappa_{4}(d-1)^{2}\mu r)^{p}\dashint_{\Omega_{\mu r/2}(\hat{x}_{0})}|\nabla u|^{p},
\end{align}
where we also utilized the fact that
\begin{align*}
&\dashint_{\Omega_{\mu r/2}(\hat{x}_{0})}|\partial_{y_{d}}u\nabla_{y'}\chi-(\partial_{y_{d}}u\nabla_{y'}\chi)_{\Omega_{\mu r/2}(\hat{x}_{0})}|^{p}\notag\\
&\leq2^{p-1}\left(\dashint_{\Omega_{\mu r/2}(\hat{x}_{0})}|\partial_{y_{d}}u\nabla_{y'}\chi|^{p}+|(\partial_{y_{d}}u\nabla_{y'}\chi)_{\Omega_{\mu r/2}(\hat{x}_{0})}|^{p}\right)\notag\\
&\leq2^{p}\dashint_{\Omega_{\mu r/2}(\hat{x}_{0})}|\partial_{y_{d}}u\nabla_{y'}\chi|^{p}.
\end{align*}
By the same argument, we have
\begin{align}\label{Q98}
&\dashint_{B_{r}^{+}}|\nabla_{z'}u_{1}-(\nabla_{z'}u_{1})_{B^{+}_{r}}|^{p}+|\partial_{z_{d}}u|^{p}\notag\\
&\leq2^{p-1}(\psi(\hat{x}_{0},2r))^{p}+2^{3p-1}(\sqrt{2}\kappa_{4}(d-1)^{2}r)^{p}\dashint_{\Omega_{2r}(\hat{x}_{0})}|\nabla u|^{p}.
\end{align}
A consequence of \eqref{DJ093} and \eqref{Q96}--\eqref{Q98} shows that
\begin{align}\label{MME60}
\psi(\hat{x}_{0},\mu r/2)\leq C_{1}\mu^{\alpha}\psi(\hat{x}_{0},2r)+C_{2}\mu^{-\frac{d}{p}}r^{\theta}\left(\dashint_{\Omega_{2r}(\hat{x}_{0})}|\nabla u|^{p}\right)^{\frac{1}{p}},
\end{align}
where $C_{1}=C_{1}(d,p)$ and $C_{2}=C_{2}(d,p,\kappa_{4})$. Letting $\mu,r$ substitute for $\mu/4,2r$ in \eqref{MME60}, we obtain that for any $\mu\in(0,\frac{1}{4})$ and $r\in(0,\frac{1}{8}\delta_{0}]$,
\begin{align*}
\psi(\hat{x}_{0},\mu r)\leq C_{1}\mu^{\alpha}\psi(\hat{x}_{0},r)+C_{2}\mu^{-\frac{d}{p}}r^{\theta}\left(\dashint_{\Omega_{r}(\hat{x}_{0})}|\nabla u|^{p}\right)^{\frac{1}{p}}.
\end{align*}
Then for $r\in(0,\frac{1}{8}\delta_{0}]$, and $0\leq i_{0}<k$,
\begin{align*}
&\psi(\hat{x}_{0},\mu^{k}r)\notag\\
&\leq(C_{1}\mu^{\alpha})^{k-i_{0}}\psi(\hat{x}_{0},\mu^{i_{0}}r)+C_{2}\mu^{-\frac{d}{p}}\sum^{k-1}_{i=i_{0}}(C_{1}\mu^{\alpha})^{k-1-i}(\mu^{i}r)^{\theta}\left(\dashint_{\Omega_{\mu^{i}r}(\hat{x}_{0})}|\nabla u|^{p}\right)^{\frac{1}{p}},
\end{align*}
which leads to that for any $0<\mu<\min\{(2C_{1})^{-\frac{1}{\alpha}},\frac{1}{4}\}$ and $0\leq i_{0}<j$,
\begin{align*}
&\sum^{j}_{k=i_{0}}\psi(\hat{x}_{0},\mu^{k}r)\notag\\
&\leq\frac{\psi(\hat{x}_{0},\mu^{i_{0}}r)}{1-C_{1}\mu^{\alpha}}+C_{2}\mu^{-\frac{d}{p}}\sum^{j}_{k=i_{0}+1}\sum^{k-1}_{i=i_{0}}(C_{1}\mu^{\alpha})^{k-1-i}(\mu^{i}r)^{\theta}\left(\dashint_{\Omega_{\mu^{i}r}(\hat{x}_{0})}|\nabla u|^{p}\right)^{\frac{1}{p}}\notag\\
&\leq2\psi(\hat{x}_{0},\mu^{i_{0}}r)+C_{2}\mu^{-\frac{d}{p}}\sum^{j}_{k=i_{0}+1}\sum^{k-1}_{i=i_{0}}(C_{1}\mu^{\alpha})^{k-1-i}(\mu^{i}r)^{\theta}\left(\dashint_{\Omega_{\mu^{i}r}(\hat{x}_{0})}|\nabla u|^{p}\right)^{\frac{1}{p}}.
\end{align*}

\end{proof}

\subsection{Mean oscillation decay estimates under large $r$}

In this subsection, we aim to establish mean oscillation estimates for the solution to the transformed equation in the case when $B_{r}(x_{0})$ may potentially intersect both $\Gamma_{2R_{0}}^{+}$ and $\Gamma_{2R_{0}}^{-}$. Consider $r\in[\frac{1}{24}\delta_{0},\tilde{c}_{0}\delta_{0}^{1/m}]$, where $\tilde{c}_{0}$ is given by \eqref{C681}. Define a constant as follows:
\begin{align}\label{WE860}
\overline{R}_{0}:=\left(\frac{12}{13}\tilde{c}_{0}\right)^{\frac{1}{m-1}}\kappa_{2}^{-1/m}.
\end{align}
Then if $0<R_{0}\leq \overline{R}_{0}$, then $\frac{1}{24}\delta_{0}<\tilde{c}_{0}\delta_{0}^{1/m}$. Introduce a cylinder as follows: for $s,t>0$,
\begin{align}\label{CI90}
Q_{s,t}:=\{\tilde{z}=(\tilde{z}',\tilde{z}_{d})\in\mathbb{R}^{d}\,|\,|\tilde{z}'|<s,\,|\tilde{z}_{d}|<t\}.
\end{align}
Under the flattening transform $\tilde{z}=\tilde{\Lambda}(x)$ as follows:
\begin{align*}
\begin{cases}
\tilde{z}'=x'-x_{0}',\\
\tilde{z}_{d}=\delta_{0}\left(\frac{x_{d}-h_{2}(x')+\varepsilon/2}{\varepsilon+h_{1}(x')-h_{2}(x')}-\frac{1}{2}\right),
\end{cases}
\end{align*}
the narrow region $\Omega_{x_{0}',R_{0}}$ is mapped to be a cylinder $Q_{R_{0},\frac{1}{2}\delta_{0}}$ with its upper and lower boundaries written as
\begin{align*}
\widetilde{\Gamma}^{\pm}_{R_{0},\frac{1}{2}\delta_{0}}:=\Big\{\tilde{z}=(\tilde{z}',\tilde{z}_{d})\in\mathbb{R}^{d}\,|\,|\tilde{z}'|<R_{0},\,\tilde{z}_{d}=\pm\frac{1}{2}\delta_{0}\Big\}.
\end{align*}
Define $\tilde{u}_{1}(\tilde{z})=u(\tilde{\Lambda}^{-1}(\tilde{z}))$. Then $\tilde{u}_{1}$ satisfies
\begin{align*}
\begin{cases}
-\mathrm{div}_{\tilde{z}}(|\tilde{A}^{T}\nabla_{\tilde{z}}\tilde{u}_{1}|^{p-2}(\det(\tilde{A}))^{-1}\tilde{A}\tilde{A}^{T}\nabla_{\tilde{z}}\tilde{u}_{1})=0,&\mathrm{in}\;Q_{R_{0},\frac{1}{2}\delta_{0}},\\
(|\tilde{A}^{T}\nabla_{\tilde{z}}\tilde{u}_{1}|^{p-2}(\det(\tilde{A}))^{-1}\tilde{A}\tilde{A}^{T}\nabla_{\tilde{z}}\tilde{u}_{1})_{d}=0,&\mathrm{on}\;\widetilde{\Gamma}^{\pm}_{R_{0},\frac{1}{2}\delta_{0}},
\end{cases}
\end{align*}
where $\tilde{A}:=\tilde{A}(\tilde{z}):=(\tilde{a}_{ij}(\tilde{z})):=\nabla\tilde{\Lambda}(\tilde{\Lambda}^{-1}(\tilde{z}))$. For $\tilde{z}\in Q_{R_{0},\frac{1}{2}\delta_{0}}$, let $x=\tilde{\Lambda}^{-1}(\tilde{z})$. Recall that $\delta_{0}=\delta(x_{0}')$ and $\delta=\delta(x')$. Then the elements of $\tilde{A}$ satisfy that $\tilde{a}_{ii}=1$ for $i=1,...,d-1$, $\tilde{a}_{ij}=0$ for $i\neq j,\,i=1,...,d-1,\,j=1,2,...,d,$ $\det(\tilde{A})=\tilde{a}_{dd}=\delta_{0}\delta^{-1}$, and
\begin{align*}
\tilde{a}_{dj}=-\delta_{0}\delta^{-1}\partial_{x_{j}}h_{2}(x')-\delta^{-1}\Big(\tilde{z}_{d}+\frac{\delta_{0}}{2}\Big)\partial_{x_{j}}\delta,\quad\mathrm{for}\;j=1,...,d-1.
\end{align*}
Therefore, for any $\tilde{z}\in Q_{R_{0},\frac{1}{2}\delta_{0}}$ with $|\tilde{z}'|=|x'-x_{0}'|\leq r$, $r\in[\frac{1}{24}\delta_{0},\tilde{c}_{0}\delta_{0}^{1/m}]$, $0<\tilde{c}_{0}\leq c_{0}$, it follows from conditions ({\bf{H1}})--({\bf{H2}}) and \eqref{QWN001} that
\begin{align}\label{FQ013}
\sum^{d-1}_{j=1}|\tilde{a}_{dj}|^{2}\leq&\sum^{d-1}_{j=1}(2|\partial_{x_{j}}h_{1}|+4|\partial_{x_{j}}h_{2}|)^{2}\leq8|\nabla_{x'}h_{1}|^{2}+32|\nabla_{x'}h_{2}|^{2}\notag\\
\leq&40\kappa_{3}^{2}|x'|^{2m-2}\leq40\kappa_{1}^{2/m-2}\kappa_{3}^{2}\delta^{2-2/m}\leq90\kappa_{1}^{2/m-2}\kappa_{3}^{2}\delta_{0}^{2-2/m}\notag\\
\leq&51840\kappa_{1}^{2/m-2}\kappa_{3}^{2}r^{2}\delta_{0}^{-2/m}<\frac{1}{4},
\end{align}
and, there exist two points $\xi_{i}'$, $i=1,2$ between $x_{0}'$ and $x'$ such that
\begin{align}
|\tilde{a}_{dd}-1|=&|\delta-\delta_{0}|\delta^{-1}\leq\kappa_{3}\delta^{-1}(|\nabla_{x'}h_{1}(\xi'_{1})|+|\nabla_{x'}h_{2}(\xi_{2}')|)|x'-x_{0}'|\notag\\
\leq&2^{m-2}\kappa_{3}r\delta^{-1}(|\xi'_{1}-x'_{0}|^{m-1}+|\xi'_{2}-x'_{0}|^{m-1}+2|x_{0}'|^{m-1})\notag\\
\leq&2^{m}\kappa_{3}r\delta_{0}^{-1}(r^{m-1}+|x_{0}'|^{m-1})\leq2^{m}\kappa_{3}(\tilde{c}_{0}^{m-1}+\kappa_{1}^{1/m-1})r\delta_{0}^{-1/m}\notag\\
\leq&2^{m+1}\kappa_{1}^{1/m-1}\kappa_{3}r\delta_{0}^{-1/m}<\frac{1}{2},\notag\\
|(\det(\tilde{A}))^{-1}-1|=&|\tilde{a}_{dd}^{-1}-1|=|\delta-\delta_{0}|\delta_{0}^{-1}\leq2^{m}\kappa_{1}^{1/m-1}\kappa_{3}r\delta_{0}^{-1/m}<\frac{1}{4},\label{FQ015}
\end{align}
and thus
\begin{align}\label{EQM896}
|\tilde{A}(\tilde{z})-I_{d}|=&\sqrt{\sum^{d-1}_{j=1}\tilde{a}_{dj}^{2}+(\tilde{a}_{dd}-1)^{2}}\notag\\
\leq&\sqrt{51840+4^{m+1}}\kappa_{1}^{1/m-1}\kappa_{3}r\delta_{0}^{-1/m}\leq\frac{1}{2}.
\end{align}
We now perform the even extension of $\tilde{u}_{1}$, $\tilde{a}_{dd}$ and $\tilde{a}_{ij}$, $i,j=1,...,d-1$ and the odd extension of $\tilde{a}_{id}$ and $\tilde{a}_{di}$, $i=1,...,d-1$ with respect to the hyperplane $\tilde{z}_{d}=\frac{1}{2}\delta_{0}$, respectively. Then we carry out the periodic extension for them in $\tilde{z}_{d}$ axis with the period $2\delta_{0}.$ For brevity, we still denote the extended function and coefficient matrix by $\tilde{u}_{1}$ and $\tilde{A}$. Hence by the conormal boundary condition, we see that $\tilde{u}_{1}$ solves
\begin{align}\label{NP066}
-\mathrm{div}_{\tilde{z}}(\boldsymbol{\mathcal{\tilde{A}}}(\tilde{z},\nabla_{\tilde{z}}\tilde{u}_{1}))=0,\quad\mathrm{in}\;Q_{R_{0},\infty},
\end{align}
where $\boldsymbol{\mathcal{\tilde{A}}}$ is a nonlinear operate defined by $\boldsymbol{\mathcal{\tilde{A}}}(\tilde{z},\xi)=(\det(\tilde{A}))^{-1}|\tilde{A}^{T}\xi|^{p-2}\tilde{A}\tilde{A}^{T}\xi$ with $\tilde{z}\in Q_{R_{0},\infty}$ and $\xi\in\mathbb{R}^{d}$. Similarly as in Lemma \ref{LEM030}, it follows from \eqref{QWN001}, \eqref{WME01} and \eqref{FQ013}--\eqref{EQM896} that for any $r\in[\frac{1}{24}\delta_{0},\tilde{c}_{0}\delta_{0}^{1/m}]$, $\tilde{z}\in B_{r}$, and $\xi\in\mathbb{R}^{d}$,
\begin{align*}
|\tilde{A}|=\sqrt{d-1+(\delta_{0}\delta^{-1})^{2}+\sum^{d-1}_{j=1}|\tilde{a}_{dj}|^{2}}\leq\sqrt{d+\frac{13}{4}},
\end{align*}
and then
\begin{align}\label{DM683}
&|\boldsymbol{\mathcal{\tilde{A}}}(\tilde{z},\xi)-|\xi|^{p-2}\xi|\notag\\
&\leq|((\det(\tilde{A}))^{-1}-1)|\tilde{A}^{T}\xi|^{p-2}\tilde{A}\tilde{A}^{T}\xi|+||\tilde{A}^{T}\xi|^{p-2}(\tilde{A}-I_{d})A^{T}\xi|\notag\\
&\quad+||\tilde{A}^{T}\xi|^{p-2}\tilde{A}^{T}\xi-|\xi|^{p-2}\xi|\notag\\
&\leq|(\det(\tilde{A}))^{-1}-1||\tilde{A}|^{p}|\xi|^{p-1}+\left(|\tilde{A}|^{p-1}+\bar{c}_{p}(|\tilde{A}|^{2}+1)^{\frac{p-2}{2}}\right)|\tilde{A}-I_{d}||\xi|^{p-1}\notag\\
&\leq\widetilde{\mathcal{M}}r\delta_{0}^{-1/m}|\xi|^{p-1},
\end{align}
where $\bar{c}_{p}$ is defined by \eqref{UP01} and $\widetilde{\mathcal{M}}:=\widetilde{\mathcal{M}}(d,p,m,\kappa_{1},\kappa_{3})$ is given by
\begin{align}\label{A908}
\widetilde{\mathcal{M}}=&2^{m}\Big(d+\frac{13}{4}\Big)^{\frac{p}{2}}\kappa_{1}^{1/m-1}\kappa_{3}\notag\\
&+\sqrt{51840+4^{m+1}}\left(\Big(d+\frac{13}{4}\Big)^{\frac{p-1}{2}}+\bar{c}_{p}\Big(d+\frac{17}{4}\Big)^{\frac{p-2}{2}}\right)\kappa_{1}^{1/m-1}\kappa_{3}.
\end{align}
Let $\tilde{v}_{1}$ be the unique solution to the following Dirichlet problem:
\begin{align*}
\begin{cases}
-\mathrm{div}_{\tilde{z}}(|\nabla_{\tilde{z}}\tilde{v}_{1}|^{p-2}\nabla_{\tilde{z}}\tilde{v}_{1})=0,&\mathrm{in}\;B_{r},\\
\tilde{v}_{1}=\tilde{u}_{1},&\mathrm{on}\;\partial B_{r}.
\end{cases}
\end{align*}
Making use of \eqref{DM683}, it follows from the proof of Lemma \ref{Lem005} with minor modification that if $0<R_{0}\leq \min\{R_{0,1},R_{0,2},\overline{R}_{0}\}$, then for any $r\in[\frac{1}{24}\delta_{0},\tilde{c}_{0}\delta_{0}^{1/m}]$,
\begin{align}\label{MFD095}
\dashint_{B_{r}}|\nabla_{\tilde{z}}\tilde{u}_{1}-\nabla_{\tilde{z}}\tilde{v}_{1}|^{p}\leq\widetilde{\mathfrak{S}}_{p}(r\delta_{0}^{-1/m})^{\min\{2,p\}}\dashint_{B_{r}}|\nabla_{\tilde{z}}\tilde{u}_{1}|^{p},
\end{align}
where
\begin{align*}
\widetilde{\mathfrak{S}}_{p}=
\begin{cases}
\frac{\bar{c}^{p}_{\frac{p-2}{2}}2^{\frac{p(2-p)}{2}}\widetilde{\mathcal{M}}^{p}}{\underline{c}^{p}_{\frac{p-2}{2}}(p-1)^{p}}\left[\left(\frac{2^{p}(p-1)^{p-1}}{p^{p}}\right)^{\frac{2-p}{2}}+1\right],&1<p<2,\\
4\max\{1,2^{p-3}\}\widetilde{\mathcal{M}}^{2},&p\geq2
\end{cases}
\end{align*}
with $\bar{c}_{\frac{p-2}{2}},\,\underline{c}_{\frac{p-2}{2}},\,\widetilde{\mathcal{M}}$ defined by \eqref{INEC01} and \eqref{A908}. For $r\in[\frac{1}{24}\delta_{0},\tilde{c}_{0}\delta_{0}^{1/m}]$, write
\begin{align*}
\tilde{\psi}(r)=\left(\dashint_{B_{r}}|\nabla_{\tilde{z}}\tilde{u}_{1}-(\nabla_{\tilde{z}}\tilde{u}_{1})_{B_{r}}|^{p}\right)^{\frac{1}{p}}.
\end{align*}
Therefore, following the same proof of \eqref{FEA90} except replacing \eqref{M09} with \eqref{MFD095}, we obtain that for $\mu\in(0,1)$ and $r\in[\frac{1}{24}\delta_{0},\tilde{c}_{0}\delta_{0}^{1/m}]$,
\begin{align*}
&\tilde{\psi}(\mu r)\leq \widetilde{C}_{1}\mu^{\alpha}\tilde{\psi}(r)+\widetilde{C}_{2}\mu^{-\frac{d}{p}}(r\delta_{0}^{-1/m})^{\theta}\left(\dashint_{B_{r}}|\nabla_{\tilde{z}}\tilde{u}_{1}|^{p}\right)^{\frac{1}{p}},
\end{align*}
where $\widetilde{C}_{1}=\widetilde{C}_{1}(d,p)$, $\widetilde{C}_{2}=\widetilde{C}_{2}(d,p,m,\kappa_{1},\kappa_{3})$, $\theta=\min\{1,2/p\}$, and $\alpha$ is given by Lemma \ref{Lem001}. Denote
\begin{align*}
\mu_{3}=\min\{(2\widetilde{C}_{1})^{-\frac{1}{\alpha}},1\}.
\end{align*}
Then similar to Lemma \ref{Lem096}, we have
\begin{lemma}\label{Lem659}
Assume as above. Let $\tilde{u}_{1}\in W^{1,p}_{loc}(Q_{R_{0},\infty})$ be the solution to \eqref{NP066}. If $0<R_{0}\leq \min\{R_{0,1},R_{0,2},\overline{R}_{0}\}$ with $R_{0,1},\,R_{0,2}$ and $\overline{R}_{0}$ given by \eqref{C09}--\eqref{C12} and \eqref{WE860}, then for any $r\in[\frac{1}{24}\delta_{0},\tilde{c}_{0}\delta_{0}^{1/m}]$, $\mu\in(0,\mu_{3})$, and $0\leq i_{0}<j$,
\begin{align*}
&\sum^{j}_{k=i_{0}}\tilde{\psi}(\mu^{k}r)\notag\\
&\leq2\tilde{\psi}(\mu^{i_{0}}r)+\widetilde{C}_{2}\mu^{-\frac{d}{p}}\sum^{j}_{k=i_{0}+1}\sum^{k-1}_{i=i_{0}}(\widetilde{C}_{1}\mu^{\alpha})^{k-1-i}(\mu^{i}r\delta_{0}^{-1/m})^{\theta}\left(\dashint_{B_{\mu^{i}r}}|\nabla_{\tilde{z}}\tilde{u}_{1}|^{p}\right)^{\frac{1}{p}}.
\end{align*}
\end{lemma}

We now give the proof of Theorem \ref{thm001} by combining Lemmas \ref{Lem001}, \ref{Lem096} and \ref{Lem659}.
\begin{proof}[Proof of Theorem \ref{thm001}]
For any fixed $x_{0}\in\Omega_{R_{0}}$, $0<R_{0}\leq\min\{\overline{R}_{0},R_{0,1},R_{0,2}\}$, and $0<\mu\leq\frac{2}{3}\min\limits_{1\leq i\leq3}\mu_{i}\leq\frac{1}{6}$, denote $r_{j}=\frac{\tilde{c}_{0}}{13}\mu^{j}\delta_{0}^{1/m}$ for $j\geq0$, where $\tilde{c}_{0}$ and $\delta_{0}$ are defined by \eqref{C681} and \eqref{DELTA02}, and $\mu_{i}$, $i=1,2,3$ are, respectively, given by Lemmas \ref{Lem001}, \ref{Lem096} and \ref{Lem659}. Pick two nonnegative integers $j_{1}$ and $j_{2}$ such that
\begin{align*}
r_{j_{1}}\geq\frac{1}{24}\delta_{0},\quad r_{j_{1}+1}<\frac{1}{24}\delta_{0},\quad r_{j_{2}}\geq\mathrm{dist}(x_{0},\Gamma_{2R_{0}}^{\pm}),\quad r_{j_{2}+1}<\mathrm{dist}(x_{0},\Gamma_{2R_{0}}^{\pm}).
\end{align*}
We here explain the implication of indices $j_{1}$ and $j_{2}$. To be specific, when $j\geq j_{1}+1$, $B_{r_{j}}(x_{0})$ intersects at most one of $\Gamma_{2R_{0}}^{\pm}$; when $j\geq j_{2}+1$, $B_{r_{j}}(x_{0})$ is located in the interior of $\Omega_{2R_{0}}$; when $j\leq j_{1}$ or $j\leq j_{2}$, there may appear the third case that $B_{r_{j}}(x_{0})$ intersects both $\Gamma_{2R_{0}}^{+}$ and $\Gamma_{2R_{0}}^{-}.$ Note that since $0<R_{0}\leq\overline{R}_{0},$ we have $\frac{1}{24}\delta_{0}\leq\frac{\tilde{c}_{0}}{13}\delta_{0}^{1/m}$, where $\overline{R}_{0}$ is given by \eqref{WE860}.

For $j\geq0$, define
\begin{align*}
\phi_{j}=
\begin{cases}
\left(\displaystyle\dashint_{B_{13r_{j}}}|\nabla_{\tilde{z}}\tilde{u}_{1}-(\nabla_{\tilde{z}}\tilde{u}_{1})_{B_{13r_{j}}}|^{p}\right)^{\frac{1}{p}},&\mathrm{for}\;0\leq j\leq j_{1},\\
\left(\displaystyle\dashint_{\Omega_{r_{j}}(x_{0})}|\nabla u-(\nabla u)_{\Omega_{r_{j}}(x_{0})}|^{p}\right)^{\frac{1}{p}},&\mathrm{for}\;j\geq j_{1}+1,
\end{cases}
\end{align*}
and
\begin{align*}
T_{j}=
\begin{cases}
\left(\displaystyle\dashint_{B_{13r_{j}}}|\nabla_{\tilde{z}}\tilde{u}_{1}|^{p}\right)^{\frac{1}{p}},&\mathrm{for}\;0\leq j\leq j_{1},\\
\left(\displaystyle\dashint_{\Omega_{3 r_{j}}(x_{0})}|\nabla u|^{p}\right)^{\frac{1}{p}},&\mathrm{for}\;j\geq j_{1}+1,
\end{cases}
\end{align*}
and
\begin{align*}
\omega_{j}=
\begin{cases}
(\nabla_{\tilde{z}}\tilde{u}_{1})_{B_{13r_{j}}},&\mathrm{for}\;0\leq j\leq j_{1},\\
(\nabla u)_{\Omega_{r_{j}}(x_{0})},&\mathrm{for}\;j\geq j_{1}+1.
\end{cases}
\end{align*}

{\bf Step 1.} Since $0<\mu\leq\frac{1}{6}$, we have $\Omega_{3r_{j+1}}(x_{0})\subset\Omega_{\frac{1}{2}r_{j}}(x_{0})$ for $j\geq0$. Note that $3r_{j_{1}+1}<\frac{1}{8}\delta_{0}$, it then follows from \eqref{PER02} that
\begin{align}\label{WZQA9062}
T_{j_{1}+1}\leq&\bigg(\frac{|\Omega_{\frac{1}{2}r_{j_{1}}}(x_{0})|}{|\Omega_{3r_{j_{1}+1}}(x_{0})|}\dashint_{\Omega_{\frac{1}{2}r_{j_{1}}}(x_{0})}|\nabla u|^{p}\bigg)^{\frac{1}{p}}\notag\\
\leq& C(d)\mu^{-\frac{d}{p}}\bigg(\dashint_{\Omega_{\frac{1}{2}r_{j_{1}}}(x_{0})}|\nabla u|^{p}\bigg)^{\frac{1}{p}}.
\end{align}
In view of $r_{j_{1}}\geq\frac{1}{24}\delta_{0}$, we have $Q_{\frac{1}{2}r_{j_{1}},\frac{1}{2}\delta_{0}}\subset B_{13r_{j_{1}}}$. Then we obtain that $\tilde{\Lambda}(\Omega_{\frac{1}{2}r_{j_{1}}}(x_{0}))\subset\tilde{\Lambda}(\Omega_{x_{0}',\frac{1}{2}r_{j_{1}}})=Q_{\frac{1}{2}r_{j_{1}},\frac{1}{2}\delta_{0}}\subset B_{13r_{j_{1}}}.$ This, together with \eqref{PER02}, leads to that if $0<R_{0}\leq\min\{\overline{R}_{0},R_{0,1},R_{0,2}\}$,
\begin{align*}
\Bigg(\dashint_{\Omega_{\frac{1}{2}r_{j_{1}}}(x_{0})}|\nabla u|^{p}\Bigg)^{\frac{1}{p}}\leq&C(d)\left(\frac{|B_{13r_{j_{1}}}|}{|\Omega_{\frac{1}{2}r_{j_{1}}}(x_{0})|}\dashint_{B_{13r_{j_{1}}}}|\nabla_{\tilde{z}}\tilde{u}_{1}|^{p}\right)^{\frac{1}{p}}\leq C(d)T_{j_{1}}.
\end{align*}
Inserting this into \eqref{WZQA9062}, we have $T_{j_{1}+1}\leq C(d)\mu^{-\frac{d}{p}}T_{j_{1}}.$ Hence we further obtain that for any $j\geq0,$
\begin{align*}
T_{j+1}\leq& C(d)\mu^{-\frac{d}{p}}T_{j},
\end{align*}
which, in combination with the triangle inequality, reads that for any $j\leq j_{1},$
\begin{align*}
T_{j+1}\leq& C(d)\mu^{-\frac{d}{p}}T_{j}\leq C(d)\mu^{-\frac{d}{p}}(|\omega_{j}|+\phi_{j}).
\end{align*}
Observe that for $j\geq0,$ there has $\Omega_{3r_{j+1}}(x_{0})\subset\Omega_{r_{j}}(x_{0})$ in virtue of $0<\mu\leq\frac{1}{6}$. It then follows from \eqref{PER02} and the triangle inequality again that for any $j\geq j_{1}+1$,
\begin{align*}
T_{j+1}\leq& C(d)\mu^{-\frac{d}{p}}\left(\dashint_{\Omega_{r_{j}}(x_{0})}|\nabla u|^{p}\right)^{\frac{1}{p}}\leq C(d)\mu^{-\frac{d}{p}}(|\omega_{j}|+\phi_{j}).
\end{align*}
Consequently, we have
\begin{align}\label{DEA678}
T_{j+1}\leq& C(d)\mu^{-\frac{d}{p}}(|\omega_{j}|+\phi_{j}),\quad\text{for any}\;j\geq0.
\end{align}

Note that for $0\leq k\leq j_{1}$,
\begin{align*}
|\omega_{k}-\omega_{k-1}|^{p}\leq2^{p-1}(|\omega_{k}-\nabla_{\tilde{z}}\tilde{u}_{1}(\tilde{z})|^{p}+|\nabla_{\tilde{z}}\tilde{u}_{1}(\tilde{z})-\omega_{k-1}|^{p}).
\end{align*}
Integrating it in $\tilde{z}$ on $B_{13r_{k}}$ and taking the $p$-th root, we obtain
\begin{align*}
|\omega_{k}-\omega_{k-1}|\leq2^{\frac{p-1}{p}}(\phi_{k}+\mu^{-\frac{d}{p}}\phi_{k-1}),
\end{align*}
which, together with H\"{o}lder's inequality, gives that
\begin{align}\label{AD9150}
|\omega_{j}|\leq
\begin{cases}
T_{j_{0}},&\text{if }0\leq j=j_{0}\leq j_{1},\\
T_{j_{0}}+4\mu^{-\frac{d}{p}}\sum^{j}_{k=j_{0}}\phi_{k},&\text{if }0\leq j_{0}<j\leq j_{1},
\end{cases}
\end{align}
In exactly the same way, we have from \eqref{PER02} that
\begin{align}\label{AD9153}
|\omega_{j}|\leq
\begin{cases}
C(d)T_{j},&\text{if }j=l\geq j_{1}+1,\\
C(d)(T_{l}+\mu^{-\frac{d}{p}}\sum^{j}_{k=l}\phi_{k}),&\text{if }j>l\geq j_{1}+1,
\end{cases}
\end{align}
When $0\leq j_{0}\leq j\leq j_{1}$, it follows from Lemma \ref{Lem659}, \eqref{AD9150}, the triangle inequality and H\"{o}lder's inequality that

$(i)$ if $j=j_{0}$,
\begin{align}\label{D961}
|\omega_{j_{0}}|+\phi_{j_{0}}\leq3T_{j_{0}};
\end{align}

$(ii)$ if $j>j_{0}$,
\begin{align}\label{D980}
&|\omega_{j}|+\sum^{j}_{k=j_{0}}\phi_{k}\leq T_{j_{0}}+5\mu^{-\frac{d}{p}}\sum^{j}_{k=j_{0}}\phi_{k}\notag\\
&\leq T_{j_{0}}+10\mu^{-\frac{d}{p}}\phi_{j_{0}}+\widetilde{C}_{2}\mu^{-\frac{d}{p}}\sum^{j}_{k=j_{0}+1}\sum^{k-1}_{i=j_{0}}(\widetilde{C}_{1}\mu^{\alpha})^{k-1-i}(\mu^{i}\tilde{c}_{0})^{\theta}\left(\dashint_{B_{13r_{i}}}|\nabla_{\tilde{z}}\tilde{u}_{1}|^{p}\right)^{\frac{1}{p}}\notag\\
&\leq(1+20\mu^{-\frac{d}{p}})T_{j_{0}}+\widetilde{C}_{2}\mu^{-\frac{d}{p}}\sum^{j}_{k=j_{0}+1}\sum^{k-1}_{i=j_{0}}(\widetilde{C}_{1}\mu^{\alpha})^{k-1-i}\mu^{i\theta}T_{i},
\end{align}
where $\widetilde{C}_{1}=\widetilde{C}_{1}(d,p)$ and $\widetilde{C}_{2}=\widetilde{C}_{2}(d,p,m,\kappa_{1},\kappa_{3})$.

We now proceed to consider the case when $j\geq j_{1}+1$. Since there may appear no intersection point or only one intersection point between $B_{r_{j}}(x_{0})$ and $\Gamma_{2R_{0}}^{\pm}$ in this case, we need to compare the values of $j_{1}$ and $j_{2}$ for the purpose of establishing the similar estimates as in \eqref{D961}--\eqref{D980}. Therefore, we divide into two subcases to discuss as follows.

{\bf Case 1.} Consider the case when $j_{1}<j_{2}$. First, let $j_{1}+1\leq j\leq j_{2}$. Then we have $r_{j}\geq\mathrm{dist}(x_{0},\Gamma_{2R_{0}}^{\pm})$. Pick $\hat{x}_{0}\in\Gamma_{2R_{0}}^{\pm}$ such that $\mathrm{dist}(x_{0},\Gamma_{2R_{0}}^{\pm})=|x_{0}-\hat{x}_{0}|$. Then we have $\Omega_{r_{j}}(x_{0})\subset\Omega_{2r_{j}}(\hat{x}_{0})\subset\Omega_{3r_{j}}(x_{0})$ and $2r_{j}<\frac{1}{8}\delta_{0}$. Therefore, it follows from \eqref{PER02} and Lemma \ref{Lem096} that

$(i)$ if $j=j_{1}+1$,
\begin{align}\label{QW905}
\phi_{j_{1}+1}\leq&C(d)\psi(\hat{x}_{0},2r_{j_{1}+1})\leq C(d)Y_{j_{1}+1}\leq C(d)T_{j_{1}+1};
\end{align}

$(ii)$ if $j>j_{1}+1$,
\begin{align}\label{QW96}
\sum^{j}_{k=j_{1}+1}\phi_{k}\leq&C(d)\sum^{j}_{k=j_{1}+1}\psi(\hat{x}_{0},2r_{k})\notag\\
\leq&C(d)Y_{j_{1}+1}+C_{2}\mu^{-\frac{d}{p}}\sum^{j}_{k=j_{1}+2}\sum^{k-1}_{i=j_{1}+1}(C_{1}\mu^{\alpha})^{k-1-i}\mu^{i\theta}Y_{i}\notag\\
\leq&C(d)T_{j_{1}+1}+C_{2}\mu^{-\frac{d}{p}}\sum^{j}_{k=j_{1}+2}\sum^{k-1}_{i=j_{1}+1}(C_{1}\mu^{\alpha})^{k-1-i}\mu^{i\theta}T_{i},
\end{align}
where $C_{1}=C_{1}(d,p)$, $C_{2}=C_{2}(d,p,m,\kappa_{1},\kappa_{3},\kappa_{4})$, and
\begin{align*}
Y_{j}:=\left(\dashint_{\Omega_{2r_{j}}(\hat{x}_{0})}|\nabla u|^{p}\right)^{\frac{1}{p}},\quad\text{for }j_{1}+1\leq j\leq j_{2}.
\end{align*}
Then combining \eqref{DEA678}, \eqref{AD9153} and \eqref{QW905}--\eqref{QW96}, we obtain that for $j_{1}+1\leq j\leq j_{2}$,

$(i)$ if $j=j_{1}+1$,
\begin{align}\label{M9061}
&|\omega_{j_{1}+1}|+\phi_{j_{1}+1}\leq C(d)T_{j_{1}+1};
\end{align}

$(ii)$ if $j>j_{1}+1$,
\begin{align}\label{M96}
&|\omega_{j}|+\sum^{j}_{k=j_{1}+1}\phi_{k}\leq C(d)\bigg(T_{j_{1}+1}+\mu^{-\frac{d}{p}}\sum^{j}_{k=j_{1}+1}\phi_{k}\bigg)\notag\\
&\leq C(d)\mu^{-\frac{d}{p}}T_{j_{1}+1}+C_{2}\mu^{-\frac{2d}{p}}\sum^{j}_{k=j_{1}+2}\sum^{k-1}_{i=j_{1}+1}(C_{1}\mu^{\alpha})^{k-1-i}\mu^{i\theta}T_{i}\notag\\
&\leq\mu^{-\frac{2d}{p}}\bigg(C(d)(|\omega_{j_{1}}|+\phi_{j_{1}})+C_{2}\sum^{j}_{k=j_{1}+2}\sum^{k-1}_{i=j_{1}+1}(C_{1}\mu^{\alpha})^{k-1-i}\mu^{i\theta}T_{i}\bigg).
\end{align}

On the other hand, for $j\geq j_{2}+1$, we deduce from Lemma \ref{Lem001}, \eqref{DEA678} and \eqref{AD9153} that
\begin{align*}
|\omega_{j}|+\sum^{j}_{k=j_{2}+1}\phi_{k}\leq& C(d)\bigg(T_{j_{2}+1}+\mu^{-\frac{d}{p}}\sum^{j}_{k=j_{2}+1}\phi_{k}\bigg)\leq C(d)\mu^{-\frac{d}{p}}T_{j_{2}+1}\notag\\
\leq&C(d)\mu^{-\frac{2d}{p}}(|\omega_{j_{2}}|+\phi_{j_{2}}).
\end{align*}
This, in combination with \eqref{M96}, yields that for any $j\geq j_{2}+1>j_{1}+1$,
\begin{align}\label{QAE9600}
&|\omega_{j}|+\sum^{j}_{k=j_{1}+1}\phi_{k}=|\omega_{j}|+\sum^{j}_{k=j_{2}+1}\phi_{k}+\sum^{j_{2}}_{k=j_{1}+1}\phi_{k}\notag\\
&\leq C(d)\mu^{-\frac{2d}{p}}(|\omega_{j_{2}}|+\phi_{j_{2}})+\sum^{j_{2}}_{k=j_{1}+1}\phi_{k}\notag\\
&\leq\mu^{-\frac{4d}{p}}\bigg(C(d)(|\omega_{j_{1}}|+\phi_{j_{1}})+C_{2}\sum^{j}_{k=j_{1}+2}\sum^{k-1}_{i=j_{1}+1}(C_{1}\mu^{\alpha})^{k-1-i}\mu^{i\theta}T_{i}\bigg),
\end{align}
which, together with \eqref{M96} again, shows that \eqref{QAE9600} holds for any $j>j_{1}+1$.

{\bf Case 2.} Consider the case when $j_{1}\geq j_{2}$. Then for any $j\geq j_{1}+1$, we have $r_{j}\leq\mathrm{dist}(x_{0},\Gamma_{2R_{0}}^{\pm})$. Hence, using Lemma \ref{Lem001}, \eqref{DEA678} and \eqref{AD9153}, we derive
\begin{align}\label{QAE9601}
|\omega_{j}|+\sum^{j}_{k=j_{1}+1}\phi_{k}\leq&C(d)\mu^{-\frac{2d}{p}}(|\omega_{j_{1}}|+\phi_{j_{1}}).
\end{align}

A combination of \eqref{D961}--\eqref{D980}, \eqref{M9061} and \eqref{QAE9600}--\eqref{QAE9601} gives that for any $0\leq j_{0}\leq j_{1}$ and $j\geq j_{0}$,

$(i)$ if $j=j_{0}$,
\begin{align}\label{QAE983}
|\omega_{j_{0}}|+\phi_{j_{0}}\leq3T_{j_{0}};
\end{align}

$(ii)$ if $j>j_{0}$,
\begin{align}\label{QAE9603}
&|\omega_{j}|+\sum^{j}_{k=j_{0}}\phi_{k}\leq\mu^{-\frac{5d}{p}}\bigg(C(d)T_{j_{0}}+\overline{C}_{2}\sum^{j}_{k=j_{0}+1}\sum^{k-1}_{i=j_{0}}(\overline{C}_{1}\mu^{\alpha})^{k-1-i}\mu^{i\theta}T_{i}\bigg),
\end{align}
where $\overline{C}_{1}=\overline{C}_{1}(d,p)$ and $\overline{C}_{2}=\overline{C}_{2}(d,p,m,\kappa_{1},\kappa_{3},\kappa_{4})$.

{\bf Step 2.} Choose an integer $j_{0}$ such that $0\leq j_{0}\leq j_{1}$, whose value will be determined later. We first calculate the series $\sum\limits^{\infty}_{k=j_{0}+1}\sum\limits^{k-1}_{i=j_{0}}(\overline{C}_{1}\mu^{\alpha})^{k-1-i}\mu^{i\theta}$, where $\overline{C}_{1}$ is given by \eqref{QAE9603}. On one hand, if $\theta\neq\alpha$, then for any $j\geq j_{0}+1$,
\begin{align*}
&\sum^{j}_{k=j_{0}+1}\sum^{k-1}_{i=j_{0}}(\overline{C}_{1}\mu^{\alpha})^{k-1-i}\mu^{i\theta}=\sum^{j}_{k=j_{0}+1}(\overline{C}_{1}\mu^{\alpha})^{k-1}\sum_{i=j_{0}}^{k-1}(\overline{C}_{1}^{-1}\mu^{\theta-\alpha})^{i}\notag\\
&=\frac{(\overline{C}_{1}^{-1}\mu^{\theta-\alpha})^{j_{0}}}{1-\overline{C}_{1}^{-1}\mu^{\theta-\alpha}}\sum^{j}_{k=j_{0}+1}(\overline{C}_{1}\mu^{\alpha})^{k-1}-\frac{1}{\overline{C}_{1}\mu^{\alpha}(1-\overline{C}_{1}^{-1}\mu^{\theta-\alpha})}\sum^{j}_{k=j_{0}+1}\mu^{k\theta}\notag\\
&=\frac{1}{1-\overline{C}_{1}^{-1}\mu^{\theta-\alpha}}\left(\frac{\mu^{j_{0}\theta}(1-(\overline{C}_{1}\mu^{\alpha})^{j-i_{0}})}{1-\overline{C}_{1}\mu^{\alpha}}-\frac{\mu^{(j_{0}+1)\theta-\alpha}(1-\mu^{(j-j_{0})\theta})}{\overline{C}_{1}(1-\mu^{\theta})}\right),
\end{align*}
where we require $\mu\neq\overline{C}_{1}^{\frac{1}{\theta-\alpha}}$ and $\mu\neq\overline{C}_{1}^{-\frac{1}{\alpha}}$ by decreasing $\mu$ if necessary. By sending $j\rightarrow\infty$, we obtain that if $\mu\leq(2\overline{C}_{1})^{-\frac{1}{\alpha}}$,
\begin{align*}
&\sum^{\infty}_{k=j_{0}+1}\sum^{k-1}_{i=j_{0}}(\overline{C}_{1}\mu^{\alpha})^{k-1-i}\mu^{i\theta}=\frac{1}{1-\overline{C}_{1}^{-1}\mu^{\theta-\alpha}}\left(\frac{\mu^{j_{0}\theta}}{1-\overline{C}_{1}\mu^{\alpha}}-\frac{\mu^{(j_{0}+1)\theta-\alpha}}{\overline{C}_{1}(1-\mu^{\theta})}\right)\notag\\
&=\frac{\mu^{j_{0}\theta}}{(1-\overline{C}_{1}\mu^{\alpha})(1-\mu^{\theta})}\leq\frac{2\mu^{j_{0}\theta}}{1-6^{-\theta}},
\end{align*}
where we also used the fact that $0<\mu\leq\frac{1}{6}.$

On the other hand, if $\theta=\alpha$, then for any $j\geq j_{0}+1$,
\begin{align*}
&\sum^{j}_{k=j_{0}+1}\sum^{k-1}_{i=j_{0}}(\overline{C}_{1}\mu^{\alpha})^{k-1-i}\mu^{i\theta}=\sum^{j}_{k=j_{0}+1}(\overline{C}_{1}\mu^{\alpha})^{k-1}\sum^{k-1}_{i=j_{0}}\overline{C}_{1}^{-i}\notag\\
&\leq\frac{1}{\overline{C}_{1}^{j_{0}-1}(\overline{C}_{1}-1)}\sum^{j}_{k=j_{0}+1}(\overline{C}_{1}\mu^{\alpha})^{k-1}=\frac{(\overline{C}_{1}\mu^{\alpha})^{j_{0}}(1-(\overline{C}_{1}\mu^{\alpha})^{j-j_{0}})}{\overline{C}_{1}^{j_{0}-1}(\overline{C}_{1}-1)(1-\overline{C}_{1}\mu^{\alpha})},
\end{align*}
which reads that if $\mu\leq(2\overline{C}_{1})^{-\frac{1}{\alpha}}$,
\begin{align*}
\sum^{\infty}_{k=j_{0}+1}\sum^{k-1}_{i=j_{0}}(\overline{C}_{1}\mu^{\alpha})^{k-1-i}\mu^{i\theta}\leq\frac{\overline{C}_{1}\mu^{j_{0}\theta}}{(\overline{C}_{1}-1)(1-\overline{C}_{1}\mu^{\alpha})}\leq\frac{2\overline{C}_{1}\mu^{j_{0}\theta}}{\overline{C}_{1}-1}.
\end{align*}
Denote
\begin{align*}
\mu^{\ast}:=\min\left\{\frac{2}{3}\min\limits_{1\leq i\leq3}\mu_{i},(2\overline{C}_{1})^{-\frac{1}{\alpha}},\left(\frac{1-6^{-\theta}}{4\overline{C}_{2}}\right)^{\frac{1}{j_{0}\theta-\frac{6d}{p}}},\left(\frac{\overline{C}_{1}-1}{4\overline{C}_{1}\overline{C}_{2}}\right)^{\frac{1}{j_{0}\theta-\frac{6d}{p}}}\right\}.
\end{align*}
Therefore, if $\mu$ is chosen to be $\mu_{0}$ with $\mu_{0}$ given by
\begin{align}\label{QM86}
\mu_{0}=
\begin{cases}
\min\{\mu^{\ast},\overline{C}_{1}^{\frac{1}{\theta-\alpha}}\},&\text{if }\theta\neq\alpha,\\
\mu^{\ast},&\text{if }\theta=\alpha,
\end{cases}
\end{align}
we have
\begin{align}\label{T902}
\overline{C}_{2}\mu^{-\frac{6d}{p}}\sum^{\infty}_{k=j_{0}+1}\sum^{k-1}_{i=j_{0}}(\overline{C}_{1}\mu^{\alpha})^{k-1-i}\mu^{i\theta}\leq\frac{1}{2},
\end{align}
where $\overline{C}_{i}$, $i=1,2$ are given by \eqref{T901} below. In order to ensure the validity of $\mu_{0}$ and $j_{0}\leq j_{1}$, we choose $j_{0}=[\frac{6d}{p\theta}]+1=[\frac{6d}{\min\{p,2\}}]+1>\frac{6d}{p\theta}$ and let $0<R_{0}\leq\min\limits_{1\leq i\leq3}R_{0,i}$, where $R_{0,i}$, $i=1,2,3$ are defined by \eqref{C09}--\eqref{WAD099}.

Claim that if $0<R_{0}\leq\min\limits_{1\leq i\leq3}R_{0,i}$ and $\mu=\mu_{0}$,
\begin{align}\label{AM9102}
|\nabla u(x_{0})|\leq CT_{j_{0}},
\end{align}
where $C=C(d,p,m,\kappa_{1},\kappa_{3},\kappa_{4})$. The pointwise gradient estimate in \eqref{AM9102} implies that for any fixed point $x_{0}\in\Omega_{R_{0}}$, the gradient $|\nabla u(x_{0})|$ can be controlled by the $p$-average integration of the transformed gradient $|\nabla_{\tilde{z}}\tilde{u}_{1}|$ over a ball $B_{13r_{j_{0}}}$ with the scale of radius larger than the height $\delta_{0}$ of the thin gap.

We now divide into four subcases to prove \eqref{AM9102} as follows.

{\bf Case 1.} If $|\nabla u(x_{0})|\leq T_{j_{0}}$, then \eqref{AM9102} holds.

{\bf Case 2.} If $|\nabla u(x_{0})|>T_{j_{0}}$ and $|\nabla u(x_{0})|\leq T_{j_{0}+1}$, we deduce from \eqref{DEA678} and \eqref{QAE983} that
\begin{align*}
|\nabla u(x_{0})|\leq T_{j_{0}+1}\leq C(d)\mu_{0}^{-\frac{d}{p}}(|\omega_{j_{0}}|+\phi_{j_{0}})\leq C(d)\mu_{0}^{-\frac{d}{p}}T_{j_{0}}.
\end{align*}

{\bf Case 3.} If $|\nabla u(x_{0})|>T_{j}$ for any $j_{0}<j\leq j_{3}$ and $|\nabla u(x_{0})|\leq T_{j_{3}+1}$, it then follows from \eqref{DEA678}, \eqref{QAE9603} and \eqref{T902} that
\begin{align}\label{T901}
|\nabla u(x_{0})|\leq& T_{j_{3}+1}\leq C(d)\mu^{-\frac{d}{p}}(|\omega_{j_{3}}|+\phi_{j_{3}})\notag\\
\leq&C(d)\mu_{0}^{-\frac{6d}{p}}T_{j_{0}}+\overline{C}_{2}\mu_{0}^{-\frac{6d}{p}}\sum^{j_{3}}_{k=j_{0}+1}\sum^{k-1}_{i=j_{0}}(\overline{C}_{1}\mu_{0}^{\alpha})^{k-1-i}\mu^{i\theta}|\nabla u(x_{0})|\notag\\
\leq&C(d)\mu_{0}^{-\frac{6d}{p}}T_{j_{0}}+\frac{1}{2}|\nabla u(x_{0})|,
\end{align}
where $\overline{C}_{1}=\overline{C}_{1}(d,p)$ and $\overline{C}_{2}=\overline{C}_{2}(d,p,m,\kappa_{1},\kappa_{3},\kappa_{4})$. Consequently, \eqref{AM9102} holds under this case. We additionally remark that the value of $\overline{C}_{1}$ in \eqref{T901} is consistent with that in \eqref{QAE9603}, while the value of $\overline{C}_{2}$ in \eqref{T901} is greater than that in \eqref{QAE9603}.

{\bf Case 4.} In the case when $|\nabla u(x_{0})|>T_{j}$ for any $j\geq j_{0}$, it then follows from \eqref{QAE9603} and \eqref{T902} that
\begin{align*}
|\omega_{j}|\leq C(d)\mu_{0}^{-\frac{5d}{p}}T_{j_{0}}+\frac{1}{2}|\nabla u(x_{0})|.
\end{align*}
Since $u\in C^{1}(\Omega_{2R_{0}})$, it then follows from Lebesgue-Besicovitch Differentiation Theorem (see Theorem 1.32 in \cite{EG2015}) that
\begin{align*}
|\nabla u(x_{0})|\leq C(d)\mu_{0}^{-\frac{5d}{p}}T_{j_{0}}+\frac{1}{2}|\nabla u(x_{0})|,
\end{align*}
which implies that \eqref{AM9102} holds.

{\bf Step 3.} This step aims to establish the Caccioppoli inequality for the solution $\tilde{u}_{1}$, which is used to extract the height $\delta_{0}$ from the pointwise gradient estimate \eqref{AM9102}. Fix $\lambda\in\mathbb{R}$ and choose a nonnegative smooth cutoff function $\eta$ such that $\eta=1$ in $B_{13r_{j_{0}}}$, $\eta=0$ in $\mathbb{R}^{d}\setminus B_{26 r_{j_{0}}}$, and $|\nabla\eta|\leq\frac{2}{13r_{j_{0}}}$. Since $\mu=\mu_{0}\leq\frac{1}{6}$ and $j_{0}>1$, then
\begin{align}\label{W860}
26 r_{j_{0}}=2\tilde{c}_{0}\mu_{0}^{j_{0}}\delta_{0}^{1/m}<\frac{1}{3}\tilde{c}_{0}\delta_{0}^{1/m}<R_{0}.
\end{align}
Hence, multiplying equation \eqref{NP066} by test function $\eta^{p}(\tilde{u}_{1}-\lambda)$ and integrating by parts over $B_{26r_{j_{0}}}$, we have
\begin{align*}
\int_{B_{26 r_{j_{0}}}}\eta^{p}\langle\boldsymbol{\mathcal{\tilde{A}}}(\tilde{z},\nabla_{\tilde{z}}\tilde{u}_{1}),\nabla_{\tilde{z}}\tilde{u}_{1}\rangle=-p\int_{B_{26 r_{j_{0}}}}\eta^{p-1}(\tilde{u}_{1}-\lambda)\langle\boldsymbol{\mathcal{\tilde{A}}}(\tilde{z},\nabla_{\tilde{z}}\tilde{u}_{1}),\nabla_{\tilde{z}}\tilde{u}_{1}\rangle.
\end{align*}
On one hand, it follows from \eqref{DM683} and \eqref{QM86} that
\begin{align*}
&\int_{B_{26 r_{j_{0}}}}\eta^{p}\langle\boldsymbol{\mathcal{\tilde{A}}}(\tilde{z},\nabla_{\tilde{z}}\tilde{u}_{1}),\nabla_{\tilde{z}}\tilde{u}_{1}\rangle\notag\\
&\geq\int_{B_{26 r_{j_{0}}}}\eta^{p}|\nabla_{\tilde{z}}\tilde{u}_{1}|^{p}-\int_{B_{26r_{j_{0}}}}\eta^{p}|\boldsymbol{\mathcal{\tilde{A}}}(\tilde{z},\nabla_{\tilde{z}}\tilde{u}_{1})-|\nabla_{\tilde{z}}\tilde{u}_{1}|^{p-2}\nabla_{\tilde{z}}\tilde{u}_{1}||\nabla_{\tilde{z}}\tilde{u}_{1}|\notag\\
&\geq(1-26r_{j_{0}}\widetilde{\mathcal{M}}\delta_{0}^{-1/m})\int_{B_{26r_{j_{0}}}}|\nabla_{\tilde{z}}\tilde{u}_{1}|^{p}\geq\frac{3}{4}\int_{B_{26r_{j_{0}}}}|\nabla_{\tilde{z}}\tilde{u}_{1}|^{p},
\end{align*}
where we used the fact that
\begin{align}\label{QM666}
26r_{j_{0}}\widetilde{\mathcal{M}}\delta_{0}^{-1/m}=&2\tilde{c}_{0}\widetilde{\mathcal{M}}\mu_{0}^{j_{0}}<\overline{C}_{2}\mu_{0}^{j_{0}}\leq\overline{C}_{2}\left(\frac{1-6^{-\theta}}{4\overline{C}_{2}}\right)^{\frac{j_{0}}{j_{0}\theta-\frac{6d}{p}}}<\frac{1}{4}.
\end{align}
On the other hand, combining \eqref{DM683}, \eqref{QM86}, \eqref{QM666} and Young's inequality, we obtain
\begin{align*}
&\int_{B_{26 r_{j_{0}}}}p\eta^{p-1}(\tilde{u}_{1}-\lambda)\langle\boldsymbol{\mathcal{\tilde{A}}}(\tilde{z},\nabla_{\tilde{z}}\tilde{u}_{1}),\nabla_{\tilde{z}}\tilde{u}_{1}\rangle\notag\\
&\leq\int_{B_{26r_{j_{0}}}}p\eta^{p-1}|\tilde{u}_{1}-\lambda||\nabla\eta|(|\boldsymbol{\mathcal{\tilde{A}}}(\tilde{z},\nabla_{\tilde{z}}\tilde{u}_{1})-|\nabla_{\tilde{z}}\tilde{u}_{1}|^{p-2}\nabla_{\tilde{z}}\tilde{u}_{1}|+|\nabla_{\tilde{z}}\tilde{u}_{1}|^{p-1})\notag\\
&\leq\frac{5p}{26r_{j_{0}}}\int_{B_{26r_{j_{0}}}}\eta^{p-1}|\nabla_{\tilde{z}}\tilde{u}_{1}|^{p-1}|\tilde{u}_{1}-\lambda|\notag\\
&\leq\frac{1}{4}\int_{B_{26r_{j_{0}}}}\eta^{p}|\nabla_{\tilde{z}}\tilde{u}_{1}|^{p}+\frac{C(p)}{r_{j_{0}}^{p}}\int_{B_{26r_{j_{0}}}}|\tilde{u}_{1}-\lambda|^{p}.
\end{align*}
A consequence of these above results gives the Caccioppoli inequality as follows:
\begin{align*}
\dashint_{B_{13r_{j_{0}}}}|\nabla_{\tilde{z}}\tilde{u}_{1}|^{p}\leq\frac{C(p)}{r^{p}_{j_{0}}}\int_{B_{26r_{j_{0}}}}|\tilde{u}_{1}-\lambda|^{p},
\end{align*}
which, in combination with \eqref{AM9102}, \eqref{W860} and taking $\lambda=(\tilde{u}_{1})_{B_{26r_{j_{0}}}}$, shows that if $0<R_{0}\leq\min\limits_{1\leq i\leq3}R_{0,i}$,
\begin{align*}
|\nabla u(x_{0})|\leq C\delta_{0}^{1/m}\mathop{osc}\limits_{\Omega_{x',\varrho}}u,\quad\varrho=\frac{\tilde{c}_{0}}{3}\delta_{0}^{1/m},
\end{align*}
where $C=C(d,m,p,\kappa_{1},\kappa_{3},\kappa_{4})$.

\end{proof}

\section{Improved pointwise upper bound on the gradient}\label{Sec903}

Before proving Theorem \ref{thm002}, we first briefly expound the idea and techniques developed in \cite{DYZ2023,LY202102}. The key to improving the gradient estimate obtained in Theorem \ref{thm001} lies in establishing the decay estimate for the oscillation of $u$ in small narrow regions. As shown in \cite{LY202102}, for the linear elliptic equation of divergence form, the problem is directly reduced to the establishment of the De Giorgi-Nash-Moser Harnack inequality for the solution to the transformed and extended equation for the purpose of ensuring that the constant in Harnack inequality is independent of the height of the narrow region. However, as pointed out in \cite{DYZ2023}, the situation becomes fairly different for the nonlinear $p$-Laplace equations including the singular case of $1<p<2$ and the degenerate case of $p>2$. In the nonlinear case, it needs to establish the Krylov-Safonov Harnack inequality for the solution to the normalized $p$-Laplace equation after the transform and extension, which can be actually regarded as a second-order uniformly elliptic equation of non-divergence form.
For any $\varsigma>0,$ we start by considering the following approximation equation
\begin{align}\label{APPRO}
\begin{cases}
-\mathrm{div}((\varsigma+|\nabla u_{\varsigma}|^{2})^{\frac{p-2}{2}}\nabla u_{\varsigma})=0,&\mathrm{in}\;\Omega,\\
\frac{\partial u_{\varsigma}}{\partial\nu}=0,&\mathrm{on}\;\partial D_{i},\;i=1,2,\\
u_{\varsigma}=\varphi,&\mathrm{on}\;\partial D.
\end{cases}
\end{align}
Note that $\|u_{\varsigma}\|_{C^{1,\alpha}(\Omega_{2R_{0}})}$ stays bounded independent of $\varsigma$ for some $0<\alpha<1$. Then we only need to prove Theorem \ref{thm002} for $u_{\varsigma}$. Furthermore, the weak solution $u_{\varsigma}$ of equation \eqref{APPRO} possesses better regularity than that of the original problem \eqref{con002} and is in the class $C^{2}$, see \cite{D1983}. This means that $u_{\varsigma}$ will be the solution of normalized $p$-Laplace equation in the classical sense. For simplicity, we still denote $u=u_{\varsigma}$ in the following proof. Rewrite equation \eqref{APPRO} into the following normalized $p$-Laplace equation
\begin{align}\label{E901}
a_{ij}\partial_{ij}u=0,\quad\mathrm{in}\;\Omega_{2R_{0}},
\end{align}
where
\begin{align}\label{E902}
a_{ij}=\delta_{ij}+(p-2)(\varsigma+|\nabla u|^{2})^{-1}\partial_{i}u\partial_{j}u
\end{align}
satisfies uniformly elliptic condition
\begin{align}\label{U9}
\min\{1,p-1\}|\xi|^{2}\leq a_{ij}\xi_{i}\xi_{j}\leq\max\{1,p-1\}|\xi|^{2},\quad\forall\,\xi\in\mathbb{R}^{d},\;p>1.
\end{align}
Here, for $i,j=1,2,...,d,$ $\delta_{ij}=1$ if $i=j$, while $\delta_{ij}=0$ if $i\neq j.$

Pick $r\in(2^{-1}(\varepsilon\kappa_{2}^{-1})^{1/m},r_{0}]$, where $0<r_{0}\leq R_{0}$ is independent of $\varepsilon$ and will be determined later. For $i=1,2,$ denote
\begin{align*}
\tilde{h}_{i}(x'):=
\begin{cases}
h_{i}(x'),&\text{for }|x'|\leq2r_{0},\\
0,&\text{for }|x'|>2r_{0}.
\end{cases}
\end{align*}
For $s,t>0$, let $Q_{s,t}$ be a cylinder defined by \eqref{CI90} above. For brevity, denote
\begin{align*}
\tilde{\delta}=\tilde{\delta}(y')=\varepsilon+\tilde{h}_{1}(y')-\tilde{h}_{2}(y'),
\end{align*}
for any given $|y'|\leq2r_{0}$. For $y\in Q_{2r,r^{m}}\setminus Q_{r/4,r^{m}}$, introduce the map $x=\Phi(y)$ as follows:
\begin{align}\label{QD01}
\begin{cases}
x'=y'-g(y),\\
x_{d}=\frac{1}{2}(y_{d}r^{-m}\tilde{\delta}+\tilde{h}_{1}(y')-\tilde{h}_{2}(y')),
\end{cases}
\end{align}
where $g(y):=(g_{1}(y),...,g_{d-1}(y))$ is defined by
\begin{align*}
g(y)=&(y_{d}^{2}-r^{m})(y_{d}+r^{m})f(y),\\
f(y)=&\frac{\tilde{\delta}}{8r^{3m}}
\begin{cases}
y_{d}\nabla_{y'}(\tilde{h}_{1}^{\tau}(y')+\tilde{h}_{2}^{\tau}(y'))+r^{m}\nabla_{y'}(\tilde{h}_{1}^{\tau}(y')-\tilde{h}_{2}^{\tau}(y')),&\text{if}\;2\leq m<3,\\
y_{d}\nabla_{y'}(\tilde{h}_{1}(y')+\tilde{h}_{2}(y'))+r^{m}\nabla_{y'}(\tilde{h}_{1}(y')-\tilde{h}_{2}(y')),&\text{if}\;m\geq3.
\end{cases}
\end{align*}
Here $\tilde{h}_{i}^{\tau}$ is a mollification of $\tilde{h}_{i}$ defined by
\begin{align}\label{MO9}
\tilde{h}_{i}^{\tau}(y')=\int_{\mathbb{R}^{d-1}}\tilde{h}_{i}(y'-\tau z')\varphi(z')dz',\quad\tau=\frac{r^{2m}-y_{d}^{2}}{r^{m-1}},\;i=1,2,
\end{align}
where $\varphi$ is the standard mollifier, that is, a positive $C^{\infty}$ function with unit integration and its support contained in $B_{1}'$. The objectives for the introduction of function $g$ are two-fold: on one hand, it makes its inverse transform $\Phi^{-1}$ preserve the flattening property that $\Phi^{-1}$ maps the top and bottom boundaries of $\Omega_{2r}\setminus\Omega_{r/4}$ onto the upper and lower boundaries of $Q_{2r,r^{m}}\setminus Q_{r/4,r^{m}}$; on the other hand, it ensures that the solution of the transformed equation still satisfies the Neumann boundary condition on the upper and lower boundaries $\{y_{d}=\pm r^{m}\}$. In addition, the introduction of mollification in the case of $2\leq m<3$ is to make $\Phi$ possess second-order derivatives, which is necessary to keep uniform ellipticity of the coefficients for the transformed equation.

\begin{lemma}\label{ZWQ08}
Assume as above. Then we obtain that if $2^{-1}(\varepsilon\kappa_{2}^{-1})^{1/m}<r\leq\min\limits_{1\leq i\leq3}r_{0,i}$ with $r_{0,i},\, i=1,2,3$ defined by \eqref{Z016}--\eqref{Z006},

$(i)$ there holds
\begin{align}\label{WA608}
\begin{cases}
Q_{1.9r,r^{m}}\setminus Q_{0.35r,r^{m}}\subset\Phi^{-1}(\Omega_{2r}\setminus\Omega_{r/4}),\\
\Omega_{r}\setminus\Omega_{r/2}\subset\Phi(Q_{1.1r,r^{m}}\setminus Q_{0.4r,r^{m}}),
\end{cases}
\end{align}
and, for $y\in Q_{2r,r^{m}}\setminus Q_{r/4,r^{m}}$,
\begin{align}\label{WA908}
\frac{I_{d}}{\mathcal{C}_{0}}\leq\nabla_{y}\Phi(y)\leq\mathcal{C}_{0}I_{d},
\end{align}
where
\begin{align}\label{MA90}
\mathcal{C}_{0}=\max\{2(\min\{1,4^{-(m+1)}\kappa_{1}\})^{-1},2^{-1}\min\{1,4^{-(m+1)}\kappa_{1}\}+\max\{1,2^{m}\kappa_{2}\}\};
\end{align}

$(ii)$ if $u\in W^{1,p}(\Omega_{2r}\setminus\Omega_{r/4})$ satisfies
\begin{align}\label{EQU908}
\begin{cases}
-\mathrm{div}((\varsigma+|\nabla u|^{2})^{\frac{p-2}{2}}\nabla u)=0,&\mathrm{in}\;\Omega_{2r}\setminus\Omega_{r/4},\\
\frac{\partial u}{\partial\nu}=0,&\mathrm{on}\;\overline{\Omega_{2r}\setminus\Omega_{r/4}}\cap\Gamma^{\pm}_{2r_{0}},
\end{cases}
\end{align}
for some $\varsigma>0$, then $v(y):=u(\Phi(y))$ solves
\begin{align}\label{V90}
\begin{cases}
\tilde{a}_{ij}\partial_{ij}v(y)+\tilde{b}_{i}\partial_{i}v(y)=0,&\mathrm{in}\;Q_{1.9r,r^{m}}\setminus Q_{0.35r,r^{m}},\\
\frac{\partial v}{\partial\nu}=0,&\mathrm{on}\;\{y_{d}=\pm r^{m}\},
\end{cases}
\end{align}
with
\begin{align*}
\frac{I_{d}}{C_{1}}\leq\tilde{a}\leq C_{1}I_{d},\quad|\tilde{b}|\leq\frac{C_{2}}{r},
\end{align*}
where $I_{d}$ represents the $d\times d$ identity matrix, $C_{1}=C_{1}(d,m,p,\kappa_{1},\kappa_{2})$, and $C_{2}=C_{2}(d,m,p,\kappa_{1},\kappa_{2},\kappa_{3},\kappa_{4}).$

\end{lemma}

\begin{proof}
{\bf Step 1.} To begin with, it follows from condition ({\bf{H2}}) that for $y\in Q_{2r,r^{m}}$ and $i=1,2,$ we have $|\nabla_{y'}\tilde{h}_{i}|\leq2^{m-1}\kappa_{3}r^{m-1}$, and for $j=1,...,d-1$, if $r\leq2^{1/m}$,
\begin{align}\label{QME321}
|\partial_{y_{j}}\tilde{h}_{i}^{\tau}(y')|\leq&\int_{\mathbb{R}^{d-1}}|\partial_{y_{j}}\tilde{h}_{i}(y'-\tau z')|\varphi(z')dz'\leq\kappa_{3}(|y'|+\tau)^{m-1}\notag\\
\leq&\kappa_{3}(2r+r^{m+1})^{m-1}\leq3^{m-1}\kappa_{3}r^{m-1}.
\end{align}
Then we obtain that for $r\in\big(2^{-1}(\varepsilon\kappa_{2}^{-1})^{1/m},\min\limits_{1\leq i\leq2}r_{0,i}\big]$ and $y\in Q_{2r,r^{m}}$,

$(i)$ if $2\leq m<3$,
\begin{align*}
|g(y)|\leq2^{m-1}\kappa_{2}r^{m}\big(|\nabla_{y'}\tilde{h}_{1}^{\tau}|+|\nabla_{y'}\tilde{h}_{2}^{\tau}|\big)\leq6^{m}(d-1)\kappa_{2}\kappa_{3}r^{2m-1}<\frac{r}{10};
\end{align*}

$(ii)$ if $m\geq3$,
\begin{align*}
|g(y)|\leq2^{m-1}\kappa_{2}r^{m}\big(|\nabla_{y'}\tilde{h}_{1}|+|\nabla_{y'}\tilde{h}_{2}|\big)\leq4^{m-1}\kappa_{2}\kappa_{3}r^{2m-1}<\frac{r}{10}.
\end{align*}
Then we obtain that \eqref{WA608} holds.

Using condition ({\bf{H3}}), we obtain that for $i=1,2,$ and $j,l=1,...,d-1,$ if $r\leq\frac{2}{3}R_{0}$,
\begin{align*}
|\partial_{y_{j}y_{l}}\tilde{h}_{i}(y')|\leq\kappa_{4},\quad|\partial_{y_{j}y_{l}}\tilde{h}_{i}^{\tau}(y')|\leq&\int_{\mathbb{R}^{d-1}}|\partial_{y_{j}y_{l}}\tilde{h}_{i}(y'-\tau z')|\varphi(z')dz'\leq\kappa_{4},
\end{align*}
which, together with \eqref{QME321}, leads to that for $y\in Q_{2r,r^{m}}$ and $j,l=1,...,d-1,$ if $2^{-1}(\varepsilon\kappa_{2}^{-1})^{1/m}<r\leq\min\{2^{1/m},\frac{2}{3}R_{0}\}$,

$(i)$ for $2\leq m<3$,
\begin{align*}
|\nabla_{y_{j}}g_{l}(y)|\leq&\frac{1}{4}\bigg(|\partial_{y_{j}}\tilde{\delta}|\sum^{2}_{i=1}|\partial_{y_{l}}\tilde{h}_{i}^{\tau}|+\tilde{\delta}\sum^{2}_{i=1}|\partial_{y_{j}y_{l}}\tilde{h}^{\tau}_{i}|\bigg)\notag\\
\leq&6^{m-1}\kappa_{3}^{2}r^{2m-2}+2^{m}\kappa_{2}\kappa_{4}r^{m}\leq (6^{m}\kappa_{3}^{2}+2^{m}\kappa_{2}\kappa_{4})r^{m},
\end{align*}
and thus,
\begin{align}\label{D015}
|\nabla_{y'}g(y)|\leq(d-1)^{2}(6^{m}\kappa_{3}^{2}+2^{m}\kappa_{2}\kappa_{4})r^{m};
\end{align}

$(ii)$ for $m\geq3$,
\begin{align*}
|\nabla_{y'}g(y)|\leq&\frac{1}{4}\sqrt{\sum^{d-1}_{j,l=1}\left(|\partial_{y_{j}}\tilde{\delta}|\sum^{2}_{i=1}|\partial_{y_{l}}\tilde{h}_{i}|+\tilde{\delta}\sum^{2}_{i=1}|\partial_{y_{j}y_{l}}\tilde{h}_{i}|\right)^{2}}\notag\\
\leq&\frac{1}{2}\sqrt{|\nabla_{y'}\tilde{\delta}|^{2}\sum^{2}_{i=1}|\nabla_{y'}\tilde{h}_{i}|^{2}+\tilde{\delta}^{2}\sum^{2}_{i=1}|\nabla_{y'}^{2}\tilde{h}_{i}|^{2}}\notag\\
\leq&\frac{1}{2}\sqrt{(2^{2m-1}\kappa_{3}^{2}r^{2m-2})^{2}+(2^{m+1}\kappa_{2}\kappa_{4}r^{m})^{2}}\notag\\
\leq&2^{m}(2^{m-1}\kappa_{3}^{2}+\kappa_{2}\kappa_{4})r^{m},
\end{align*}
which, together with \eqref{D015}, reads that
\begin{align}\label{D106}
&|\nabla_{y'}x'-I_{d-1}|\notag\\
&=|\nabla_{y'}g(y)|\leq
\begin{cases}
(d-1)^{2}(6^{m}\kappa_{3}^{2}+2^{m}\kappa_{2}\kappa_{4})r^{m},&\text{if }2\leq m<3,\\
2^{m}(2^{m-1}\kappa_{3}^{2}+\kappa_{2}\kappa_{4})r^{m},&\text{if }m\geq3.
\end{cases}
\end{align}

Similarly as above, we obtain that for $y\in Q_{2r,r^{m}}$, $i=1,2,$ $l=1,...,d-1$,
\begin{align*}
|\partial_{y_{l}y_{d}}\tilde{h}_{i}^{\tau}|\leq&\int_{\mathbb{R}^{d}}\Big|\frac{\partial\tau}{\partial y_{d}}\Big||\partial_{y_{j}y_{l}}\tilde{h}_{i}(y'-\tau z')||z'|\varphi(z')dz'\leq2\kappa_{4}r.
\end{align*}
Then we obtain that for $y\in Q_{2r,r^{m}}$ and $l=1,...,d-1$, if $2^{-1}(\varepsilon\kappa_{2}^{-1})^{1/m}<r\leq\min\{2^{1/m},\frac{2}{3}R_{0}\}$,

$(i)$ for $2\leq m<3$,
\begin{align*}
|\partial_{y_{d}}g_{l}(y)|\leq&2^{m+1}\kappa_{2}\sum^{2}_{i=1}(r^{m}|\partial_{y_{d}y_{l}}\tilde{h}_{i}^{\tau}|+|\partial_{y_{l}}\tilde{h}_{i}^{\tau}|)\notag\\
\leq&2^{m+2}\kappa_{2}(3^{m-1}\kappa_{3}+4\kappa_{4})r^{m-1},
\end{align*}
and then
\begin{align}\label{D018}
|\partial_{y_{d}}g(y)|\leq2^{m+2}(d-1)\kappa_{2}(3^{m-1}\kappa_{3}+4\kappa_{4})r^{m-1};
\end{align}

$(ii)$ for $m\geq3$,
\begin{align*}
|\partial_{y_{d}}g(y)|\leq&2^{m-2}\kappa_{2}\sum^{2}_{i=1}|\nabla_{y'}\tilde{h}_{i}|\leq4^{m-1}\kappa_{2}\kappa_{3}r^{m-1}.
\end{align*}
This, in combination with \eqref{D018}, gives that
\begin{align}\label{D1095}
&|\partial_{y_{d}}x'|=|\partial_{y_{d}}g(y)|\leq
\begin{cases}
2^{m+2}(d-1)\kappa_{2}(3^{m-1}\kappa_{3}+4\kappa_{4})r^{m-1},&\text{if }2\leq m<3,\\
4^{m-1}\kappa_{2}\kappa_{3}r^{m-1},&\text{if }m\geq3.
\end{cases}
\end{align}
From conditions ({\bf{H1}})--({\bf{H2}}), we obtain that for $y\in Q_{2r,r^{m}}\setminus Q_{r/4,r^{m}}$,
\begin{align}\label{D023}
|\nabla_{y'}x_{d}|\leq\sum^{2}_{i=1}|\nabla_{y'}\tilde{h}_{i}|\leq2^{m}\kappa_{3}r^{m-1},
\end{align}
and
\begin{align}\label{D026}
\frac{\kappa_{1}}{4^{m+1}}\leq\partial_{y_{d}}x_{d}=\frac{\tilde{\delta}}{2r^{m}}\leq2^{m}\kappa_{2}.
\end{align}
Define a matrix as follows:
\begin{align}\label{MAT01}
B=\begin{pmatrix}1&0&\cdots&0&0 \\ 0&1&\cdots&0&0\\ \vdots&\vdots&\ddots&\vdots&\vdots\\0&0&\cdots&1&0\\ 0&0&\cdots&0&\frac{\tilde{\delta}}{2r^{m}}.
\end{pmatrix}
\end{align}
Observe that
\begin{align*}
|\nabla_{y}x-B|\leq|\nabla_{y'}g(y)|+|\partial_{y_{d}}g(y)|+|\nabla_{y'}x_{d}|.
\end{align*}
Then combining \eqref{D106}, \eqref{D1095} and \eqref{D023}, we deduce that for $y\in Q_{2r,r^{m}}$, if $2^{-1}(\varepsilon\kappa_{2}^{-1})^{1/m}<r\leq r_{0,1}=\min\{2^{1/m},\frac{2}{3}R_{0}\}$,
\begin{align*}
|\nabla_{y}x-B|\leq
\begin{cases}
6^{m+1}(d-1)^{2}(\kappa_{2}(\kappa_{3}+\kappa_{4})+\kappa_{3}(\kappa_{3}+1))r^{m-1},&\text{if }2\leq m<3,\\
4^{m}(\kappa_{2}(\kappa_{3}+\kappa_{4})+\kappa_{3}(\kappa_{3}+1))r^{m-1},&\text{if }m\geq3.
\end{cases}
\end{align*}
If $r$ further satisfies $2^{-1}(\varepsilon\kappa_{2}^{-1})^{1/m}<r\leq\min\limits_{1\leq i\leq2}r_{0,i}$ with $r_{0,i},\, i=1,2$ given by \eqref{Z016}--\eqref{Z018}, we have
\begin{align}\label{ZQ039}
|\nabla_{y}x-B|\leq\frac{1}{2}\min\{1,4^{-(m+1)}\kappa_{1}\},
\end{align}
which leads to that for any $\xi\in\mathbb{R}^{d}$,
\begin{align*}
\xi^{T}\nabla_{y}x\xi=&\xi^{T}B\xi+\xi^{T}(\nabla_{y}x-B)\xi\geq|\xi'|^{2}+\partial_{y_{d}}x_{d}\xi_{d}^{2}-|\nabla_{y}x-B||\xi|^{2}\notag\\
\geq&\frac{1}{2}\min\{1,4^{-(m+1)}\kappa_{1}\}|\xi|^{2},
\end{align*}
and
\begin{align*}
\xi^{T}\nabla_{y}x\xi\leq&|\xi'|^{2}+\partial_{y_{d}}x_{d}\xi_{d}^{2}+|\nabla_{y}x-B||\xi|^{2}\notag\\
\leq&(2^{-1}\min\{1,4^{-(m+1)}\kappa_{1}\}+\max\{1,2^{m}\kappa_{2}\})|\xi|^{2}.
\end{align*}
These two inequalities imply that \eqref{WA908} holds.

{\bf Step 2.} Note that $u$ is in the class $C^{2}$. Then by the chain rule, we have
\begin{align*}
\partial_{x_{k}}u(x)=\partial_{y_{i}}v(y)\partial_{x_{k}}y_{i},\quad\partial_{x_{k}x_{l}}u(x)=\partial_{y_{i}y_{j}}v(y)\partial_{x_{k}}y_{i}\partial_{x_{l}}y_{j}+\partial_{y_{i}}v(y)\partial_{x_{k}x_{l}}y_{i}.
\end{align*}
Then we have from \eqref{E901}--\eqref{E902} that $v(y)$ solves
\begin{align*}
\tilde{a}_{ij}\partial_{ij}v(y)+\tilde{b}_{i}\partial_{i}v(y)=0,
\end{align*}
where $\tilde{a}_{ij}=a_{kl}\partial_{x_{k}}y_{i}\partial_{x_{l}}y_{j}$ and $\tilde{b}_{i}=a_{kl}\partial_{x_{k}x_{l}}y_{i}$. Using \eqref{WA908}, we also have
\begin{align*}
\frac{I_{d}}{\mathcal{C}_{0}}\leq\nabla_{x}y=\nabla_{x}\Phi^{-1}(x)\leq \mathcal{C}_{0}I_{d},
\end{align*}
where $\mathcal{C}_{0}$ is given by \eqref{MA90}. Denote
\begin{align*}
\mathcal{K}(y):=-2r^{m}\tilde{\delta}^{-1}\nabla_{y'}\tilde{h}_{2}-\tilde{\delta}^{-1}(y_{d}+r^{m})\nabla_{y'}\tilde{\delta}.
\end{align*}
If $2^{-1}(\varepsilon\kappa_{2}^{-1})^{1/m}<r\leq r_{0,3}$ with $r_{0,3}$ given by \eqref{Z006}, it then follows from conditions ({\bf{H1}})--({\bf{H2}}) that for $y\in Q_{2r,r^{m}}\setminus Q_{r/4,r^{m}}$,
\begin{align}\label{ZQ006}
|\mathcal{K}(y)|\leq2^{3m+1}\kappa_{1}^{-1}\kappa_{3}r^{m-1}\leq\frac{1}{4}.
\end{align}
In view of \eqref{QD01}, we deduce from \eqref{ZQ039} and \eqref{ZQ006} that if $2^{-1}(\varepsilon\kappa_{2}^{-1})^{1/m}<r\leq\min\limits_{1\leq i\leq3}r_{0,i}$ and $y\in Q_{2r,r^{m}}\setminus Q_{r/4,r^{m}}$,
\begin{align*}
|\nabla_{x}y|\leq|B^{-1}|+\big(|\nabla_{y}g(y)|+|\mathcal{K}(y)|\big)|\nabla_{x}y|\leq|B^{-1}|+\frac{3}{4}|\nabla_{x}y|,
\end{align*}
and thus,
\begin{align}\label{E96}
|\nabla_{x}y|\leq 4|B^{-1}|\leq4\sqrt{d-1+4r^{2m}\tilde{\delta}^{-2}}\leq4\sqrt{d-1+4^{2m+1}\kappa_{1}^{-2}},
\end{align}
where $B^{-1}$ is the inverse matrix of $B$ defined by \eqref{MAT01}. Then we have from \eqref{D026} and \eqref{ZQ039} that
\begin{align*}
|\nabla_{y}x|\leq|\nabla_{y}x-B|+|B|\leq\frac{1}{2}+\sqrt{d-1+4^{-1}r^{-2m}\tilde{\delta}^{2}}\leq\frac{1}{2}+\sqrt{d-1+2^{2m}\kappa_{2}^{2}},
\end{align*}
which gives that
\begin{align*}
|\xi|=|\nabla_{y}x\nabla_{x}y\xi|\leq C(d,m,\kappa_{2})|\nabla_{x}y\xi|,\quad\text{for any }\xi\in\mathbb{R}^{d}.
\end{align*}
This, together with \eqref{U9} and \eqref{E96}, reads that for any $\xi\in\mathbb{R}^{d}$,
\begin{align*}
\frac{1}{C_{1}}|\xi|^{2}\leq\min\{1,p-1\}|\nabla_{x}y\xi|^{2}\leq\xi_{i}\tilde{a}_{ij}\xi_{j}\leq\max\{1,p-1\}|\nabla_{x}y|^{2}|\xi|^{2}\leq C_{1}|\xi|^{2},
\end{align*}
where $C_{1}=C_{1}(d,m,p,\kappa_{1},\kappa_{2})$. That is, $\frac{I_{d}}{C_{1}}\leq\tilde{a}\leq C_{1}I_{d}$.

We now proceed to estimate $|\nabla_{x}^{2}y|$ for the purpose of estimating $|\tilde{b}|$. By differentiating $\partial_{x_{k}}y_{i}\partial_{y_{j}}x_{k}=\delta_{ij}$ in $x_{l}$, we have from the chain rule that
\begin{align*}
\partial_{x_{k}x_{l}}y_{i}\partial_{y_{j}}x_{k}+\partial_{x_{k}}y_{i}\partial_{x_{l}}y_{s}\partial_{y_{j}y_{s}}x_{k}=0.
\end{align*}
This, in combination with \eqref{WA908} and \eqref{E96}, reads that
\begin{align}\label{QE015}
|\nabla_{x}^{2}y|\leq C(d,m,\kappa_{1})|\nabla_{y}^{2}x|.
\end{align}
Using ({\bf{H2}})--({\bf{H3}}), we deduce that $|\nabla_{y}^{2}x_{d}|\leq C(m,\kappa_{3},\kappa_{4})r^{-1}$. It remains to estimate $|\nabla_{y}^{2}x'|$. When $m\geq3$, it follows from a direct computation that $|\nabla_{y}^{2}x'|\leq C(d,m,\kappa_{2},\kappa_{3},\kappa_{4})r^{-1}$. In the case when $2\leq m<3$, the key is to estimate the following terms
\begin{align*}
\nabla_{y'}^{3}\tilde{h}_{i}^{\tau}(y'),\quad\nabla_{y'}^{2}\partial_{y_{d}}\tilde{h}_{i}^{\tau}(y'),\quad\nabla_{y'}\partial_{y_{d}}^{2}\tilde{h}_{i}^{\tau}(y'),\quad i=1,2.
\end{align*}
Observe from integration by parts that for $i=1,2,$
\begin{align*}
\nabla_{y'}\tilde{h}_{i}^{\tau}=&\int_{\mathbb{R}^{d-1}}\nabla_{y'}\tilde{h}_{i}(y'-\tau z')\varphi(z')dz'=-\frac{1}{\tau}\int_{\mathbb{R}^{d-1}}\nabla_{z'}\tilde{h}_{i}(y'-\tau z')\varphi(z')dz'\notag\\
=&\frac{1}{\tau}\int_{\mathbb{R}^{d-1}}\tilde{h}_{i}(y'-\tau z')\nabla_{z'}\varphi(z')dz',
\end{align*}
and
\begin{align}\label{WEQ012}
\partial_{y_{d}}\tilde{h}_{i}^{\tau}=&\frac{2y_{d}}{r^{m-1}}\int_{\mathbb{R}^{d-1}}\nabla_{y'}\tilde{h}_{i}(y'-\tau z')\cdot z'\varphi(z')dz'\notag\\
=&-\frac{2y_{d}}{\tau r^{m-1}}\int_{\mathbb{R}^{d-1}}\nabla_{z'}\tilde{h}_{i}(y'-\tau z')\cdot z'\varphi(z')dz'\notag\\
=&\frac{2y_{d}}{\tau r^{m-1}}\int_{\mathbb{R}^{d-1}}\tilde{h}_{i}(y'-\tau z')\mathrm{div}_{z'}(z'\varphi(z'))dz'.
\end{align}
Then we obtain from condition ({\bf{H3}}) that if $2^{-1}(\varepsilon\kappa_{2}^{-1})^{1/m}<r\leq\min\limits_{1\leq i\leq3}r_{0,i}$ and $y\in Q_{2r,r^{m}}\setminus Q_{r/4,r^{m}}$, for $i=1,2,$
\begin{align*}
|\nabla_{y'}^{3}\tilde{h}_{i}^{\tau}|=\frac{1}{\tau}\left|\int_{\mathbb{R}^{d-1}}\nabla_{y'}^{2}\tilde{h}_{i}(y'-\tau z')\nabla_{z'}\varphi(z')dz'\right|\leq \frac{Cr^{m-1}}{r^{2m}-y_{d}^{2}},
\end{align*}
and
\begin{align*}
|\nabla_{y'}^{2}\partial_{y_{d}}\tilde{h}_{i}^{\tau}|=\left|\frac{2y_{d}}{\tau r^{m-1}}\int_{\mathbb{R}^{d-1}}\nabla_{y'}^{2}\tilde{h}_{i}(y'-\tau z')\mathrm{div}_{z'}(z'\varphi(z'))dz'\right|\leq\frac{Cr^{m}}{r^{2m}-y_{d}^{2}},
\end{align*}
where $C=C(d,m,\kappa_{4})$. By differentiating the last line of \eqref{WEQ012} in $y_{d}$, we obtain
\begin{align*}
\partial_{y_{d}}^{2}\tilde{h}_{i}^{\tau}=&\frac{2}{\tau r^{m-1}}\int_{\mathbb{R}^{d-1}}\tilde{h}_{i}(y'-\tau z')\mathrm{div}_{z'}(z'\varphi(z'))dz'\notag\\
&+\frac{4y_{d}^{2}}{\tau r^{2m-2}}\int_{\mathbb{R}^{d-1}}\nabla_{y'}\tilde{h}_{i}(y'-\tau z')\cdot z'\mathrm{div}_{z'}(z'\varphi(z'))dz'.
\end{align*}
Then we have
\begin{align*}
|\partial_{y_{d}}^{2}\tilde{h}_{i}^{\tau}|\leq&\frac{2}{r^{2m}-y_{d}^{2}}\left|\int_{\mathbb{R}^{d-1}}\nabla_{y'}\tilde{h}_{i}(y'-\tau z')\mathrm{div}_{z'}(z'\varphi(z'))dz'\right|\notag\\
&+\frac{4r^{m+1}}{r^{2m}-y_{d}^{2}}\left|\int_{\mathbb{R}^{d-1}}\nabla_{y'}^{2}\tilde{h}_{i}(y'-\tau z')\cdot z'\mathrm{div}_{z'}(z'\varphi(z'))dz'\right|\notag\\
\leq&\frac{Cr^{m-1}}{r^{2m}-y_{d}^{2}},
\end{align*}
where $C=C(d,m,\kappa_{3},\kappa_{4}).$ In light of these above results, it follows from a direct calculation that
\begin{align*}
|\nabla_{y}^{2}x|\leq C(d,m,\kappa_{2},\kappa_{3},\kappa_{4})r^{-1},
\end{align*}
which, together with \eqref{QE015}, yields that $|\tilde{b}|\leq C r^{-1}$, where $C$ depends only on $d,m$ and $\kappa_{i},$ $i=1,2,3,4.$

It remains to demonstrate that $\frac{\partial v}{\partial\nu}=0$ on $\{y_{d}=\pm r^{m}\}$. For that purpose, it suffices to explain the construction of $g(y)$ in \eqref{QD01}. For $i=1,...,d-1,$ let
\begin{align*}
g_{i}(y)=(y_{d}^{2}-r^{2m})(\beta_{i}(y')y_{d}+\gamma_{i}(y')),
\end{align*}
where $\beta_{i}(y')$ and $\gamma_{i}(y')$ are constructed to satisfy the Neumann boundary condition. From the chain rule, we know
\begin{align*}
\frac{\partial v}{\partial\nu}=\nabla v\cdot e_{d}=\nabla_{x}u(x)\cdot\nabla_{y}\Phi(y)e_{d},\quad\mathrm{on}\;\{y_{d}=\pm r^{m}\},
\end{align*}
where $e_{d}=(0',1).$ So in order to let $\frac{\partial v}{\partial\nu}=0$ on $\{y_{d}=\pm r^{m}\}$, it requires that $\nabla_{y}\Phi(y)e_{d}\pll(-\nabla_{x'}\tilde{h}_{j},1)$ on $\{y_{d}=(-1)^{j+1}r^{m}\}$ for $j=1,2.$ This means that for $i=1,...,d-1$ and $j=1,2,$
\begin{align*}
2r^{2m}\beta_{i}(y')+(-1)^{j+1}2r^{m}\gamma_{i}(y')=\frac{\tilde{\delta}}{2r^{m}}\partial_{y_{i}}\tilde{h}_{j},\quad\mathrm{on}\;\{y_{d}=(-1)^{j+1}r^{m}\},
\end{align*}
where we used the fact that $x'=y'$ on $\{y_{d}=\pm r^{m}\}$. This leads to that
\begin{align*}
\beta_{i}(y')=\frac{\tilde{\delta}}{8r^{3m}}\partial_{y_{i}}(\tilde{h}_{1}(y')+\tilde{h}_{2}(y')),\quad\gamma_{i}(y')=\frac{\tilde{\delta}}{8r^{2m}}\partial_{y_{i}}(\tilde{h}_{1}(y')-\tilde{h}_{2}(y')).
\end{align*}
From the above proof, we see that the third-order derivatives of $\tilde{h}_{1}$ and $\tilde{h}_{2}$ are involved. Hence, we replace $\tilde{h}_{i}$ with its mollification $\tilde{h}_{i}^{\tau}$ in the case of $2\leq m<3$ for $i=1,2$, as shown in \eqref{MO9}. The proof is complete.

\end{proof}

In view of Lemma \ref{ZWQ08} and using the proof of Lemma 3.2 in \cite{DYZ2023} with minor modification, we obtain the Krylov-Safonov Harnack inequality for the solution $u$ to equation \eqref{EQU908} as follows.
\begin{corollary}\label{COR09}
Let $d\geq3$ and $2^{-1}(\varepsilon\kappa_{2}^{-1})^{1/m}<r\leq\min\limits_{1\leq i\leq3}r_{0,i}$ with $r_{0,i},\, i=1,2,3$ defined by \eqref{Z016}--\eqref{Z006}. Assume that $u\in W^{1,p}(\Omega_{2r}\setminus\Omega_{r/4})$ is a nonnegative solution to \eqref{EQU908} with some $\varsigma>0$. Then we have
\begin{align}\label{ZA026}
\sup\limits_{\Omega_{r}\setminus\Omega_{r/2}}u\leq C\inf\limits_{\Omega_{r}\setminus\Omega_{r/2}}u,
\end{align}
where $C=C(d,m,p,\kappa_{1},\kappa_{2},\kappa_{3},\kappa_{4}).$
\end{corollary}
For readers' convenience, we leave the proof of Corollary \ref{COR09} in the appendix.
\begin{remark}
We here would like to emphasize that the Harnack inequality gives a quantitative characterization in terms of small propagation property for the solution, which is critical to the establishment of oscillation decay estimate of the solution in the following.
\end{remark}
\begin{remark}\label{RE096}
By slightly modifying the proofs in Lemma \ref{ZWQ08} and Corollary \ref{COR09}, we obtain that \eqref{ZA026} also holds in $\Omega_{x',r}\setminus\Omega_{x',r/2}$ for any $x\in\Omega_{r}$, $0<r\leq\min\limits_{1\leq i\leq3}r_{0,i}.$
\end{remark}

\begin{lemma}\label{LEM903}
For $d\geq3$, let $u\in W^{1,p}(\Omega_{2R_{0}})$ be the solution to equation \eqref{APPRO} with some $\varsigma>0$. For any fixed $0<R_{0}\leq\min\limits_{1\leq i\leq3}R_{0,i}$ and $0<\beta<1$, if $0<r_{0}\leq\min\limits_{1\leq i\leq4}r_{0,i}$, then for a arbitrarily small $\varepsilon>0$ and $x\in\Omega_{r_{0}}$, there exists a positive constant $0<\gamma<1$ depending only on $d,m,p$ and $\kappa_{i}$, $i=1,2,3,4$ such that
\begin{align*}
\mathop{osc}_{\Omega_{x',\tilde{\varrho}}}u\leq2\tilde{\varrho}^{\beta\gamma}\mathop{osc}_{\Omega_{x',\tilde{\varrho}^{1-\beta}}}u,
\end{align*}
where $\tilde{\varrho}=\tilde{c}_{0}\delta^{1/m}.$

\end{lemma}
\begin{proof}
For the convenience of presentation, we drop the $x'$ subscript and write $\Omega_{r}=\Omega_{x',r}$ in this proof. Since $0<r_{0}\leq r_{0,4}$, we have $2\tilde{\varrho}<\tilde{\varrho}^{1-\beta}\leq R_{0}$. Then for any $\tilde{\varrho}\leq r<\tilde{\varrho}^{1-\beta}$, define $w:=\sup\limits_{\Omega_{2r}}u-u$. Note that $w$ satisfies equation \eqref{APPRO} in $\Omega_{2r}$, it then follows from the maximum principle that
\begin{align*}
\sup\limits_{\Omega_{r}\setminus\Omega_{r/2}}w=\sup\limits_{\Omega_{r}}w,\quad\inf\limits_{\Omega_{r}\setminus\Omega_{r/2}}w=\inf\limits_{\Omega_{r}}w,
\end{align*}
which, together with Corollary \ref{COR09} and Remark \ref{RE096}, shows that
\begin{align*}
\sup\limits_{\Omega_{r}}w\leq C_{0}\inf\limits_{\Omega_{r}}w,\quad\text{with }C_{0}=C_{0}(d,m,p,\kappa_{1},\kappa_{2},\kappa_{3},\kappa_{4})>1.
\end{align*}
This leads to that
\begin{align}\label{K06}
C_{0}\sup\limits_{\Omega_{r}}u-\inf\limits_{\Omega_{r}}u\leq(C_{0}-1)\sup\limits_{\Omega_{2r}}u.
\end{align}
By adding $-(C_{0}-1)\inf\limits_{\Omega_{r}}u\leq-(C_{0}-1)\inf\limits_{\Omega_{2r}}u$ to \eqref{K06}, we obtain
\begin{align*}
\mathop{osc}_{\Omega_{r}}u\leq\frac{C_{0}-1}{C_{0}}\mathop{osc}_{\Omega_{2r}}u.
\end{align*}
Pick $\gamma=\frac{\ln\frac{C_{0}}{C_{0}-1}}{\ln2}$. Then we have
\begin{align}\label{ITE03}
\mathop{osc}_{\Omega_{r}}u\leq2^{-\gamma}\mathop{osc}_{\Omega_{2r}}u.
\end{align}
For $i\geq0$, write $r_{i}=\frac{\tilde{\varrho}^{1-\beta}}{2^{i}}$. Take $k$ such that $\tilde{\varrho}\leq r_{k}<2\tilde{\varrho}$. Then iterating \eqref{ITE03} for $k$ times, we have
\begin{align*}
\mathop{osc}_{\Omega_{\tilde{\varrho}}}u\leq\mathop{osc}_{\Omega_{r_{k}}}u\leq2^{-k\gamma}\mathop{osc}_{\Omega_{r_{0}}}u\leq2^{\gamma}\tilde{\varrho}^{\beta\gamma}\mathop{osc}_{\Omega_{\tilde{\varrho}^{1-\beta}}}u,
\end{align*}
where we used the fact that $2^{-k}\tilde{\varrho}^{1-\beta}=2^{-k}r_{0}=r_{k}<2\tilde{\varrho}.$ The proof is complete.

\end{proof}

Now we are ready to give the proof of Theorem \ref{thm002}.
\begin{proof}[Proof of Theorem \ref{thm002}]
A combination of Theorem \ref{thm001} and Lemma \ref{LEM903} shows that for any $x\in\Omega_{r_{0}}$, $0<r_{0}\leq\min\limits_{1\leq i\leq4}r_{0,i}$,
\begin{align*}
|\nabla u(x)|\leq C\delta^{-1/m}\mathop{osc}_{\Omega_{x',\tilde{\varrho}}}u\leq C\delta^{\frac{\beta\gamma-1}{m}}\mathop{osc}_{\Omega_{x',\tilde{\varrho}^{1-\beta}}}u,\quad\tilde{\varrho}=\tilde{c}_{0}\delta^{1/m},
\end{align*}
where $C=C(d,m,p,\kappa_{1},\kappa_{2},\kappa_{3},\kappa_{4}).$ The proof is finished.

\end{proof}

\section{Gradient estimates with explicit blow-up rates}\label{Sec905}

In order to prove Theorem \ref{thm003}, we see from Theorem \ref{thm001} that it suffices to establish the oscillation decay estimate of the solution $u$ with an explicit rate. For that purpose, we plan to construct its explicit supersolution in the following. For $m>2$, $\tau>0$ and $\gamma>0$, define an auxiliary function as follows:
\begin{align}\label{W365}
w(x)=\big(|x'|^{m}+a|x'|^{m-2}x_{d}^{2}\big)^{\frac{\gamma}{m}},\quad a=\frac{m(m+\tau)}{2},\quad\text{for }x\in\Omega_{\tilde{r}_{0}}.
\end{align}
The auxiliary function $w$ will be proved to be a local supersolution as follows.
\begin{lemma}\label{Lem695}
Assume as in Theorem \ref{thm003}. Let $d\geq2$, $m>2$ and $p>d+m-1$. Then for any $\tau\in(0,p-d-m+1)$ and $\gamma\in\big(0,\frac{p-d-m+1-\tau}{p-1}\big)$, there exists a small constant $0<\hat{r}_{0}=\hat{r}_{0}(d,m,p,\lambda,\gamma,\tau,\tilde{r}_{0})\leq \tilde{r}_{0}$ such that for a sufficiently small $\varepsilon>0,$
\begin{align}\label{QP012}
\begin{cases}
-\mathrm{div}(|\nabla w|^{p-2}\nabla w)>0,&\mathrm{in}\;\Omega_{\hat{r}_{0}}\setminus\Omega_{\varepsilon^{2/m}},\\
\frac{\partial w}{\partial\nu}>0,&\mathrm{on}\;\Gamma_{\hat{r}_{0}}^{\pm}\cap\overline{\Omega_{\hat{r}_{0}}\setminus\{x'=0'\}},
\end{cases}
\end{align}
where $w$ is given by \eqref{W365}.

\end{lemma}

\begin{remark}
If we take $m=2$ in \eqref{W365}, the auxiliary function becomes the separation form of $|x'|$ and $x_{d}$, which is adopted in recent work \cite{DYZ2023}. However, when $m>2$, the required form becomes more complex and leads to a different computation in the following. Moreover, the results of $m=2$ cannot be derived by directly letting $m\rightarrow2$. See the following proofs for more details.

\end{remark}

\begin{proof}[Proof of Lemma \ref{Lem695}]
For simplicity, in the following we use $O(1)$ to represent that $|O(1)|\leq C$ for some constant $C$, which may differ at each occurrence and depend only on $d,m,p,\lambda,\gamma,\tau,\sigma_{0},\tilde{r}_{0},$ but not on $\varepsilon$. Denote $Q=(|x'|^{m}+a|x'|^{m-2}x_{d}^{2})^{1/m}$ and then $w(x)=Q(x)^{\gamma}.$ Note that
\begin{align}\label{DQE019}
|x_{d}|\leq&\frac{\varepsilon}{2}+\lambda|x'|^{m}+C_{0}|x'|^{m+\sigma_{0}}\notag\\
\leq&\frac{1}{2}\big(1+2\lambda|\tilde{r}_{0}|^{\frac{m}{2}}+2C_{0}|\tilde{r}_{0}|^{\frac{m}{2}+\sigma_{0}}\big)|x'|^{\frac{m}{2}},\quad\text{for }\varepsilon^{2/m}\leq|x'|<\tilde{r}_{0}.
\end{align}
In light of \eqref{DQE019}, it follows from a direct calculation that for $\varepsilon^{2/m}\leq|x'|<\tilde{r}_{0}$,
\begin{align*}
\nabla_{x'}w=&\gamma Q^{\gamma-m}|x'|^{m-1}\left(x_{i}|x'|^{-1}+O(1)|x'|^{m-2}\right),\quad\partial_{d}w=\frac{2a\gamma}{m}Q^{\gamma-m}|x'|^{m-2}x_{d},\\
\partial_{ii}w=&\gamma Q^{\gamma-2m}|x'|^{2m-2}\left(1+(\gamma-2)x_{i}^{2}|x'|^{-2}+O(1)|x'|^{m-2}\right),\quad i=1,...,d-1,\\
\partial_{ij}w=&\gamma Q^{\gamma-2m}|x'|^{2m-2}((\gamma-2)x_{i}x_{j}|x'|^{-2}+O(1)|x'|^{m-2}),\;\, i\neq j,\;i,j=1,...,d-1,\\
\partial_{id}w=&\frac{2a\gamma(\gamma-2)}{m}Q^{\gamma-2m}|x'|^{2m-2}\big(x_{i}x_{j}|x'|^{-2}+O(1)|x'|^{\frac{3m-6}{2}}\big),\quad i=1,...,d-1,\\
\partial_{dd}w=&\gamma Q^{\gamma-2m}|x'|^{2m-2}(m+\tau+O(1)|x'|^{m-2}).
\end{align*}
We here would like to remark that when $m>2$, there produces explicit high order terms $|x'|^{m-2}$ and $|x'|^{\frac{3m-6}{2}}$, which ensure that it suffices to decrease the horizontal length of the thin gap for the purpose of letting $w$ become the supersolution in the following computations. This is greatly different from the case of $m=2$, as shown in Lemma 4.1 of \cite{DYZ2023}. In fact, besides decreasing $\tilde{r}_{0}$, it also needs to require that $|x'|\geq C\varepsilon$ for a large constant $C$ in the case when $m=2$. In view of \eqref{U0913}, we have
\begin{align*}
\nu=\frac{(-\nabla_{x'}h(x'),1)}{\sqrt{1+|\nabla_{x'}h(x')|^{2}}},\quad\mathrm{on}\; \Gamma_{\tilde{r}_{0}}^{+}.
\end{align*}
Then on $\Gamma_{\tilde{r}_{0}}^{+}\cap \overline{\Omega_{\tilde{r}_{0}}\setminus\{x'=0'\}},$ we obtain from \eqref{U091302} that
\begin{align*}
\frac{\partial w}{\partial\nu}=&\frac{\partial_{d}w-\nabla_{x'}h\nabla_{x'} w}{\sqrt{1+|\nabla_{x'}h|^{2}}}\notag\\
=&\frac{\lambda\gamma Q^{\gamma-m}|x'|^{2m-2}}{\sqrt{1+|\nabla_{x'}h|^{2}}}(\tau+O(1)|x'|^{\min\{m-2,\sigma_{0}\}})+\frac{\gamma(m+\tau)Q^{\gamma-m}|x'|^{m-2}\varepsilon}{2\sqrt{1+|\nabla_{x'}h|^{2}}}\notag\\
>&\frac{\lambda\gamma Q^{\gamma-m}|x'|^{2m-2}}{\sqrt{1+|\nabla_{x'}h|^{2}}}(\tau-C_{1}|x'|^{\min\{m-2,\sigma_{0}\}}).
\end{align*}
Pick $\hat{r}_{0,1}=\big(\frac{\tau}{C_{1}}\big)^{\frac{1}{\min\{m-2,\sigma_{0}\}}}$. Then we have $\frac{\partial w}{\partial\nu}>0$ on  $\Gamma_{\hat{r}_{0,1}}^{+}\cap \overline{\Omega_{\hat{r}_{0,1}}\setminus\{x'=0'\}}.$ Similarly, we also obtain that $\frac{\partial w}{\partial\nu}>0$ on $\Gamma_{\hat{r}_{0,1}}^{-}\cap \overline{\Omega_{\hat{r}_{0,1}}\setminus\{x'=0'\}}$.

Note that
\begin{align*}
\mathrm{div}(|\nabla w|^{p-2}\nabla w)|\nabla w|^{4-p}=|\nabla w|^{2}\Delta w+(p-2)\sum^{d}_{i,j=1}\partial_{i}w\partial_{j}w\partial_{ij}w.
\end{align*}
From \eqref{DQE019}, we obtain that for $\varepsilon^{2/m}\leq|x'|<\tilde{r}_{0}$,
\begin{align*}
|\nabla w|^{2}=&\gamma^{2}Q^{2\gamma-2m}|x'|^{2m-2}(1+O(1)|x'|^{m-2}),
\end{align*}
and
\begin{align*}
\Delta w=&\gamma Q^{\gamma-2m}|x'|^{2m-2}\left(d+m+\gamma-3+\tau+O(1)|x'|^{m-2}\right).
\end{align*}
Therefore, we deduce that for $\varepsilon^{2/m}\leq|x'|<\tilde{r}_{0}$,
\begin{align}\label{WM963}
\frac{|\nabla w|^{2}\Delta w}{\gamma^{3}Q^{3\gamma-4m}|x'|^{4m-4}}=d+m+\gamma-3+\tau+O(1)|x'|^{m-2}.
\end{align}
By a similar computation, we have from \eqref{DQE019} that
\begin{align*}
\sum^{d-1}_{i,j=1}\partial_{i}w\partial_{j}w\partial_{ij}w=&\gamma^{3}Q^{3\gamma-4m}|x'|^{4m-4}(\gamma-1+O(1)|x'|^{m-2}),\notag\\
2\sum^{d-1}_{i=1}\partial_{i}w\partial_{d}w\partial_{id}w=&O(1)\gamma^{3}Q^{3\gamma-4m}|x'|^{\frac{9m}{2}-5},\notag\\
(\partial_{d}w)^{2}\partial_{dd}w=&O(1)\gamma^{3}Q^{3\gamma-4m}|x'|^{5m-6}.
\end{align*}
Combining these three equations, we derive that for $\varepsilon^{2/m}\leq|x'|<\tilde{r}_{0}$,
\begin{align*}
\frac{(p-2)\sum^{d}_{i,j=1}\partial_{i}w\partial_{j}w\partial_{ij}w}{\gamma^{3}Q^{3\gamma-4m}|x'|^{4m-4}}=(p-2)(\gamma-1)+O(1)|x'|^{\frac{m-2}{2}},
\end{align*}
which, together with \eqref{WM963}, reads that
\begin{align}\label{ME009}
&\frac{\mathrm{div}(|\nabla w|^{p-2}\nabla w)|\nabla w|^{4-p}}{\gamma^{3}Q^{3\gamma-4m}|x'|^{4m-4}}\leq d+m-2+(p-1)(\gamma-1)+\tau+C_{2}|x'|^{\frac{m-2}{2}}.
\end{align}
Define
\begin{align*}
\hat{r}_{0,2}=\left(\frac{(p-1)(1-\gamma)+2-d-m-\tau}{C_{2}}\right)^{\frac{2}{m-2}}.
\end{align*}
Therefore, we have
\begin{align*}
\mathrm{div}(|\nabla w|^{p-2}\nabla w)<0,\quad\text{for }\varepsilon^{2/m}\leq|x'|<\hat{r}_{0,2}.
\end{align*}
By taking $\hat{r}_{0}=\min\limits_{1\leq i\leq2}\hat{r}_{0,i}$, we obtain that \eqref{QP012} holds.
\end{proof}

Once Lemma \ref{Lem695} is proved, we can give the proof Theorem \ref{thm003}.
\begin{proof}[Proof of Theorem \ref{thm003}]
Assume without loss of generality that $u(0)=1$ and $\mathop{osc}\limits_{\Omega_{\hat{r}_{0}}}u=1$. Applying Theorem \ref{thm001}, we have
\begin{align}\label{EP005}
|u(x)|\leq C\varepsilon^{1/m},\quad\text{for }x\in\overline{\Omega}_{\varepsilon^{2/m}},
\end{align}
which implies that $\pm u\leq C w$ on $(\{|x'|=\varepsilon^{2/m}\}\cup\{|x'|=\hat{r}_{0}\})\cap\overline{\Omega}_{\hat{r}_{0}}$, where $w$ is defined by \eqref{W365} with $\gamma=\frac{p-d-m+1-2\tau}{p-1}$ for any $\tau\in\big(0,\frac{1}{2}(p+1-d-m)\big)$ and $\hat{r}_{0}$ is given in Lemma \ref{Lem695}. It then follows from the comparison principle that $|u|\leq Cw\leq C|x'|^{\gamma}$ in $\Omega_{\hat{r}_{0}}\setminus\Omega_{\varepsilon^{2/m}}$. This, in combination with \eqref{EP005}, leads to that
\begin{align*}
|u|\leq C(|x'|^{\gamma}+\varepsilon^{1/m})\leq C(\varepsilon+|x'|^{m})^{\frac{p-d-m+1-2\tau}{m(p-1)}},\quad\text{in }\Omega_{\hat{r}_{0}},
\end{align*}
which reads that for $x\in\Omega_{\hat{r}_{0}/2}$,
\begin{align}\label{K016}
\mathop{osc}\limits_{\Omega_{x',\varrho}}u\leq 2\|u\|_{L^{\infty}(\Omega_{x',\varrho})}\leq C(\varepsilon+|x'|^{m})^{\frac{p-d-m+1-2\tau}{m(p-1)}},
\end{align}
where $C=C(d,m,p,\tau,\lambda,\sigma_{0},\tilde{r}_{0})$, $\varrho=\frac{\tilde{c}_{0}}{3}\delta^{1/m}\leq\frac{\hat{r}_{0}}{3}<\frac{\hat{r}_{0}}{2}$ with $\delta$ and $\tilde{c}_{0}$ defined by \eqref{delta} and \eqref{C681}, respectively. A consequence of \eqref{K016} and Theorem \ref{thm001} shows that Theorem \ref{thm003} holds.

\end{proof}

We next prove the almost sharpness of gradient blow-up rates obtained in Theorem \ref{thm001} for $1<p\leq m+1$ and Theorem \ref{thm003} for $p>m+1$ in dimension two. Similarly as above, we first construct an auxiliary function as follows: for $m>2$ and $\tau,\gamma>0$,
\begin{align}\label{AU05}
\underline{w}(x):=[(|x_{1}|^{m}+b|x_{1}|^{m-2}x_{2}^{2})^{\frac{\gamma}{m}}-(2m\varepsilon/(m-2-\tau))^{\frac{\gamma}{m}}]_{+}, \quad\mathrm{in}\;\Omega_{1},
\end{align}
where $b=\frac{m(m-\tau)}{2}.$ We now prove that the auxiliary function $\underline{w}$ is actually a subsolution for problem \eqref{problem006}. To be specific, we have
\begin{lemma}\label{Lem958}
Assume as in Theorem \ref{thm005}. Let $d=2$, $m>2$ and $p>1$. Then for any $\tau\in(0,m-2)$ and $\gamma>\max\big\{0,\frac{p-m-1+\tau}{p-1}\big\}$, there exists a small constant $0<\hat{r}_{0}=\hat{r}_{0}(m,p,\gamma,\tau)<1$ such that if $\varepsilon>0$ is sufficiently small,
\begin{align*}
\begin{cases}
-\mathrm{div}(|\nabla \underline{w}|^{p-2}\nabla \underline{w})\leq0,&\mathrm{in}\;\Omega_{\hat{r}_{0}},\\
\frac{\partial \underline{w}}{\partial\nu}\leq0,&\mathrm{on}\;\Gamma_{\hat{r}_{0}}^{\pm},
\end{cases}
\end{align*}
where $\underline{w}$ is defined by \eqref{AU05}.
\end{lemma}

\begin{proof}
For simplicity, denote $\Theta=(|x_{1}|^{m}+b|x_{1}|^{m-2}x_{2}^{2})^{\frac{1}{m}}$, $v=\Theta^{\gamma}$ and $h(x_{1})=1-(1-|x_{1}|^{m})^{\frac{1}{m}}$. A straightforward computation shows that
\begin{align*}
\partial_{1}v=&\gamma \Theta^{\gamma-m}x_{1}|x_{1}|^{m-4}\left(x_{1}^{2}+\frac{b(m-2)}{m}x_{2}^{2}\right),\notag\\
\partial_{2}v=&\frac{2a\gamma}{m}\Theta^{\gamma-m}|x_{1}|^{m-2}x_{2},\quad\partial_{1}h=x_{1}|x_{1}|^{m-2}(1-|x_{1}|^{m})^{1/m-1}.
\end{align*}
Observe that the upward normal vector on $\Gamma_{1}^{+}$ is $\nu=(-x_{1}|x_{1}|^{m-2},(1-|x_{1}|^{m})^{1-1/m})$. Then on $\Gamma_{1}^{+}\cap \{|x_{1}|\geq(m\varepsilon/(m-2-\tau))^{1/m}\},$ we have
\begin{align}\label{EK019}
&\frac{\frac{\partial v}{\partial\nu}}{\gamma\Theta^{\gamma-m}|x_{1}|^{m-2}}\notag\\
&=(m-\tau)x_{2}(\varepsilon/2+1-x_{2})^{m-1}-|x_{1}|^{m}-\frac{b(m-2)}{m}|x_{1}|^{m-2}x_{2}^{2}\notag\\
&<(m-\tau)x_{2}(\varepsilon/2+1-x_{2})^{m-1}+(\varepsilon/2+1-x_{2})^{m}-1\notag\\
&=(\varepsilon/2+1-x_{2})^{m-2}\Big(-(m-1-\tau)x_{2}^{2}+(m-2-\tau)(1+\varepsilon/2)x_{2}+\varepsilon+\frac{\varepsilon^{2}}{4}\Big)\notag\\
&\quad+(\varepsilon/2+1-x_{2})^{m-2}-1.
\end{align}
It is worth emphasizing that since the implication of $\tau$ lies in keeping its arbitrary smallness to show the almost sharpness of the blow-up rate, then we require $\tau\in(0,m-2)$ if $m>2$. By contrast, there is no such a restrictive condition in the case of $m=2$. Note that $x_{2}=\varepsilon/2+1-(1-|x_{1}|^{m})^{1/m}$ on $\Gamma_{1}^{+}$. Then for $|x_{1}|\geq(m\varepsilon/(m-2-\tau))^{1/m}$, we have $x_{2}\geq\max\{\frac{\varepsilon}{m-2-\tau},\frac{\varepsilon}{2}\}$. This leads to that
\begin{align*}
\begin{cases}
(\varepsilon/2+1-x_{2})^{m-2}-1\leq0,\\
-(m-1-\tau)x_{2}^{2}+(m-2-\tau)(1+\varepsilon/2)x_{2}+\varepsilon+\frac{\varepsilon^{2}}{4}\leq0.
\end{cases}
\end{align*}
Inserting this into \eqref{EK019}, we deduce
\begin{align}\label{UE012}
\frac{\partial v}{\partial\nu}\leq0,\quad\mathrm{on}\; \Gamma_{1}^{+}\cap \{|x_{1}|\geq(m\varepsilon/(m-2-\tau))^{1/m}\}.
\end{align}
When $|x_{1}|\leq(m\varepsilon/(m-2-\tau))^{1/m},$ we have $v\leq(2m\varepsilon/(m-2-\tau))^{\gamma/m}$ on $\overline{\Omega}_{(m\varepsilon/(m-2-\tau))^{1/m}}$ for a sufficiently small $\varepsilon$. That is,
\begin{align}\label{QM0135}
\underline{w}=0,\quad \text{on }\overline{\Omega}_{(m\varepsilon/(m-2-\tau))^{1/m}}.
\end{align}
Then we have $\frac{\partial \underline{w}}{\partial\nu}=0$ on $\Gamma_{(m\varepsilon/(m-2-\tau))^{1/m}}^{+}$, which, together with \eqref{UE012}, gives that
$\frac{\partial \underline{w}}{\partial\nu}\leq0$ on $\Gamma_{1}^{+}$. By the same argument, we also have $\frac{\partial \underline{w}}{\partial\nu}\leq0$ on $\Gamma_{1}^{-}$.

On the other hand, in exactly the same way to \eqref{ME009}, we obtain that for $x\in\Omega_{1}\cap\{|x_{1}|\geq(m\varepsilon/(m-2-\tau))^{1/m}\}$,
\begin{align*}
&\frac{\mathrm{div}(|\nabla \underline{w}|^{p-2}\nabla \underline{w})|\nabla \underline{w}|^{4-p}}{\gamma^{3}Q^{3\gamma-4m}|x_{1}|^{4m-4}}\geq m+(p-1)(\gamma-1)-\tau-\widehat{C}_{0}|x_{1}|^{m-2},
\end{align*}
where $\widehat{C}_{0}=\widehat{C}_{0}(m,d,\gamma,\tau)$. By picking
\begin{align*}
\hat{r}_{0}=\left(\frac{m+(p-1)(\gamma-1)-\tau}{\widehat{C}_{0}}\right)^{\frac{1}{m-2}},
\end{align*}
we deduce that $\mathrm{div}(|\nabla \underline{w}|^{p-2}\nabla \underline{w})\geq0$ in $\Omega_{\hat{r}_{0}}\cap\{|x_{1}|\geq(m\varepsilon/(m-2-\tau))^{1/m}\}$. This, in combination with \eqref{QM0135}, yields that $\mathrm{div}(|\nabla \underline{w}|^{p-2}\nabla \underline{w})\geq0$ in $\Omega_{\hat{r}_{0}}$.

\end{proof}

We are now ready to prove Theorem \ref{thm005}.
\begin{proof}[Proof of Theorem \ref{thm005}]
For any $m>2$ and $\tau\in(0,m-2)$, let $\underline{w}$ be given by \eqref{AU05} with
\begin{align*}
\gamma=
\begin{cases}
\tau,&\text{if }1<p\leq m+1,\\
\frac{p-m-1+2\tau}{p-1},&\text{if }p>m+1.
\end{cases}
\end{align*}
Observe from symmetry of the domain and the maximum principle that
\begin{align*}
u(0,x_{2})=0,\;\,\mathrm{if}\;|x_{2}|<\varepsilon/2,\quad\text{and }\,u(x)>0,\;\,\text{if }x_{1}>0.
\end{align*}
This, in combination with the comparison principle, yields that $u\geq\frac{1}{C}\underline{w}$ in $\Omega_{\hat{r}_{0}}\cap\{x_{1}>0\}$. Then for a sufficiently small $\varepsilon>0$,
\begin{align*}
u((4m\varepsilon/(m-2-\tau))^{1/m},0)\geq\frac{1}{C}\varepsilon^{\frac{\gamma}{m}},
\end{align*}
where $C=C(m,p,\tau)$. Due to the fact that $u(0)=0$, it then follows from the mean value theorem that
\begin{align*}
\|\nabla u\|_{L^{\infty}(\Omega_{(4m\varepsilon/(m-2-\tau))^{1/m}})}\geq\frac{1}{C}\varepsilon^{\frac{\gamma-1}{m}}.
\end{align*}
The proof is finished.

\end{proof}

\section{Further discussions and remarks}\label{SEC050}

From Theorems \ref{thm001}, \ref{thm003} and \ref{thm005}, it follows that for any $p>1$ and $m>2$, the blow-up rate $\varepsilon^{-1/\max\{p-1,m\}}$ is almost optimal when $d=2$. With regard to the two-dimensional results corresponding to the case of $m=2$, see \cite{DYZ2023}. By contrast, the optimality of the blow-up rate for the nonlinear insulated conductivity problem with $d\geq3$ still remains to be open, although an explicit upper bound in Theorem \ref{thm003} has been given in the case of $p>d+m-1$. Before further breakthrough in rigorously proving the sharpness of the stress blow-up rate for $d\geq3$, we here would like to point out a potentially efficient way to find the exact blow-up rate in dimensions greater than two by studying the local behavior for the solutions to a class of $(d-1)$-dimensional weighted $p$-Laplace equations. This idea originates from \cite{DLY2021}, while the subsequent work \cite{DLY2022} gives its further extension and refinement such that it can be used to deal with general divergence form elliptic equations with H\"{o}lder continuous coefficients in the presence of general strictly convex inclusions. For readers' convenience, we take the Laplace equation for example and give a brief description for the idea in the following.

As shown in \cite{DLY2021,DLY2022}, the key to the reduction from the origin problem to the $(d-1)$-dimensional weighted elliptic problem lies in using the linearity of the Laplace equation, which is achieved by first straightening out the upper and lower boundaries of the thin channel and then integrating the transformed equation in $x_{d}$ to obtain a $(d-1)$-dimensional weighted elliptic equation in a ball. Then spherical harmonic expansion is used to precisely capture the decay rate of solution to the $(d-1)$-dimensional weighted elliptic equation near the origin. Finally, the small difference between these two solutions of different dimensions is estimated to ensure an explicit upper bound on the gradient with its blow-up rate determined by the decay rate exponent captured above.

To be specific, the reduced $(d-1)$-dimensional weighted elliptic problem can be formulated as follows:
\begin{align}\label{WED016}
\begin{cases}
\mathrm{div}(a(x')|x'|^{2}\nabla v(x'))=0,&\mathrm{in}\; B'_{1},\\
\|v\|_{L^{\infty}(B'_{1})}\leq M,
\end{cases}
\end{align}
where $M$ is a positive constant independent of $\varepsilon$, $a(x')>0$ satisfies that $\kappa^{-1}\leq a(x')\leq\kappa$ in $B'_{1}$ for some constant $\kappa>0$, and $\int_{\mathbb{S}^{d-2}}a x_{i}=0$, $i=1,2,...,d-1$, $\mathbb{S}^{d-2}=\{x'\in\mathbb{R}^{d-1}\,|\,|x'|=1\}$. For problem \eqref{WED016}, Dong, Li and Yang \cite{DLY2022} used spherical harmonic expansion to capture the optimal asymptotic behavior of the solution as follows:
\begin{align}\label{EKQ9096}
v(x')=v(0')+O(1)|x'|^{\alpha},\quad\alpha=\frac{-(d-1)+\sqrt{(d-1)^{2}+4\lambda_{1}}}{2},\quad\mathrm{in}\;B'_{1/2},
\end{align}
where $|O(1)|\leq C=C(d,\kappa,M),$ $\lambda_{1}\leq d-2$ is the first nonzero eigenvalue of the following divergence form elliptic operator on $\mathbb{S}^{d-2}$:
\begin{align*}
-\mathrm{div}_{\mathbb{S}^{d-2}}(a(\xi)\nabla_{\mathbb{S}^{d-2}}u(\xi))=\lambda a(\xi) u(\xi),\quad\xi\in\mathbb{S}^{d-2}.
\end{align*}
In particular, $\lambda_{1}=d-2$ if $a$ is a constant, see also \cite{DLY2021}. By finding the exact decay rate exponent $\alpha$ in \eqref{EKQ9096}, they succeeded in sharpening the upper bound in Theorem \ref{thm002} with an explicit blow-up rate $\varepsilon^{-(1-\alpha)/2}$ in the case of $m=2$ and $d\geq3$. Their investigations in \cite{DLY2021,DLY2022} indicate that the optimal stress blow-up rate in high dimensions can be identified by solving the decay rate of solution to the $(d-1)$-dimensional weighted elliptic problem \eqref{WED016} around the origin. This illuminating insight may also break new ground in the exploration of optimal gradient blow-up rate for the corresponding elasticity problem modeled by the Lam\'{e} system in future work.

Although it cannot be reduced as above for nonlinear $p$-Laplace equation with $p\neq2$, it is still an interesting and enlightening way to predict the stress blow-up rate for any $d\geq3$ and $p\neq2$ by studying the local behavior of the solution to the weighted $p$-Laplace equation. Besides its theoretical interest, the weighted elliptic problem has a stronger advantage than the original problem from the view of numerical computations. Indeed, it suffices to use regular meshes to carry out numerical simulations for the solution to weighted $p$-Laplace equation in a ball, which is more effective and feasible in practice than that of the original problem in a narrow region. In addition, with regard to the study on the local regularity for the solution to problem \eqref{WED016}, it can be trace back to the well-known work \cite{FKS1982} completed by Fabes, Kenig and Serapioni in the last century. For the investigation related to the weighted $p$-Laplace equation, see e.g. \cite{MZ2023}.

\section{Appendix:\,The proofs of Lemma \ref{Lem06} and Corollary \ref{COR09} }

\begin{proof}[Proof of Lemma \ref{Lem06}]

We first prove \eqref{DP01} under the condition of $1<p<2$. Note that for any $\xi_{1},\xi_{2}\in\mathbb{R}^{d}$,
\begin{align*}
&|\xi_{1}|^{p-2}\xi_{1}-|\xi_{2}|^{p-2}\xi_{2}\notag\\
&=\int_{0}^{1}\frac{d}{dt}\left[|\xi_{2}+t(\xi_{1}-\xi_{2})|^{p-2}(\xi_{2}+t(\xi_{1}-\xi_{2}))\right]dt\notag\\
&=(\xi_{1}-\xi_{2})\int_{0}^{1}|\xi_{2}+t(\xi_{1}-\xi_{2})|^{p-2}dt\notag\\
&\quad+(p-2)\int_{0}^{1}|\xi_{2}+t(\xi_{1}-\xi_{2})|^{p-4}\langle\xi_{2}+t(\xi_{1}-\xi_{2}),\xi_{1}-\xi_{2}\rangle(\xi_{2}+t(\xi_{1}-\xi_{2}))dt.
\end{align*}
Since $1<p<2$, we then have
\begin{align*}
&\langle|\xi_{1}|^{p-2}\xi_{1}-|\xi_{2}|^{p-2}\xi_{2},\xi_{1}-\xi_{2}\rangle\notag\\
&=|\xi_{1}-\xi_{2}|^{2}\int_{0}^{1}|\xi_{2}+t(\xi_{1}-\xi_{2})|^{p-2}dt\notag\\
&\quad+(p-2)\int_{0}^{1}|\xi_{2}+t(\xi_{1}-\xi_{2})|^{p-4}(\langle\xi_{2}+t(\xi_{1}-\xi_{2}),\xi_{1}-\xi_{2})^{2}dt\notag\\
&\geq(p-1)|\xi_{1}-\xi_{2}|^{2}\int_{0}^{1}|\xi_{2}+t(\xi_{1}-\xi_{2})|^{p-2}dt\notag\\
&\geq(p-1)|\xi_{1}-\xi_{2}|^{2}(|\xi_{1}|+|\xi_{2}|)^{p-2}\geq\frac{p-1}{2^{\frac{2-p}{2}}}|\xi_{1}-\xi_{2}|^{2}(|\xi_{1}|^{2}+|\xi_{2}|^{2})^{\frac{p-2}{2}}.
\end{align*}
That is, \eqref{DP01} holds if $1<p<2$.

Observe that
\begin{align}\label{AMP09}
&\langle|\xi_{1}|^{p-2}\xi_{1}-|\xi_{2}|^{p-2}\xi_{2},\xi_{1}-\xi_{2}\rangle\notag\\
&=\frac{1}{2}(|\xi_{1}|^{p-2}+|\xi_{2}|^{p-2})|\xi_{1}-\xi_{2}|^{2}+\frac{1}{2}(|\xi_{1}|^{p-2}-|\xi_{2}|^{p-2})(|\xi_{1}|^{2}-|\xi_{2}|^{2}).
\end{align}
If $p\geq2$, it then follows from \eqref{AMP09} that
\begin{align*}
\langle|\xi_{1}|^{p-2}\xi_{1}-|\xi_{2}|^{p-2}\xi_{2},\xi_{1}-\xi_{2}\rangle\geq\frac{1}{2}(|\xi_{1}|^{p-2}+|\xi_{2}|^{p-2})|\xi_{1}-\xi_{2}|^{2}.
\end{align*}
Then \eqref{DP02} holds if $p\geq2$.

\end{proof}

\begin{proof}[Proof of Corollary \ref{COR09}]
Recall that $v$ is the solution to equation \eqref{V90} in the annular cylinder $Q_{1.9r,r^{m}}\setminus Q_{0.35r,r^{m}}$. We now use the Neumann boundary to extend the solution along $y_{d}$-axis to a larger domain $Q_{1.9r,2r}\setminus Q_{0.35r,2r}$. For $i,j=1,...,d-1,$ we carry out the odd extension of $\tilde{a}_{id},\tilde{a}_{di},\tilde{b}_{d}$ and the even extension of $\tilde{a}_{ij},\tilde{a}_{dd},\tilde{b}_{i}$ with respect to $y_{d}=r^{m}$, respectively. Subsequently, we perform the periodic extension with the period $4r^{m}$. Also denote by $v,\tilde{a},\tilde{b}$ the extended solution and coefficients. Therefore, $v$ verifies
\begin{align*}
\tilde{a}_{ij}\partial_{ij}v(y)+\tilde{b}_{i}\partial_{i}v(y)=0,\quad\mathrm{in}\;Q_{1.9r,2r}\setminus Q_{0.35r,2r}.
\end{align*}
Let $\hat{a}_{ij}(y)=\tilde{a}_{ij}(ry)$, $\hat{b}_{i}(y)=r\tilde{b}_{i}(ry)$, and $\hat{v}(y)=v(ry).$ Then $\hat{v}$ solves
\begin{align*}
\hat{a}_{ij}\partial_{ij}\hat{v}(y)+\hat{b}_{i}\partial_{i}\hat{v}(y)=0,\quad\mathrm{in}\;Q_{1.9,2}\setminus Q_{0.35,2}.
\end{align*}
Here
\begin{align*}
\frac{I_{d}}{C_{1}}\leq\hat{a}\leq C_{1}I_{d},\quad|\hat{b}|\leq C_{2},
\end{align*}
where $C_{1}=C_{1}(d,m,p,\kappa_{1},\kappa_{2})$, and $C_{2}=C_{2}(d,m,p,\kappa_{1},\kappa_{2},\kappa_{3},\kappa_{4}).$ Observe that $Q_{1.9,2}\setminus Q_{0.35,2}$ is connected in the case when $d\geq3$. It then follows from the Krylov-Safonov theorem (see Section 4.2 in \cite{K1987}) that
\begin{align*}
\sup\limits_{Q_{1.1,1}\setminus Q_{0.4,1}}\hat{v}\leq C\sup\limits_{Q_{1.1,1}}\setminus Q_{0.4,1},
\end{align*}
which yields that
\begin{align*}
\sup\limits_{Q_{1.1r,r^{m}}\setminus Q_{0.4r,r^{m}}}v\leq C\sup\limits_{Q_{1.1r,r^{m}}}\setminus Q_{0.4r,r^{m}},
\end{align*}
where $C=C(d,m,p,\kappa_{1},\kappa_{2},\kappa_{3},\kappa_{4}).$ This, together with \eqref{WA608}, gives that Corollary \ref{COR09} holds.

\end{proof}

\noindent{\bf{\large Acknowledgements.}} The authors would like to thank Prof. C.X. Miao for his constant encouragement and useful discussions. This work was supported in part by the National Key research and development program of
China (No. 2020YFA0712903). Q. Chen was partially supported by the National Natural Science Foundation of China (No. 12071043). Z. Zhao was partially supported by China Postdoctoral Science Foundation (No. 2021M700358).

\bibliographystyle{plain}

\def\cprime{$'$}

\end{document}